\documentclass[10pt,reqno]{amsart}
\usepackage{amssymb,mathrsfs,graphicx}
\usepackage{ifthen}
\usepackage{hyperref}
\usepackage[margin=1in]{geometry}
\usepackage{caption}
\usepackage{sidecap}
\usepackage{rotating,dsfont}
\DeclareMathOperator*{\esssup}{ess\,sup}
\usepackage{colortbl}
\definecolor{black}{rgb}{0.0, 0.0, 0.0}
\definecolor{red}{rgb}{1.0, 0.5, 0.5}
\provideboolean{shownotes} 
\setboolean{shownotes}{true} 
\newcommand{\margnote}[1]{
\ifthenelse{\boolean{shownotes}}%
{\marginpar{\raggedright\tiny\texttt{#1}}}%
{}%
}
\newcommand{\hole}[1]{
\ifthenelse{\boolean{shownotes}}%
{\begin{center} \fbox{ \rule {.25cm}{0cm} \rule[-.1cm]{0cm}{.4cm}
\parbox{.85\textwidth}{\begin{center} \texttt{#1}\end{center}} \rule
{.25cm}{0cm}}\end{center}} {} }

\graphicspath{{pics/}}
\numberwithin{equation}{section}

\title[Local well-posedness for the compressible NS--BGK model]{Local well-posedness for the compressible Navier--Stokes--BGK model in Sobolev spaces with exponential weight}

\author[Choi]{Young-Pil Choi}
\address[Young-Pil Choi]{\newline Department of Mathematics \newline
Yonsei University, 50 Yonsei-Ro, Seodaemun-Gu, Seoul 03722, Republic of Korea}
\email{ypchoi@yonsei.ac.kr}

\author[Jung]{Jinwook Jung}
\address[Jinwook Jung]{\newline Department of Mathematics and Institute of Pure and Applied Mathematics \newline Jeonbuk National University, 567 Baekje-daero, Deokjin-gu, Jeonju-si, Jeollabuk-do 54896,  Republic of Korea}
\email{2jwook12@gmail.com}

\numberwithin{equation}{section}

\newtheorem{theorem}{Theorem}[section]
\newtheorem{lemma}{Lemma}[section]

\newtheorem{proposition}{Proposition}[section]
\newtheorem{remark}{Remark}[section]
\newtheorem{definition}{Definition}[section]

\newcommand{\R}{\mathbb R}

\newcommand{\N}{\mathbb N}
\newcommand{\ls}{\lesssim}

\newcommand{\T}{\mathbb T}

\newcommand{\mc}{\mathcal C}

\newcommand{\bq}{\begin{equation}}
\newcommand{\eq}{\end{equation}}
\newcommand{\e}{\varepsilon}
\newcommand{\lt}{\left}
\newcommand{\rt}{\right}

\newcommand{\pa}{\partial}

\newcommand{\me}{\mathcal{E}}

\newcommand{\md}{\mathcal{D}}

\newcommand{\ml}{\mathcal{L}}
\newcommand{\mm}{\mathcal{M}}

\newcommand{\intr}{\int_{\R^3}}
\newcommand{\la}{\langle}
\newcommand{\ra}{\rangle}
\newcommand{\tZ}{\tilde{Z}}
\newcommand{\tX}{\tilde{X}}
\newcommand{\tV}{\tilde{V}}
\newcommand{\inttr}{\iint_{\T^3\times\R^3}}
\newcommand{\intt}{\int_{\T^3}}

\newcommand{\sfI}{\mathsf{I}}
\newcommand{\sfJ}{\mathsf{J}}
\newcommand{\sfK}{\mathsf{K}}

\begin{document}
\allowdisplaybreaks

\date{\today}

\subjclass[]{}

\begin{abstract} Sprays are complex flows constituted of dispersed particles in an underlying gas. In this paper, we are interested in the equations for moderately thick sprays consisting of the compressible Navier--Stokes equations and Boltzmann BGK equation. Here the coupling of two equations is through a friction (or drag) force which depends on the density of compressible fluid and the relative velocity between particles and fluid. For the Navier--Stokes--BGK system, we establish the existence and uniqueness of solutions in Sobolev spaces with exponential weight.
\end{abstract}

\maketitle \centerline{\date}

\tableofcontents

%
%
%
%

\section{Introduction}\label{sec:1}
\setcounter{equation}{0}

In the context of sprays consisting of dispersed particles such as droplets, dust, etc., in an underlying gas, a coupling of particles and gas was proposed in \cite{R81, W58}. In this modeling, the sprays can be classified depending on the volume fraction of the gas \cite{D10, R81}, and it has a wide rage of applications in the study of biotechnology, bioaerosols in medicine, chemical engineering, combustion theory, etc. We refer to \cite{BDM14, CJ22+, D10, GJV04, M10, R81} and references therein for more physical background, modeling, and applications of the particle-fluid systems.

Among the various levels of possible descriptions depending on the physical regimes under consideration, in the present work, we concentrate on the so-called {\it moderately thick sprays}. In that modeling, the volume fraction of the gas is negligible but the inter-particle interactions, for instances, collision, breakup, coalescences, etc, are taken into account. To be more specific, we are concerned with a coupled particle-fluid system where the Boltzmann BGK equation is coupled with the compressible Navier--Stokes equations through a friction (or drag force):
\begin{align}\label{main_sys}
\begin{aligned}
&\pa_t f + v \cdot \nabla_x f + \nabla_v\cdot(\rho(u-v)f) =Q,\\
&\pa_t \rho + \nabla_x \cdot (\rho u)=0,\\
&\pa_t (\rho u) + \nabla_x \cdot (\rho u\otimes u) + \nabla_x p -\mu\Delta_x u = -\rho\intr(u-v)f\,dv,
\end{aligned}
\end{align}
subject to initial data:
\bq\label{init}
(f(x, v,0), \rho(x,0), u(x,0)) =: (f_0(x,v), \rho_0(x), u_0(x)), \quad (x,v)\in\T^3\times\R^3,
\eq
where $f = f(x,v,t)$ denotes the number density of particles of which at time $t \in \R_+$ and physical position $x \in \T^3$ with velocity $v \in \R^3$, and $\rho = \rho(x,t)$ and $u = u(x,t)$ are the density and the velocity of compressible fluid on $x\in \T^3$ at time $t \in \R_+$, respectively. For simplicity, we consider the viscous term $-\mu\Delta$ with $\mu > 0$. The pressure law is given by $p(\rho) = \rho^\gamma$ with $\gamma>1$ and the inter-particle interaction operator $Q=Q(f)$ is given by the BGK operator $Q(f) = \nu(\rho_f)( \mm(f)-f)$ with the collision frequency $\nu=\nu(\rho_f)$ . Here the local Maxwellian $\mm = \mm(f)$ is given by
\[
\mm(f)(x,v,t) = \frac{\rho_f(x,t)}{(2\pi T_f(x,t))^{3/2}}e^{-\frac{|u_f(x,t)-v|^2}{2T_f(x,t)}},
\]
where
\begin{align}\label{maco}
\begin{aligned}
\rho_f(x,t) &:= \intr f(x,v,t)\,dv,\\
\rho_f(x,t)u_f(x,t)&:= \intr v f(x,v,t)\,dv,\\
3\rho_f(x,t)T_f(x,t)&:= \intr |v-u_f(x,t)|^2 f(x,v,t)\,dv,
\end{aligned}
\end{align}
and we take $\nu(\rho_f) = \rho_f^\alpha$ with $\alpha\in[0,1]$, thus we would also consider the constant collision frequency. We notice that if we ignore the friction force $\rho(u-v)$ in the kinetic equation in \eqref{main_sys}, then the resulting kinetic equation is the Boltzmann BGK model \cite{BGK54}, which is one of the most widely used  models for the Boltzmann equation in physics and engineering.

Over the past two decades, existence theories for coupled kinetic-fluid systems have been widely developed. For the Vlasov equation coupled with incompressible or inhomoegenous Navier--Stokes system, the global existence of weak solutions is studied in \cite{BDGM09, WY08, Y13}. The local existence of strong solutions to the inhomogeneous Navier--Stokes--Vlasov system and its large-time behavior are discussed in \cite{CK15}. For the compressible Euler--Vlasov system, the local classical solutions are obtained in \cite{BD06}.

For the incompressible Navier--Stokes--Vlasov--Fokker--Planck system, the global weak and classical solutions are constructed in \cite{CKL11}.  The global existence of classical solutions near the global Maxwellian for the incompressible Euler--Vlasov--Fokker--Planck system and its large-time behavior are investigated in \cite{CDM11}. Later, this result is extended to the compressible Navier--Stokes or Euler equations in \cite{CKL13, DL13, LMW17}. The global existence of weak solutions to the compressible Navier--Stokes--Vlasov--Fokker--Planck is shown in \cite{LSpre, MV07}.

However, when it comes to the moderately thick sprays case, there is little literature. The global existence of weak solutions to the incompressible or compressible Navier--Stokes--Vlasov system with a linear particle interaction operator describing the breakup phenomena is studied in \cite{GY20, YY18}. The local-in-time unique classical solution for the compressible Euler equations coupled with the Vlasov--Boltzmann equation with a hard-sphere type collision kernel is constructed in \cite{M10}. Recently, the Navier--Stokes--BGK system is dealt with and the global existence of weak solutions and local existence of strong solutions are obtained in \cite{CY20} and \cite{CLY21}, respectively. Apart from those existence results, we also refer to \cite{CCK16, Choi16_2, C17, CJ21, CJ22, CJ22+, GJV04, GJV04_2, H22+, HM22+, HMI20, MV08} for hydrodynamic limits, large-time behaviors, and finite-time blow-up phenomena. 

 
In the current work, we prove the local-in-time well-posedness for the compressible Navier--Stokes--BGK (in short, NS--BGK) system \eqref{main_sys}. We would like to mention that the BGK operator is highly nonlinear and the density-dependent friction force gives a strong coupling between particles and fluid, thus it causes significant difficulties in analysis.

\subsection{Difficulties and comparison with previous works}

There is little literature on the well-posedness theory for the kinetic-fluid model with fluid density-dependent friction forces. One of the main difficulties in analysis arises from the term $\rho v \cdot \nabla_v f$ in the kinetic equation of \eqref{main_sys}. This term depends on both the particle velocity $v$ linearly and the fluid density $\rho$. To show the existence of strong solutions in Sobolev spaces, it is necessary to estimate the term $\rho v \cdot \nabla_v f$ and its derivatives.  This implies that we need to handle the velocity growth of $f$ in the analysis. More specifically, if we estimate higher-order derivatives of $f$ with certain velocity weights of order $m \in \N \cup \{0\}$, then the term $\rho v \cdot \nabla_v f$ requires us to control the corresponding weight of order $m+1/2$. In \cite{CJ22}, the authors imposed an exponential weight $e^{|v|^2}$ on $f$ to obtain a dissipation on $f$ with the weight $|v|e^{|v|^2}$, which is enough to handle the aforementioned difficulty. However, our system \eqref{main_sys} has also the local Maxwellian $\mm$ in the BGK operator, and this cannot be readily bounded in a Sobolev space with the weight $e^{|v|^2}$. Thus, we imposed the exponential weight $e^{(1+|v|^2)^{k/2}}$ with $k \in (1,2)$ on $f$ to bound the local Maxwellian in the weighted Sobolev space and simultaneously obtain a dissipation on $f$ with the weight $|v|^{k/2}e^{|v|^2}$. Here we note that the case $k=1$ can also be handled if we impose a smallness condition on solutions, especially on the fluid density $\rho$. However, we did not include this case to deal with large initial data.   

%

To the best of the authors' knowledge, well-posedness for the kinetic-fluid model with fluid density-dependent drag force and nonlinear collision operator for particle interactions has not been established yet. In the absence of nonlinear collisional operators, to the best of our knowledge, there are only four papers are available on the local well-posedness theory for the Euler--Vlasov equations \cite{BD06}, the inhomogeneous Naiver--Stokes--Vlasov equations \cite{CK15},  the kinetic thermomechanical Cucker--Smale equation with the compressible Navier--Stokes system \cite{CHJK20}, and the Vlasov equation coupled with the compressible Navier--Stokes system with degenerate viscosity and vacuum \cite{CJ22}. In \cite{BD06, CHJK20, CK15}, to handle the density-dependent friction force, the finite-speed of the propagation of the support of $f$ in velocity is significantly used, i.e., the initial data $f_0$ is compactly supported in velocity, $f(t)$ has a compact support in velocity in a compact time interval. By using that, roughly speaking, we can bound the term $\rho v \cdot \nabla_v f$ by $\rho \nabla_v f$, and also its derivatives by the derivatives of $\rho \nabla_v f$. Thus, the problem with velocity growth does not occur in those works.

On the other hand, in \cite{CLY21, M10}, the Boltzmann or BGK collisional operator is considered in the kinetic equation, and local well-posedness theory is developed. We would like to remark that in this case, we cannot have a finite speed of propagation of the velocity-support due to the collisional operator $Q(f)$. For that reason, the polynomial, exponential, or Mittag-Leffler weight in velocity is employed for Boltzmann or BGK equation \cite{ACGM13, F21, PP93,TAGP18}. However, in \cite{CLY21, M10}, the friction force only depends on $u-v$, but independent of the fluid density $\rho$. Thus, the term $(\rho - \bar\rho) v \cdot \nabla_v f$ does not appear, and there is no additional difficulty with  the velocity growth of $f$ in those works.  
\subsection{Outline of our strategy}

In order to handle the strong nonlinear coupling between the kinetic particles and fluids, we propose a exponential weighted solution space for $f$. To be more concrete, for $p, k \in [1,\infty)$, we denote by $L_k^p = L_k^p(\T^3 \times \R^3)$ the space of measurable functions which are weighted by $e^{\langle v \rangle^k}$, where $\langle v \rangle := (1+|v|^2)^{1/2}$, and equipped with the norm
\[
\|f\|_{L_k^p} :=  \|e^{\langle v \rangle^k}f\|_{L^p} = \lt(\inttr e^{p\langle v \rangle^k}|f|^p \,dxdv\rt)^{1/p}.
\]
The limiting case $p = \infty$ is defined by
\[
\|f\|_{L_k^\infty} :=  \esssup_{x,v} e^{\langle v \rangle^k}|f(x,v)|.
\]
For any $s \in \N$, $W_k^{s,p} = W_k^{s,p}(\T^3 \times \R^3)$ represents for $L_k^p$ Sobolev space of $s$-th order equipped with the norm
\[
\|f\|_{W_k^{s,p}}:= \lt(\sum_{|\alpha|+|\beta|\le s}\inttr e^{p\langle v \rangle^k} |\pa_x^\alpha \pa_v^\beta f|^p\,dxdv \rt)^{1/p}.
\]
The space $W_k^{s,\infty}=W_k^{s,\infty}(\T^3 \times \R^3)$ is analogously defined. In particular, when $p=2$, we denote by $H^s_k = W_k^{s,2}$. 

By employing the function space $H^2_k$ for $f$ with $k \in(1,2)$, we establish the local well-posedness for the NS--BGK system \eqref{main_sys} (Theorem \ref{T1.1}). Although we also encounter the problem with the velocity growth of $f$, in our weighted Sobolev space, we figure out an appropriate way of using the viscous effect from the Navier--Stokes system together with the dissipative effect from the friction force, see Lemma \ref{f_est_l2}. We would like to stress that in our strategy, neither polynomial weights nor the exponential weight with $k \notin (1,2)$ is applicable. Note that the above function space for the kinetic equation in \eqref{main_sys}, BGK equation, is different from that of previous works \cite{PP93, Y15}. Thus, we need to redevelop a theory for the existence of solutions to the BGK equation in our newly defined solution space for $f$. In fact, we establish similar results to that of \cite{PP93, Y15} in our solution space $H^2_k$ or $W^{1,\infty}_k$ with $k \in(1,2)$. Note that $H_k^2$-regularity for $f$ does not imply the boundedness of $f$. On the other hand, if we assume additional boundedness assumptions on the initial data $f_0$, then we have $f\in\mc([0,T]; H_k^2(\T^3\times\R^3))\cap L^\infty(0,T; W^{1,\infty}_k(\T^3))$ (Theorem \ref{T2.1}). In fact, this additional regularity of solutions $f \in L^\infty(0,T; W^{1,\infty}_k(\T^3))$ can simplify many computations made in Sections \ref{sec:2}--\ref{sec:4} below.

\subsection{Main results}

Before presenting our main results, we introduce several notations used throughout the paper. For simplicity, we often drop $x$-dependence of differential operators $\pa_x$, $\nabla_x$, and $\Delta_x$, i.e. $\pa_x = \pa$, $\nabla_x = \nabla$, and $\Delta_x = \Delta$. We denote by $C$ a generic positive constant. 

We then define a notion of our regular solution to the NS-BGK system \eqref{main_sys}.

\begin{definition}\label{D1.1}
For $T\in(0,\infty)$, we say a triplet $(f,\rho,u)$ is a regular solution to system \eqref{main_sys}-\eqref{init} on $[0,T]$ if it satisfies \eqref{main_sys} in the sense of distributions with the following regularity:
\begin{enumerate}
\item[(i)] $f\in\mc([0,T]; H_k^2(\T^3\times\R^3))$ with $k\in(1,2)$, \vspace{.1cm}
\item[(ii)] $\rho\in\mc([0,T];H^3(\T^3))$, and $u\in\mc([0,T];H^3(\T^3))\cap L^2([0,T];H^4(\T^3))$.
\end{enumerate}
\end{definition}

\begin{theorem}\label{T1.1}
For given $N<M$, there exists $T^*>0$ only depending on $M$ and $N$ such that if the initial data satisfies the following conditions:
\begin{align*}
&{\rm (i)}~~ \max\lt\{ \|f_0\|_{H_k^2}^2, \frac{4\gamma}{(\gamma-1)^2}\|(\rho_0)^{\frac{\gamma-1}{2}}-1\|_{H^3}^2+ \|u_0\|_{H^3}^2\rt\}<N \quad \mbox{and} \quad \inf_{x\in\T^3} \rho_0^{\frac{\gamma-1}{2}}(x) >\delta>0, \\
&{\rm (ii)}~~  \mbox{for some } a>0 \mbox{ and }  \e_1 > 0, \quad f_0(x,v) \ge \e_1 e^{-(1+a)\langle v\rangle^k} \quad \mbox{for all} \ (x,v)\in\T^3 \times \R^3,
\end{align*}
then the system \eqref{main_sys}-\eqref{init} admits a unique regular solution on $[0,T^*]$ satisfying
\[
\max\lt\{ \sup_{0\le t \le T^*} \|f(t)\|_{H_k^2}^2, \sup_{0\le t \le T^*}\lt(\frac{4\gamma}{(\gamma-1)^2}\|\rho^{\frac{\gamma-1}{2}}(t)-1\|_{H^3}^2 + \|u(t)\|_{H^3}^2 \rt)\rt\} <M
\]
\[\mbox{and}  \quad \inf_{(x,t)\in\T^3\times[0,T^*]}\rho^{\frac{\gamma-1}{2}}(x,t)>\frac{\delta}{2}. \]
\end{theorem}

Note that the above theorem does not give the bounded solution $f$  to the BGK equation. In this regard, our second theorem provides more regular solution $f$. 
\begin{theorem}\label{T2.1}
For given $N<M$, there exists $T^*>0$ only depending on $M$ and $N$ such that if the initial data satisfies the following conditions:
\begin{align*}
&{\rm (i)}~~ \max\lt\{\|f_0\|_{W_k^{1,\infty}}^2, \|f_0\|_{H_k^2}^2, \frac{4\gamma}{(\gamma-1)^2}\|(\rho_0)^{\frac{\gamma-1}{2}}-1\|_{H^3}^2+ \|u_0\|_{H^3}^2\rt\}<N \quad \mbox{and} \quad \inf_{x\in\T^3} \rho_0^{\frac{\gamma-1}{2}}(x) >\delta>0, \\
&{\rm (ii)}~~\mbox{for some } a>0 \mbox{ and }  \e_1^2<N, \quad f_0(x,v) \ge \e_1 e^{-(1+a)\langle v\rangle^k} \quad \mbox{for all} \ (x,v)\in\T^3 \times \R^3,
\end{align*}
then the system \eqref{main_sys}-\eqref{init} admits a unique regular solution on $[0,T^*]$ satisfying
\[
\max\lt\{\sup_{0\le t \le T^*} \|f(t)\|_{W_{k}^{1,\infty}}^2, \sup_{0\le t \le T^*} \|f(t)\|_{H_k^2}^2, \sup_{0\le t \le T^*}\lt(\frac{4\gamma}{(\gamma-1)^2}\|\rho^{\frac{\gamma-1}{2}}(t)-1\|_{H^3}^2 + \|u(t)\|_{H^3}^2 \rt)\rt\} <M
\]
\[\mbox{and}  \quad \inf_{(x,t)\in\T^3\times[0,T^*]}\rho^{\frac{\gamma-1}{2}}(x,t)>\frac{\delta}{2}. \]
\end{theorem}

\begin{remark}Even though we only provide the well-posedness theory for the NS--BGK system \eqref{main_sys} with the isentropic pressure, i.e., $\gamma > 1$, our framework can be applied to the case with isothermal pressure, $\gamma = 1$. More precisely, if the initial data satisfies the following conditions:
\begin{align*}
&{\rm (i)}~~ \max\lt\{\|f_0\|_{W_k^{1,\infty}}^2, \|f_0\|_{H_k^2}^2, \|\log\rho_0\|_{H^3}^2+ \|u_0\|_{H^3}^2\rt\}<N \quad \mbox{and} \quad \inf_{x\in\T^3} \rho_0(x) >\delta>0, \\
&{\rm (ii)}~~\mbox{for some } a>0 \mbox{ and }  \e_1^2<N, \quad f_0(x,v) \ge \e_1 e^{-(1+a)\langle v\rangle^k} \quad \mbox{for all} \ (x,v)\in\T^3 \times \R^3,
\end{align*}
then following almost the same argument, we can find a unique regular solution on $[0,T^*]$ to the system \eqref{main_sys}-\eqref{init} satisfying
\[
\max\lt\{\sup_{0\le t \le T^*} \|f(t)\|_{W_{k}^{1,\infty}}^2, \sup_{0\le t \le T^*} \|f(t)\|_{H_k^2}^2, \sup_{0\le t \le T^*}\lt(\|\log\rho\|_{H^3}^2 + \|u(t)\|_{H^3}^2 \rt)\rt\} <M
\]
\[\mbox{and}  \quad \inf_{(x,t)\in\T^3\times[0,T^*]}\rho(x,t)>\frac{\delta}{2}. \]
\end{remark}

\begin{remark} In Theorems \ref{T1.1} and \ref{T2.1}, if the initial data is sufficiently small in our solution space, then the life span of solutions can be extended over a certain fixed time. However, in \cite{M10} where the local well-posedness for the compressible Euler--Boltzmann system is established, no matter how small the initial data are, there is a fixed upper bound on the life span of solutions. 
\end{remark}

\begin{remark}\label{rmk_comm} As one may expect, the proof of Theorem \ref{T2.1} would be easier than that of Theorem \ref{T1.1}. In Theorem \ref{T1.1}, we only deal with $H^2_k$ regularity for $f$, thus it does not imply the boundedness of $f$ in three dimensions. Thus, we cannot make use of $\|f\|_{L^\infty}$ in estimating the local Maxwellian and other macroscopic quantities  $\rho_f$, $u_f$, and $T_f$. For that reason, the proof of Theorem \ref{T1.1} requires more delicate analyses on the BGK operator and the drag force in the compressible Navier--Stokes equations.
\end{remark}

\begin{remark} Note that $Q$ satisfies the following cancellation properties:
\[
\intr Q(f) (1, v, |v|^2 )\,dv = 0.
\]
This provides that one can apply almost the same argument as in \cite{Choi16_2} to have a priori estimate of large time behavior of solutions to the NS-BGK system \eqref{main_sys}. More precisely, let us define a total modulated energy of the system \eqref{main_sys} as 
\[
\ml(t) := \inttr |v-v_c|^2f \,dxdv + \intt \rho|u-m_c|^2\,dx + \intt (\rho-\rho_c)^2\,dx +  |v_c-m_c|^2,
\]
where 
\[
\rho_c :=\intt \rho\,dx, \quad  v_c := \frac{\inttr vf\,dxdv}{\inttr f\,dxdv}, \quad \mbox{and} \quad m_c:= \frac{\intt \rho u \,dx}{\intt \rho\,dx}.
\]
We then assume that the global classical solutions $(f,\rho,u)$ to \eqref{main_sys} satisfies the following conditions:
\[
\rho_f \in L^\infty(\R_+; L^{3/2}(\T^3)), \quad \rho, u \in L^\infty(\T^3 \times \R_+),
\]
and
\[
\inf_{(x,t) \in \T^3 \times \R_+} \rho(x,t) > 0.
\]
Then we have the exponential decay of the energy function $\ml(t)$ as time goes to infinity:
\[
\ml(t)\le C\ml(0)e^{-Ct}, \quad \forall t>0,
\]
where $C>0$ is a constant independent of $t$.  In particular, this implies
\begin{align*}
&m_c(t), v_c(t)  \to \frac{1}{\rho_c + f_c}\lt(\inttr vf_0\,dxdv + \intt \rho_0 u_0 \,dx  \rt), \quad f_c := \inttr f_0\,dxdv,\\[2mm]
&\rho \to \rho_c \quad \mbox{in } L^2(\T^3), \qquad \rho u \to \rho_c m_c \quad \mbox{in } L^1(\T^3), \quad \mbox{and} \quad d_{{\rm BL}}(f\,dxdv, \rho_f dx \otimes \delta_{v_c}(dv))  \to 0
\end{align*}
as $t \to \infty$, where $d_{{\rm BL}}$ denotes the bounded Lipschitz distance.
\end{remark}

\subsection{Organization of the paper} The rest of this paper is organized as follows. In Section \ref{sec:2}, we make preliminary preparations on the estimates for the local Maxwellian $\mm(f)$ in our proposed $H^2_k$ space and macroscopic quantities $\rho_f$, $u_f$, and $T_f$. Section \ref{sec:3} is devoted to the construction of the approximate solutions and its uniform bound estimates locally in time. In Section \ref{sec:4}, we provide that the approximate solutions are Cauchy sequences in the proposed Sobolev spaces. From which, we complete the proof of Theorem \ref{T1.1}. Finally, in Section \ref{sec:5}, we discuss the existence and uniqueness of $W^{1,\infty}_k$-solutions $f$ which proves Theorem \ref{T2.1}.

%
%
%
%
%
%

\section{Preliminaries}\label{sec:2}
\setcounter{equation}{0}

We begin with an auxiliary lemma which will be frequently used in the rest of this paper.
\begin{lemma}\label{lem_uf} Let $d\geq 1$ and $g=g(x,v)$ be a sufficiently regular function satisfying $g \in H^2_{x,v}(\T^d \times \R^d)$. If we set
\[
h(x) := \lt(\int_{\R^d} |g(x,v)|^2\,dv \rt)^{1/2},
\]
then $h \in H^1(\T^d)$ with
\bq\label{le}
\|h\|_{L^2} = \|g\|_{L^2} \quad \mbox{and} \quad \|\nabla h\|_{L^2} \leq \|\nabla g\|_{L^2}.
\eq
Furthermore, if $h(x) \geq c_h > 0$ for all $x \in \T^d$, then there exists a constant $C>0$ independent of $g$ such that 
\[
\|\nabla^2 h\|_{L^2} \leq \frac{C}{c_h} \sum_{\ell = 1}^2 \|\nabla^\ell g\|_{L^2}^2 + C\|\nabla^2 g\|_{L^2}.
\]
\end{lemma}
\begin{proof} We first easily find $\|h\|_{L^2} = \|g\|_{L^2}$. By using H\"older's inequality, we also obtain
\[
|\pa h(x)| = \frac{1}{h(x)} \lt|\int_{\R^d} g(x,v) \pa g(x,v)\,dv \rt| \leq \lt(\int_{\R^d} |\pa g(x,v)|^2\,dv\rt)^{1/2},
\]
and thus, $\|\nabla h\|_{L^2} \leq \|\nabla g\|_{L^2}$. 

For the estimate of $\|h\|_{\dot{H}^2}$, we notice that for $i,j=1,\dots,d$,
\begin{align*}
\pa_{ij} h &= - \frac12 \frac{\pa_j h}{h^2} \int_{\R^d}  g (\pa_i g)\,dv + \frac{1}{ h } \int_{\R^d}   (\pa_j g)(\pa_i g)\,dv + \frac{1}{ h } \int_{\R^d}   g(\pa_{ij} g)\,dv.
\end{align*}
Applying H\"older's inequality to the above gives
\begin{align*}
|\nabla^2  h | \leq \frac{C}{ h } \int_{\R^d}   |\nabla g|^2\,dv + C\lt(\int_{\R^d}   |\nabla^2 g|^2\,dv\rt)^{1/2} \leq \frac{C}{c_h} \int_{\R^d} |\nabla g|^2\,dv + C\lt(\int_{\R^d}   |\nabla^2 g|^2\,dv\rt)^{1/2}.
\end{align*}
Thus we have
\bq\label{le2}
\|\nabla^2  h\|_{L^2} \leq  \frac{C}{c_h} \|\bar h\|_{L^4}^2 + C\|\nabla^2 g\|_{L^2} \leq  \frac{C}{c_h} \|\bar h\|_{H^1}^2 + C\|\nabla^2 g\|_{L^2},
\eq
where
\[
\bar h := \lt( \int_{\R^d}  |\nabla g|^2\,dv\rt)^{1/2}.
\]
On the other hand, by \eqref{le}, we get
\[
\|\bar h\|_{H^1}^2 \leq  \sum_{\ell = 1}^2 \|\nabla^\ell g\|_{L^2}^2.
\]
Combining this and \eqref{le2} concludes the desired result.
\end{proof}

In the lemma below, we show the upper bound estimates on the macroscopic quantities $\rho_f$, $u_f$, and $T_f$. Since its proof is rather lengthy and technical, we postpone it to Appendix \ref{app.A0} for smoothness of reading.
\begin{lemma}\label{L2.20}
Suppose that $\|f\|_{H_k^{2}}<\infty$ for $k\in(1,2)$ and $\rho_f$ and $f$ satisfy
\[
\rho_f + \lt(\intr   f^2\,dv\rt)^{1/2}> c_2.
\]
Then there exists $C = C(c_2, k)>0$ independent of $f$ such that
\begin{itemize}
\item[(i)] $\|\rho_f\|_{H^\ell} \leq C\|f\|_{H^\ell_k}$  for $\ell = 0,1,2$, 
\item[(ii)]  $\|u_f\|_{H^2} \leq C\|f\|_{H^2_k}(1 + \|f\|_{H^2_k}^3)$, and $\|T_f\|_{H^2} \leq C\|f\|_{H^2_k}(1 + \|f\|_{H^2_k}^{11})$.
\end{itemize}
Here $\rho_f, u_f$, and $T_f$ are given as in \eqref{maco}.
\end{lemma}

We then show the relation between $\rho_f$ and $T_f$, which will be used later to obtain the lower bound estimate of $T_f$, and thus the local Maxwellian $\mm(f)$ is well-defined.
\begin{lemma}\label{L2.1}
There exists a constant $C>0$ independent of $f$ such that
\[
\rho_f \le C \lt(\intr f^2\,dv\rt)^{1/2}T_f^{3/4}.
\]
Here $\rho_f$ and $T_f$ are given as in \eqref{maco}.
\end{lemma}
\begin{proof} For any $R>0$, we estimate
\[
\rho_f = \intr f\,dv = \lt(\int_{|u_f - v| > R} +  \int_{|u_f - v| \le R}\rt) f\,dv \leq \frac1{R^2} \intr |u_f -v|^2 f\,dv + C_0\lt(\intr f^2\,dv\rt)^{1/2} R^{3/2},
\]
where $C_0 > 0$ independent of $f$. We now choose $R>0$ such that 
\[
\frac1{R^2} \intr |u_f -v|^2 f\,dv = C_0\lt(\intr f^2\,dv\rt)^{1/2} R^{3/2} \quad \mbox{i.e.} \quad R = \lt(\frac{3\rho_f T_f}{C_0\lt(\intr f^2\,dv\rt)^{1/2}} \rt)^{2/7},
\]
then
\[
\rho_f \leq 2 \lt(\intr |u_f -v|^2 f\,dv \rt)^{3/7}\lt(C_0\lt(\intr f^2\,dv\rt)^{1/2} \rt)^{4/7} = 2\lt(3\rho_f T_f \rt)^{3/7}\lt(C_0\lt(\intr f^2\,dv\rt)^{1/2} \rt)^{4/7}.
\]
Hence we have
\[
\rho_f \leq 2^{7/4} 3^{3/4}C_0\lt(\intr f^2\,dv\rt)^{1/2} T_f^{3/4}.
\]
\end{proof}

\begin{remark}
If $f \in L^\infty(\T^3 \times \R^3)$, we obtain
\[
\rho_f \leq C\|f\|_{L^\infty} T_f^{3/2}.
\]
This estimate is typically used in the study of existence of solutions to the BGK model, for instance see \cite{PP93, Y15}.
\end{remark}

We next provide the bound estimate on the local Maxwellian $\mm(f)$ in our solution space. 
\begin{lemma}\label{L2.2}
Suppose that $\|f\|_{H^2_k}<\infty$ for $k\in(1,2)$ and $\rho_f$, $u_f$ and $T_f$, given as in \eqref{maco}, satisfy
\[
\rho_f + |u_f| + T_f< c_1,  \quad \rho_f>c_2, \quad \mbox{and} \quad T_f > c_3^{-1}. 
\]
Then we have
\[  
\|\mm(f)\|_{H_k^2} \le  C(1+c_3)^3e^{Cc_1^{\frac{k}{2-k}}}\|f\|_{H^2_k}(1 + \|f\|_{H^2_k}^{24}),
\]
where $C$ depends only on  $c_2$ and $k$.
\end{lemma}
\begin{proof}  First, we use Young's inequality to get
\begin{align*}
 \la v \ra^k  \le  2^{\frac k2}\lt( 1+|v|^k\rt) &\le 2^{\frac k2} + 2^{\frac{3k}{2}}\lt( |v-u_f|^k + |u_f|^k\rt)\\
 &\le 2^{\frac k2} + 2^{\frac{3k}{2}} c_1^k + 2^{\frac{3k}{2}}\lt( T_f^{\frac k2}\frac{|v-u_f|^k}{T_f^{\frac k2}} \rt)\\
 &\le 2^{\frac k2} + 2^{\frac{3k}{2}} c_1^k  + \frac{\lt(2^{\frac{3k}{2}} T_f^{\frac k2} (2k)^{\frac k2}\rt)^{\frac{2}{2-k}}}{\frac{2}{2-k}} + \frac{|v-u_f|^2}{4T_f}\\
 &\le 2^{\frac k2} + 2^{\frac{3k}{2}} c_1^k +  \frac{(2-k) \lt(16c_1k\rt)^{\frac{k}{2-k}}}{2}+ \frac{|v-u_f|^2}{4T_f},
\end{align*}
where we also used
\[
|a+b|^p \le 2^p(|a|^p + |b|^p), \quad p>0.
\]
Then we obtain
\bq\label{weight_calc0}
e^{ \la v \ra^k} \mm (f) \le Ce^{C\lt(c_1^k + c_1^{\frac{k}{2-k}}\rt)} \frac{\rho_f}{(2\pi T_f)^{3/2}}e^{-\frac{|v-u_f|^2}{4T_f}} \le Ce^{Cc_1^{\frac{k}{2-k}}}\frac{\rho_f}{(2\pi T_f)^{3/2}}e^{-\frac{|v-u_f|^2}{4T_f}},
\eq
where $C>0$ is a constant depending only on $c_2$ and $k$ and we used $k \in (1,2)$. Thus, we find
\[
\|\mm(f)\|_{L_k^2}  = \|e^{\la v \ra^k} \mm(f)\|_{L^2} \le  Ce^{Cc_1^{\frac{k}{2-k}}}\lt\| \frac{\rho_f}{(2\pi T_f)^{3/2}}e^{-\frac{|u_f-v|^2}{4T_f}}\rt\|_{L^2} \le   Ce^{Cc_1^{\frac{k}{2-k}}} \|\rho_f\|_{L^2}\le Ce^{Cc_1^{\frac{k}{2-k}}} \|f\|_{L_k^2}.
\]
Here $C > 0$ only depends on $c_2$ and $k$. For first-order derivatives, we can deduce from the estimates in the previous step that for $i=1,2,3$
\begin{align*}
\|\pa_i \mm(f)\|_{L_k^2} &\le Ce^{Cc_1^{\frac{k}{2-k}}}\lt\| \frac{e^{-\frac{|u_f-v|^2}{4T_f}}}{(2\pi T_f)^{3/2}} \bigg( |\pa_i \rho_f| + \frac{\rho_f |\pa_i T_f|}{T_f} + |\pa_i u_f \cdot (u_f -v)| + \frac{|u_f-v|^2 |\pa_i T_f|}{T_f^2}\bigg)\rt\|_{L^2}\\
&\le  C(1+c_3)^2e^{Cc_1^{\frac{k}{2-k}}}\lt( \|\pa_i \rho_f\|_{L^2} + \|\rho_f \pa_i T_f\|_{L^2} + \|\pa_i u_f\|_{L^2} + \|\pa_i T_f\|_{L^2}\rt)\\
&\le C(1+c_3)^2e^{Cc_1^{\frac{k}{2-k}}}\lt( \|\pa_i \rho_f\|_{L^2} + (\|\rho_f\|_{H^2} + 1) \| \pa_i T_f\|_{L^2} + \|\pa_i u_f\|_{L^2}\rt)\cr
&\leq C(1+c_3)^2e^{Cc_1^{\frac{k}{2-k}}}\|f\|_{H^2_k}\lt(1 + \|f\|_{H^2_k}^{12}\rt)
\end{align*}
due to Lemma \ref{L2.20}. We also have
\[\begin{aligned}
\|\pa_{v_i}\mm(f)\|_{L_k^2} &\le Ce^{Cc_1^{\frac{k}{2-k}}}\lt\| \frac{\rho_f |u_f -v|}{T_f^{5/2}}e^{-\frac{|u_f-v|^2}{4T_f}} \rt\|_{L^2} \\
&\le  C(1+c_3)e^{Cc_1^{\frac{k}{2-k}}}\|\rho_f\|_{L^2}\le C(1+c_3)e^{Cc_1^{\frac{k}{2-k}}}\|f\|_{L_k^2}.
\end{aligned}\]
For second-order derivative estimates, we first have that for $i,j=1,2,3$
\begin{align*}
\pa_{ij} \mm(f) &= \pa_i\lt[\Bigg(\pa_j \rho_f -  \frac32 \frac{\rho_f \pa_j T_f}{ T_f} -\rho_f\bigg(\frac{\pa_j u_f \cdot (u_f - v)}{T_f} - \frac{|v-u_f|^2}{2T_f^2} \pa_j T_f\bigg)\Bigg)\frac{e^{-\frac{|v-u_f|^2}{2T_f}}}{(2\pi T_f)^{3/2}} \rt]\\
&= \Bigg[ \pa_{ij}\rho_f - \frac32 \bigg(\frac{\pa_i \rho_f \pa_j T_f}{T_f} + \frac{\rho_f \pa_{ij}T_f}{T_f} - \frac{\rho_f \pa_i T_f \pa_j T_f}{T_f^2}\bigg) - \pa_i \rho_f\bigg(\frac{\pa_j u_f \cdot (u_f - v)}{T_f} - \frac{|v-u_f|^2}{2T_f^2} \pa_j T_f\bigg)\\
&\quad -\rho_f \bigg( \frac{\pa_{ij}u_f \cdot (u_f -v)}{T_f} + \frac{\pa_i u_f \cdot \pa_j u_f}{T_f} - \frac{(u_f-v)\cdot( \pa_j u_f \pa_i T_f + \pa_i u_f \pa_j T_f)}{T_f^2}  \\
&\hspace{5cm}+ \frac{|v-u_f|^2}{T_f^3}\pa_i T_f \pa_j T_f - \frac{|v-u_f|^2}{2T_f^2}\pa_{ij}T_f\bigg)\Bigg]\frac{e^{-\frac{|u_f-v|^2}{2T_f}}}{(2\pi T_f)^{3/2}}\\
&\quad + \Bigg[\pa_j \rho_f -  \frac32 \frac{\rho_f \pa_j T_f}{ T_f} -\rho_f\bigg(\frac{\pa_j u_f \cdot (u_f - v)}{T_f} - \frac{|v-u_f|^2}{2T_f^2} \pa_j T_f\bigg)\Bigg]\\
&\qquad \times\bigg(\frac{\pa_i u_f \cdot (u_f - v)}{T_f} - \frac{|v-u_f|^2}{2T_f^2} \pa_i T_f - \frac32 \frac{\pa_i T_f}{T_f}\bigg)\frac{e^{-\frac{|v-u_f|^2}{2T_f}}}{(2\pi T_f)^{3/2}},
\end{align*}
\[
\pa_i  \pa_{v_j}\mm(f) = \lt( \frac{\pa_i (u_f)_j}{T_f} - \frac{(u_f - v)_j}{T_f^2} \pa_i T_f\rt)\mm(f) +  \frac{(u_f-v)_j}{T_f} \pa_i \mm(f),
\]
and
\[
\pa_{v_i v_j}\mm(f) = \lt(\frac{-\delta_{ij}}{T_f} + \frac{(u_f-v)_i (u_f-v)_j}{T_f^2} \rt)\mm(f),
\]
where $\delta_{ij}$ denotes Kronecker's delta. Thus, one gets
\begin{align*}
&\|\pa_{ij} \mm(f)\|_{L_k^2}\cr
&\quad \le C(1+c_3)^3e^{Cc_1^{\frac{k}{2-k}}}\Bigg\| \Bigg( |\pa_{ij}\rho_f| + |\pa_i \rho_f ||\pa_j T_f| + \rho_f |\pa_{ij}T_f| + \rho_f |\pa_i T_f| |\pa_j T_f|\\
&\hspace{2.5cm}+ |\pa_i \rho_f| \lt(|\pa_ju_f| |u_f-v|+ |\pa_j T_f| |u_f-v|^2\rt) + \rho_f \lt(|\pa_{ij} u_f| |u_f-v| + |\pa_i u_f| |\pa_j u_f|\rt)\\
&\hspace{2.5cm} + \rho_f |u_f-v| (|\pa_j u_f||\pa_i T_f| + |\pa_i u_f||\pa_j T_f|)+\rho_f |u_f-v|^2 (|\pa_i T_f||\pa_j T_f| + |\pa_{ij}T_f|)\\
&\hspace{2.5cm} +\lt(|\pa_j \rho_f| + \rho_f |\pa_j T_f| + \rho_f\lt(|\pa_j u_f| |u_f-v| + |u_f-v|^2 |\pa_j T_f|\rt)\rt)\\
&\hspace{3cm}\times\lt(|\pa_i u_f| |u_f-v| + \lt(|u_f-v|^2 + 1\rt)|\pa_i T_f|\rt) \Bigg)\frac{e^{-\frac{|u_f-v|^2}{4T_f}}}{(2\pi T_f)^{3/2}}\Bigg\|_{L^2}\\
&\le  C(1+c_3)^3e^{Cc_1^{\frac{k}{2-k}}} \big( \|\pa_{ij} \rho_f\|_{L^2} + \|\pa_i \rho_f\|_{H^1}\|\pa_j T_f\|_{H^1} + \|\rho_f\|_{H^2}(\|\pa_{ij}T_f\|_{L^2}+ \|\pa_i T_f\|_{H^1}\|\pa_j T_f\|_{H^1})  \\
&\hspace{3cm}+\|\pa_i \rho_f\|_{H^1}\lt(\|\pa_j u_f\|_{H^1} + \|\pa_j T_f\|_{H^1} \rt) \cr
&\hspace{3cm}+ \|\rho_f\|_{H^2}\lt(\|\pa_j u_f\|_{H^1}\|\pa_i T_f\|_{H^1} + \|\pa_i u_f\|_{H^1}\|\pa_j T_f\|_{H^1} \rt) \cr
&\hspace{3cm}+ \|\rho_f\|_{H^2}\lt(\|\pa_j T_f\|_{H^1}\|\pa_i T_f\|_{H^1} + \|\pa_{ij} T_f\|_{L^2}\rt) \cr
&\hspace{3cm}+ \lt(\|\pa_j \rho_f\|_{H^1} + \|\rho_f\|_{H^2}\lt(\|\pa_j T_f\|_{H^1} + \|\pa_j u_f\|_{H^1}\rt)\rt)\lt(\|\pa_i u_f\|_{H^1} + \|\pa_i T_f\|_{H^1} \rt)\big)\\
&\le C(1+c_3)^3e^{Cc_1^{\frac{k}{2-k}}}\|\rho_f\|_{H^2}\lt(1 + \|u_f\|_{H^2} + \|T_f\|_{H^2} + (\|u_f\|_{H^2} + \|T_f\|_{H^2})^2 \rt)\cr
&\leq C(1+c_3)^3 e^{Cc_1^{\frac{k}{2-k}}}\|f\|_{H^2_k}\lt(1 + \|f\|_{H^2_k}^{24}\rt),
\end{align*}
where $C$ only depends on $c_2$ and $k$. Moreover, we obtain
\begin{align*}
\|\pa_i \pa_{v_j}\mm(f)\|_{L_k^2} &\leq C(1+c_3)^2e^{Cc_1^{\frac{k}{2-k}}}(\|\rho_f\|_{H^2}+1)\lt(\|\pa_i \rho_f\|_{L^2} + \|\pa_i u_f\|_{L^2} + \|\pa_i T_f\|_{L^2} \rt) \cr
&\leq Ce^{Cc_1^{\frac{k}{2-k}}}\|f\|_{H^2_k}\lt(1 + \|f\|_{H^2_k}^{12}\rt)
\end{align*}
and
\[
\|\pa_{v_iv_j}\mm(f)\|_{L_k^2} \le C(1+c_2)^2e^{Cc_1^{\frac{k}{2-k}}}\|\rho_f\|_{L^2} \le Ce^{Cc_1^{\frac{k}{2-k}}}\|f\|_{L_k^2}.
\]
Finally, we collect all the above estimates to yield the desired result.
\end{proof}

We can also derive the following lemma based on estimates in the previous lemma.
\begin{lemma}\label{L2.3}
Suppose that $f$ and $g$ satisfy ($h$ denotes either $f$ or $g$)
\[
\rho_h + |u_h|+T_h \le C_1 \quad \mbox{and} \quad \rho_h + T_h \ge C_2
\]
for some constants $C_i>0$, $i=1,2$. Then we have
\[
\|\mm(f)-\mm(g)\|_{L_k^2}\le C\|f-g\|_{L_k^2},
\]
where $C$ only depends on $C_i$ $(i=1,2)$ and $k$.
\end{lemma}
\begin{proof}
We first observe
\[
|\rho_f - \rho_g| \le  \intr |f-g|\,dv \le C\lt( \intr e^{2\la v \ra^k} |f-g|^2\,dv\rt)^{1/2},
\]
\begin{align*}
|u_f -u_g| &=\lt|\frac{1}{\rho_f}\intr vf\,dv - \frac{1}{\rho_g}\intr vg\,dv\rt|\\
&= \lt|-\frac{\rho_f-\rho_g}{\rho_g} u_f + \frac{1}{\rho_g}\intr v(f-g)\,dv\rt|\\
&\le C\lt( \intr e^{2\la v \ra^k} |f-g|^2\,dv\rt)^{1/2},
\end{align*}
and
\begin{align*}
|T_f - T_g| &= \lt|\frac{1}{3\rho_f}\intr |v-u_f|^2f\,dv - \frac{1}{3\rho_g}\intr |v-u_g|^2g\,dv\rt|\\
&=\lt| -\frac{\rho_f -\rho_g}{3\rho_g}T_f + \frac{1}{3\rho_g} \intr (|v-u_f|^2 f - |v-u_g|^2 g)\,dv\rt|\\
&\le C|\rho_f -\rho_g| + C\intr |v|^2 |f-g|\,dv + C\intr |v| |u_f -u_g| f\,dv\\
&\quad + C\intr |v| |u_g| |f-g|\,dv + C\intr \Big| |u_f|^2 -|u_g|^2\Big| f\,dv + C\intr |u_g|^2 |f-g|\,dv\\
&\le C\lt( \intr e^{2\la v \ra^k} |f-g|^2\,dv\rt)^{1/2}.
\end{align*}
Moreover, we use $|e^x -e^y| \le \max\{ e^x, e^y\} |x-y| \le (e^x + e^y)|x-y|$ to get
\begin{align*}
&\lt| e^{-\frac{|u_f-v|^2}{2T_f}} - e^{-\frac{|u_g-v|^2}{2T_g}}\rt|\cr
&\quad \le \lt(e^{-\frac{|u_f-v|^2}{2T_f}} + e^{-\frac{|u_g-v|^2}{2T_g}}\rt) \lt| \frac{|u_f-v|^2}{2T_f} - \frac{|u_g-v|^2}{2T_g}\rt|\\
&\quad \le C\lt(e^{-\frac{|u_f-v|^2}{2T_f}} + e^{-\frac{|u_g-v|^2}{2T_g}}\rt) \lt( |u_f-v|^2 |T_f - T_g| + |u_f-u_g|\lt( |u_f -v| + |u_g-v|\rt)\rt).
\end{align*}
Since
\begin{align*}
\mm(f) - \mm(g) 
&= \frac{(\rho_f -\rho_g)}{(2\pi T_f)^{3/2}}e^{-\frac{|u_f-v|^2}{2T_f}} - \rho_g\frac{T_f^{3/2}- T_g^{3/2} }{(2\pi T_f T_g)^{3/2}}e^{-\frac{|u_f-v|^2}{2T_f}} + \frac{\rho_g}{(2\pi T_g)^{3/2}}\lt(e^{-\frac{|u_f-v|^2}{2T_f}} - e^{-\frac{|u_g-v|^2}{2T_g}} \rt),
\end{align*}
we combine the above estimates with \eqref{weight_calc0} to obtain
\begin{align*}
\Big|&e^{\la v \ra^k} (\mm(f)-\mm(g))\Big| \\
&\le C\lt( \frac{e^{-\frac{|u_f-v|^2}{4T_f}}}{(2\pi T_f)^{3/2}} +  \frac{e^{-\frac{|u_g-v|^2}{4T_g}}}{(2\pi T_g)^{3/2}}\rt)\lt( 1+ |u_f - v| + |u_g -v| + |u_f-v|^2\rt)\lt( \intr e^{2\la v \ra^k} |f-g|^2\,dv\rt)^{1/2}.
\end{align*}
We finally integrate the above over $\T^3 \times \R^3$ to conclude the desired result.
\end{proof}

We close this section by presenting some classical inequalities in the lemma below.
\begin{lemma}\label{lem_moser} 
 For any pair of functions $f,g \in (H^\ell \cap L^\infty)(\R^d)$, we obtain
\[
\|\nabla^\ell (fg)\|_{L^2} \le C\lt(\|f\|_{L^\infty} \|\nabla^\ell g\|_{L^2} + \|\nabla^\ell f\|_{L^2}\|g\|_{L^\infty}\rt).
\]
Furthermore, if $\nabla f \in L^\infty(\R^d)$, we have
\[
\|\nabla^\ell (fg) - f\nabla^\ell g\|_{L^2} \le C\lt(\|\nabla f\|_{L^\infty}\|\nabla^{\ell-1} g\|_{L^2} + \|g\|_{L^\infty}\|\nabla^\ell f\|_{L^2}\rt).
\]
Here $C>0$ only depends on $\ell$ and $d$.
\end{lemma}

%
%
%
%
%
%

\section{Approximations \& uniform bound estimates}\label{sec:3}
\setcounter{equation}{0}
In this section, we linearize the NS-BGK system \eqref{main_sys} and provide the uniform bound estimates on the approximate solutions. 

First, we rewrite the system to make use of the structure of symmetric hyperbolic system for the compressible Navier-Stokes equations in \eqref{main_sys} as
\begin{align}\label{main_sys2}
\begin{aligned}
&\pa_t f + v \cdot \nabla_x f + \nabla_v\cdot \lt((1+h)^{\frac{2}{\gamma-1}}(u-v)f\rt) =\rho_f^\alpha( \mm(f)-f),\\
& \pa_t h + u \cdot \nabla h + \frac{\gamma-1}{2}(1+h) \nabla \cdot u=0,\\
&\pa_t u + u \cdot \nabla u + \frac{2\gamma}{\gamma-1} (1+h) \nabla h - \frac{\mu\Delta u}{(1+h)^{\frac{2}{\gamma-1}}} = - \intr (u-v)f\,dxdv,
\end{aligned}
\end{align}
%
where  we set $1+h:=\rho^{\frac{\gamma-1}{2}}$.

\subsection{Approximate solutions} In this subsection, we construct a sequence of approximate solutions to the reformulated system \eqref{main_sys2}:
\begin{align}\label{app_seq}
\begin{aligned}
&\pa_t f^{n+1} + v \cdot \nabla f^{n+1} + \nabla_v \cdot \lt(\rho^n (u^n -v)f^{n+1}\rt) = \rho_{f^n}^\alpha(\mm(f^n) - f^{n+1}),\\
& \pa_t h^{n+1} + u^n \cdot \nabla h^{n+1} + \frac{\gamma-1}{2}(1+h^n) \nabla \cdot u^{n+1}=0,\\
& \pa_t u^{n+1} + u^n \cdot \nabla u^{n+1} + \frac{2\gamma}{\gamma-1} (1+h^n) \nabla h^{n+1} - \frac{\mu\Delta u^{n+1}}{(1+h^n)^{\frac{2}{\gamma-1}}} = - \intr (u^n-v)f^n\,dv,
\end{aligned}
\end{align}
subject to initial data and first iteration step:
\[
(f^{n+1}(x,v,0), h^{n+1}(x,0),u^{n+1}(x,0)) = (f_0(x,v), h_0(x), v_0(x))
\]
and
\[
(f^0(x,v,t), h^0(x,t), u^0(x,t)) = (f_0(x,v), h_0(x), v_0(x))
\]
for $(x,v,t)\in\T^3\times\R^3 \times (0,T)$.

For simplicity of presentation, we set
\[
\mathfrak{X}^n_{k}(T):= \max\lt\{  \sup_{0\le t \le T} \|f^n(\cdot,\cdot,t)\|_{H_k^2}^2, \  \sup_{0\le t \le T}\lt(\frac{4\gamma}{(\gamma-1)^2}\|h^n(\cdot,t)\|_{H^3}^2 + \|u^n(\cdot,t)\|_{H^3}^2 \rt)\rt\}.
\]

Our goal of this subsection is to prove the following proposition.

\begin{proposition}\label{prop_seq}
Suppose that the initial data $(f_0,\rho_0,u_0)$ satisfy the conditions in Theorem \ref{T1.1}. Then we can find $T^*>0$ depending only on $M$ and $N$ such that system \eqref{app_seq} admits the sequence of unique regular solutions $\{(f^n,h^n, u^n)\}_{n\in\N}$ on $[0,T^*]$ satisfying
\[
\sup_{n \in \N}\mathfrak{X}^n_{k}(T^*) < M \quad \mbox{and}  \quad \inf_{n \in \N} \inf_{(x,t)\in\T^3\times[0,T^*]}(1+h^n)(x,t)>\frac{\delta}{2}.
\]
\end{proposition}

For the proof of Proposition \ref{prop_seq}, we first investigate the fluid part $(h^{n+1}, u^{n+1})$.
\begin{lemma}\label{flu_est}
 Let $T\in (0,\infty)$ be a fixed constant, and suppose that the initial data $(f_0,\rho_0,u_0)$ satisfy the conditions in Theorem \ref{T1.1}. If 
\bq\label{seq_cond}
\max_{1 \leq m \leq n} \mathfrak{X}^m_{k}(T) < M \quad \mbox{and}  \quad \min_{1 \leq m \leq n} \inf_{(x,t)\in\T^3\times[0,T]}(1+h^m(x,t))>\frac{\delta}{2},
\eq
 then we can find $0<T_1\le T$ depending only on $M$ and $N$ such that the fluid system in \eqref{app_seq} admits the unique regular solution $(h^{n+1}, u^{n+1})$ on $[0,T_1]$ satisfying
\[
 \sup_{0\le t \le T_1} \lt( \frac{4\gamma}{(\gamma-1)^2}\|h^{n+1}(\cdot,t)\|^2_{H^3} + \|u^{n+1}(\cdot,t)\|^2_{H^3} \rt) + \frac{c_0\mu}{(1+M)^{\frac{2}{\gamma-1}}}\int_0^{T_1} \|\nabla u^{n+1}(\cdot,s)\|_{H^3}^2\,ds <M
\]
for some $c_0 > 0$ independent of $n$.
\end{lemma}
\begin{proof}
Since the proof is rather straightforward by now, we leave it in Appendix \ref{app.A}.
\end{proof}

Next, we obtain the infimum estimates for the fluid density $\rho^{n+1}$.
\begin{lemma}\label{inf_est}
Suppose that the initial data $(f_0,\rho_0,u_0)$ satisfy the conditions in Theorem \ref{T1.1}. If \eqref{seq_cond} holds, then we can find $0<T_2\le T_1$ depending only on $M$ and $N$ satisfying
\[
\inf_{(x,t)\in\T^3\times[0,T_2]} (1+h^{n+1}(x,t)) > \frac\delta2.
\]
\end{lemma}
\begin{proof}
Once we define the backward characteristic flow $\eta^{n+1}(s)= \eta^{n+1}(s;t,x)$ for $h^{n+1}$ as
\[
\pa_s \eta^{n+1}(s) = u^n(\eta^{n+1}(s),s), \quad \eta^{n+1}(t) = x \in \T^3,
\]
then we can deduce from the continuity equation in \eqref{app_seq} that for any $x \in \T^3$ and $0 \le t \le T_1$,
\begin{align*}
1+h^{n+1}(x,t) &= (1+h_0(\eta^{n+1}(0))) - \frac{\gamma-1}{2}\int_0^t (1+h^n)(\nabla \cdot u^{n+1})(\eta^{n+1}(s),s)\,ds\\
&\ge \delta - \frac{\gamma-1}{2} (1+\|h^n\|_{L^\infty})\|\nabla\cdot u^{n+1}\|_{L^\infty}\cr
& \ge \delta -C(1+M)^2t.
\end{align*}
Thus, we can choose a sufficiently small constant $0<T_2 \le T_1$ independent of $n$ satisfying
\[
\inf_{(x,t)\in\T^3 \times [0,T_2]} (1+h^{n+1}(x,t)) > \frac\delta2.
\]
\end{proof}

Now, it remains to estimate the kinetic density $f^{n+1}$. Before we estimate the weighted Sobolev norms for $f^{n+1}$, we provide preliminary estimates for characteristic flows and macroscopic quantities. 

Consider both forward characteristics $Z^{n+1}(s) = Z^{n+1}(s;0,z):= (X^{n+1}(s;0,z), V^{n+1}(s;0,z))$ with $z = (x,v)$ and backward characteristics $\tZ^{n+1}(s) = \tZ^{n+1}(s;t,z):= (\tX^{n+1}(s;t,z), \tV^{n+1}(s;t,z))$ given by
\bq\label{for_cha}
\begin{aligned}
&\frac{d}{ds}X^{n+1}(s) = V^{n+1}(s),\\
&\frac{d}{ds}V^{n+1}(s) = \rho^n(X^{n+1}(s),s) \lt( u^n (X^{n+1}(s),s) - V^{n+1}(s)\rt),\\
&Z^{n+1}(s)\bigg|_{s=0} = z
\end{aligned}
\eq
and
\bq\label{back_cha}
\begin{aligned}
&\frac{d}{ds}\tX^{n+1}(s) = \tV^{n+1}(s),\\
&\frac{d}{ds}\tV^{n+1}(s) = \rho^n(\tX^{n+1}(s),s) \lt( u^n(\tX^{n+1}(s),s) - \tV^{n+1}(s)\rt),\\
&\tZ^{n+1}(s)\bigg|_{s=t} = z.
\end{aligned}
\eq

We then show the growth estimate in velocity for the backward characteristic flow \eqref{back_cha} in the lemma below.
\begin{lemma}\label{back_fl}
Suppose that the initial data $(f_0,\rho_0,u_0)$ satisfy the conditions in Theorem \ref{T1.1}. If \eqref{seq_cond} holds, then we have
\[
|\tilde{V}^{n+1}(s)| \le  C\lt(1+(1+M)^{\frac{\gamma}{\gamma-1}}t\rt) \exp\lt(C(1+M)^{\frac{1}{\gamma-1}}t\rt) (1+ |v|), \quad 0 \le s \le t \le T,
\]
where $C>0$ is independent of $n$ and $T$.
\end{lemma}
\begin{proof}
From \eqref{back_cha} and $\rho^n = (1+h^n)^{\frac{2}{\gamma-1}}$, one deduces that
\begin{align*}
\tilde{X}^{n+1}(s) &= x - \int_s^t \tilde{V}^{n+1}(\tau)\,d\tau,\\
\tilde{V}^{n+1}(s) &= v\exp\lt( -\int_s^t \rho^n(\tilde{X}^{n+1}(\tau),\tau)\,d\tau\rt) - \int_s^t (\rho^n u^n)(\tilde{X}^{n+1}(\tau),\tau)\lt( -\int_\tau^t \rho^n(\tilde{X}^{n+1}(r),r)\,dr\rt)\,d\tau\\
&\le |v|\exp\lt(C(1+M)^{\frac{1}{\gamma-1}}t\rt)+ CM(1+M)^{\frac{1}{\gamma-1}}t \exp(C(1+M)^{\frac{2}{\gamma-1}}t)\\
&\le C\lt(1+(1+M)^{\frac{\gamma}{\gamma-1}}t\rt) \exp\lt(C(1+M)^{\frac{1}{\gamma-1}}t\rt) (1+ |v|).
\end{align*}
\end{proof}

Next, we obtain the uniform boundedness of macroscopic fields associated with the kinetic equation.

\begin{lemma}\label{macro_bd}
Suppose that the initial data $(f_0,\rho_0,u_0)$ satisfy the conditions in Theorem \ref{T1.1}. If \eqref{seq_cond} holds, then we can find $0<T_3\le T_2$ depending only on $M$ and $N$ satisfying
\begin{align*}
&{\rm (i)}~~\min\lt\{\inf_{(x,t)\in\T^3\times[0,T_3]}\rho_{f^n}(x,t), \ \inf_{(x,t)\in\T^3\times[0,T_3]}\lt(\intr (f^n)^2\,dv\rt)^{1/2}  \rt\}> \theta_1,\\
&{\rm (ii)}~~  \inf_{(x,t)\in\T^3\times[0,T_3]}T_{f^n}(x,t) > \theta_1 (1+M)^{-4/3},\\
&{\rm (iii)}~~\sup_{(x,t)\in\T^3 \times [0,T_3]}  \rho_{f^n}(x,t) < \theta_2(1+M), \quad \mbox{and} \cr
&{\rm (iv)}~~ \sup_{(x,t)\in\T^3 \times [0,T_3]} \lt( \rho_{f^n}(x,t) + |u_{f^n}(x,t)| + T_{f^n}(x,t)\rt) < \theta_2 (1+M)^6,
\end{align*}
where $\theta_1>0$ and $\theta_2>0$ are constants independent of $n$ and $T$.
\end{lemma}
\begin{proof}
We first notice from Lemma \ref{L2.20} that for $i = n-1, n$,
\[
\|\rho_{f^i}\|_{L^\infty} \le C\|\rho_{f^i}\|_{H^2} \leq C\|f^i\|_{H_k^2} \le C(1+M),
\]
where $C=C(k)$ is a constant independent of $n$ and $T$. Then we use the backward characteristics \eqref{back_cha} to obtain
\begin{align*}
\frac{d}{ds}f^{n+1}(\tilde{Z}^{n+1}(s),s) &= \lt(3\rho^n(\tilde{X}^{n+1}(s),s) - \rho_{f^n}^\alpha(\tilde{X}^{n+1}(s),s)\rt)f^{n+1}(\tilde{Z}^{n+1}(s),s) \\
&\quad + \rho_{f^n}^\alpha(\tilde{X}^{n+1}(s),s) \mm(f^n)(\tilde{Z}^{n+1}(s),s),
\end{align*}
and we find, for $0 \le t \le T_2$,
\begin{align*}
f^{n+1}(z,t) &=f_0(\tilde{Z}^{n+1}(0))\exp\lt(\int_0^t (3\rho^n - \rho_{f^n}^\alpha)(\tilde{X}^{n+1}(s),s)\,ds\rt) \\
&\quad + \int_0^t \rho_{f^n}^\alpha(\tilde{X}^{n+1}(s),s) \mm(f^n)(\tilde{Z}^{n+1}(s),s)\exp\lt(\int_s^t (3\rho^n - \rho_{f^n}^\alpha)(\tilde{X}^{n+1}(\tau),\tau)\,d\tau\rt)ds\cr
&\geq f_0(\tilde{Z}^{n+1}(0))\exp\lt(\int_0^t (3\rho^n - \rho_{f^n}^\alpha)(\tilde{X}^{n+1}(s),s)\,ds\rt)\cr
&\geq e^{\lt(3(\delta/2)^{\frac{2}{\gamma-1}} - C(1+M)^\alpha\rt)t}f_0(\tilde{Z}^{n+1}(0))
\end{align*}
due to $\mm(f^n) \geq 0$.
From this and Lemma \ref{back_fl}, we deduce that
\begin{align*}
\rho_{f^n} &\ge e^{\lt(3(\delta/2)^{\frac{2}{\gamma-1}} - C(1+M)^\alpha\rt)t} \intr  f_0(\tilde{Z}^n(0))\,dv\\
&\ge e^{\lt(3(\delta/2)^{\frac{2}{\gamma-1}} - C(1+M)^\alpha\rt)t}  \e_1 \intr e^{-(1+a)\langle \tilde{V}^n(0)\rangle^k} \,dv\\
&\ge e^{\lt(3(\delta/2)^{\frac{2}{\gamma-1}} - C(1+M)^{\alpha+1/2}\rt)t} \intr e^{-C(1+a)\lt[\lt(1+(1+M)^{\frac{\gamma}{\gamma-1}}t\rt) \exp\lt(C(1+M)^{\frac{1}{\gamma-1}}t\rt)\rt]^k \langle v\rangle^k} \,dv,
\end{align*}
 Since the integrand in the last equality belongs to  $L^1(\R^3)$, there exists $\theta_{11} >0$ and $0< T_{31} \le T_2$ such that
\[
\inf_{(x,t)\in\T^3\times[0,T_{31}]}\rho_{f^n}(x,t) \ge \theta_{11}.
\]
For the lower bound estimate of $T_{f^n}$, due to Lemma \ref{L2.1}, we estimate $L^\infty$ bound on
\[
\lt(\intr (f^n)^2\,dv\rt)^{1/2} =: g^n(x).
\]
For this, we use Lemma \ref{lem_uf} and this requires the lower bound estimate on $g^n$. Similarly as the above,
\[\begin{aligned}
\intr (f^n)^2\,dv& \geq e^{2\lt(3(\delta/2)^{\frac{2}{\gamma-1}} - C(1+M)^\alpha\rt)t} \intr  (f_0(\tilde{Z}^n(0)))^2\,dv \\
&\ge e^{2\lt(3(\delta/2)^{\frac{2}{\gamma-1}} - C(1+M)^{\alpha+1/2}\rt)t} \intr e^{-2C(1+a)\lt[\lt(1+(1+M)^{\frac{\gamma}{\gamma-1}}t\rt) \exp\lt(C(1+M)^{\frac{1}{\gamma-1}}t\rt)\rt]^k \langle v\rangle^k} \,dv
\end{aligned}\]
Similarly as before, we can find $\theta_{12} >0$ and $0< T_{32} \le T_2$ such that
\[
\inf_{(x,t)\in\T^3\times[0,T_{32}]}\lt(\intr (f^n)^2\,dv\rt)^{1/2} \ge \theta_{12}.
\]
Next, by Lemma \ref{lem_uf}, we obtain
 \[
\|g^n\|_{H^2} \leq C\|f^n\|_{H^2_k} + C\|f^n\|_{H^2_k}^2 \leq C(1 + M),
\]
where $C > 0$ is independent of $n$. Then we choose $T_3 := \min\{T_{31}, T_{32}\}$ and by Lemma \ref{L2.1}, the lower bound of $T_{f^n}$ can be easily obtained as
\[
\theta_{11} \leq \rho_{f^n} \leq C\|g^n\|_{L^\infty} T_{f^n}^{3/4} \leq C\|g^n\|_{H^2} T_{f^n}^{3/4}  \leq C(1 + M) T_{f^n}^{3/4}, \]
\[ \mbox{i.e.} \quad T_{f^n} \geq \lt( \frac{\theta_{11}}{C(1 + M) }\rt)^{4/3}=: \theta_{13}(1+M)^{-4/3}.
\]
We then choose $\theta_1:= \min\{\theta_{11}, \theta_{12}, \theta_{13}\}$ to obtain (i) and (ii).\\

By using the uniform-in-$n$ lower bound on $\rho_{f^n}$ and $g^n$, together with Lemma \ref{L2.20}, we also estimate the upper bounds for $u_{f^n}$ and $T_{f^n}$ as
\[
|u_{f^n}| + T_{f^n} \leq C\lt( \|u_{f^n}\|_{H^2} + \|T_{f^n}\|_{H^2} \rt) \leq C\|f\|_{H^2_k}(1 + \|f\|_{H^2_k}^{11}) \leq C(1+M)^6.
\]
This completes the proof.
\end{proof}

Now, we are ready to show the uniform bound estimates for $f^{n+1}$ in $H_k^2(\T^3 \times \R^3)$. 
\begin{lemma}\label{f_est_l2}
Suppose that the initial data $(f_0,\rho_0,u_0)$ satisfy the conditions in Theorem \ref{T1.1}. If \eqref{seq_cond} holds, then we can find $0 < T_4 \le T_3$ depending only on $M$ and $N$ such that
\[
\sup_{0 \le t \le T_4} \|f^{n+1}(t)\|_{H_k^2}^2 <M. 
\]
\end{lemma}
\begin{proof}
We split the proof into three cases as follows:

\noindent $\bullet$ (Step A: Zeroth-order estimates) From the kinetic equation in \eqref{app_seq}, we obtain
\begin{align*}
&\frac12\frac{d}{dt}\|f^{n+1}\|_{L_k^2}^2 \cr
&\quad = -\inttr e^{2\langle v \rangle^k}f^{n+1} \lt( v\cdot \nabla f^{n+1} + \nabla_v \cdot (\rho^n(u^n-v)f^{n+1}) - \rho_{f^n}^\alpha(\mm(f^n)-f^{n+1})\rt)\,dxdv\\
&\quad= 3\inttr \rho^n e^{2\langle v\rangle^k} (f^{n+1})^2\,dxdv + \frac12\inttr (f^{n+1})^2 \nabla_v\cdot \lt( e^{2\langle v\rangle^k} \rho^n(u^n-v) \rt)\,dxdv\\
&\qquad + \inttr e^{2\langle v \rangle^k} f^{n+1} \rho_{f^n}^\alpha(\mm(f^n)-f^{n+1})\,dxdv\\
&\quad\le \frac32 \inttr \rho^n e^{2\langle v \rangle^k} (f^{n+1})^2\,dx + k\inttr \rho^n \langle v \rangle^{k-2} e^{\langle v \rangle^k}(f^{n+1})^2 v \cdot (u^n-v)\,dxdv\\
&\qquad +  \inttr e^{2\langle v \rangle^k} f^{n+1} \rho_{f^n}^\alpha\mm(f^n)\,dxdv\\
&\quad\le C(1+M)^{\frac{1}{\gamma-1}}\|f^{n+1}\|_{L_k^2}^2 -\frac k2 \inttr \rho^n \langle v \rangle^{k-2} |v|^2 e^{\langle v \rangle^k}(f^{n+1})^2 \,dxdv\\
&\qquad + \frac k2 \inttr \rho^n |u^n|^2 \langle v \rangle^{k-2} e^{\langle v \rangle^k}(f^{n+1})^2 \,dxdv \\
&\qquad+ C(1+M)^{\alpha}e^{C(1+M)^\frac{6k}{2-k}}\inttr \frac{\rho_{f^n}}{(2\pi T_{f^n})^{3/2}}e^{-\frac{|u_{f^n}-v|^2}{4T_{f^n}}} e^{\langle v\rangle^k} f^{n+1}\,dxdv\\
&\quad\le C(1+M)^{\frac{1}{\gamma-1} + 1} \|f^{n+1}\|_{L_k^2}^2 +  C(1+M)^{\alpha}e^{C(1+M)^\frac{6k}{2-k}}\|\rho_{f^n}\|_{L^2}\|f^{n+1}\|_{L_k^2},
\end{align*}
where we used the estimate \eqref{weight_calc0}, $\la v \ra^{k-2} \leq 1$, and Young's inequality. This implies
\bq\label{l2_est_zero}
\frac{d}{dt}\|f^{n+1}\|_{L_k^2}^2\le C(1+M)^{\frac{1}{\gamma-1}+1+\alpha}e^{C(1+M)^\frac{6k}{2-k}}\lt( \|f^{n+1}\|_{L_k^2}^2 + 1\rt).
\eq

\noindent $\bullet$ (Step B: First-order estimates) For $x$-derivatives, one gets
\begin{align*}
&\frac12\frac{d}{dt}\|\pa_i f^{n+1}\|_{L_k^2}^2 \cr
&\quad = -\inttr \Big( v\cdot \nabla \pa_i f^{n+1} + \rho^n(u^n-v)\cdot \nabla_v \pa_i f^{n+1} + \pa_i \rho^n (u^n -v)\cdot\nabla_v f^{n+1}\\
&\hspace{2cm}+ \rho^n \pa_i u^n \cdot \nabla_v f^{n+1} -3\rho^n \pa_i f^{n+1} -3\pa_i \rho^n f^{n+1} \\
&\hspace{2cm}- \alpha \pa_i \rho_{f^n} \rho_{f^n}^{\alpha-1}(\mm(f^n)-f^{n+1}) - \rho_{f^n}^\alpha \pa_i (\mm(f^n)-f^{n+1})\Big)e^{2\langle v\rangle^k} \pa_i f^{n+1}\,dxdv \\
&\quad \le \frac12 \inttr |\pa_i f^{n+1}|^2 \nabla_v \cdot (\rho^n (u^n-v) e^{2\langle v\rangle^k})\,dxdv \cr
&\qquad + \inttr \pa_i\rho^n (u^n-v) \cdot \nabla_v f^{n+1} \pa_i f^{n+1} e^{2\langle v\rangle^k}\,dxdv\\
&\qquad + C(1+M)^{\frac{1}{\gamma-1} + 1}\|f^{n+1}\|_{H_k^1}^2 + C(1+M)^{\alpha + 3}e^{C(1+M)^{\frac{6k}{2-k}}} (\|f^{n+1}\|_{H_k^1}^2 +  1)\\
&\qquad + C(1+M)^{\alpha+1} \|\pa_i \mm(f^n)\|_{L_k^2}\|\pa_if^{n+1}\|_{L_k^2}\\
&\quad \le -\frac k2 \inttr \rho^n \langle v \rangle^{k-2}|v|^2 e^{2\langle v\rangle^k} |\pa_i f^{n+1}|^2\,dxdv + \frac k2 \inttr \rho^n |u^n|^2 \langle v\rangle^{k-2} e^{2\langle v\rangle^k} |\pa_i f^{n+1}|^2\,dxdv\\
&\qquad + \inttr \pa_i\rho^n v \cdot \nabla_v f^{n+1} \pa_i f^{n+1} e^{2\langle v\rangle^k}\,dxdv + C(1+M)^{\frac{1}{\gamma-1}+17+\alpha}e^{C(1+M)^{\frac{6k}{2-k}}}( \|f^{n+1}\|_{H_k^1}^2 + 1),
\end{align*}
where we used Lemmas \ref{lem_uf}, \ref{L2.20}, and \ref{L2.2}, Young's inequality, and 
\begin{align*}
&\inttr |\pa_i \rho_{f^n}| \rho_{f^n}^{\alpha-1}(\mm(f^n)+ f^{n+1})  e^{2\langle v\rangle^k} |\pa_i f^{n+1}|\,dxdv \cr
&\qquad \le C(1+M)^{\alpha+1}e^{C(1+M)^{\frac{6k}{2-k}}}\inttr |\pa_i \rho_{f^n}|  \frac{\rho_{f^n}}{(2\pi T_{f^n})^{3/2}}e^{-\frac{|u_{f^n}-v|^2}{4T_{f^n}}}  |\pa_i f^{n+1}| e^{\langle v \rangle^k}\,dxdv\\
&\quad \qquad + C\inttr |\pa_i \rho_{f^n}|  |\pa_i f^{n+1}| f^{n+1} e^{2\la v\ra^k}\,dxdv\\
&\qquad \le C(1+M)^{\alpha+1}e^{C(1+M)^{\frac{6k}{2-k}}} \|\pa_i \rho_{f^n}\|_{L^2}\|\rho_{f^n}\|_{L^\infty} \|\pa_i f^{n+1}\|_{L_k^2}\\
&\quad \qquad + C\intt |\pa_i \rho_{f^n}| \lt( \intr e^{2\la v\ra^k} |\pa_i f^{n+1}|^2\,dv\rt)^{1/2}\lt( \intr e^{2\la v\ra^k} |f^{n+1}|^2\,dv\rt)^{1/2}\,dx\\
&\quad \leq C(1+M)^{\alpha + 3}e^{C(1+M)^{\frac{6k}{2-k}}}   \|f^{n+1}\|_{H_k^1} + C\|\pa_i \rho_{f^n}\|_{H^1}\lt\| \lt( \intr e^{2\la v\ra^k} |f^{n+1}|^2\,dv\rt)^{1/2}\rt\|_{H^1}\|f^{n+1}\|_{H_k^1}\cr
&\quad \leq C(1+M)^{\alpha + 3}e^{C(1+M)^{\frac{6k}{2-k}}} (\|f^{n+1}\|_{H_k^1}^2 +  1).
\end{align*}
Note that
\bq\label{est_v_high12}
\begin{aligned}
&\lt|\pa_i \rho^n v\cdot  \nabla_v f^{n+1} \pa_i f^{n+1}) \rt|\mathds{1}_{\{|v| > 1\}}\cr
&\quad \le |\pa_i \log\rho^n| \lt|\rho^n v\cdot  \nabla_v f^{n+1} \pa_i f^{n+1}) \rt|\mathds{1}_{\{|v| > 1\}}\cr
&\quad \leq\frac{1}{L} \lt((\rho^n)^{\frac1k} |v| |\nabla_v  f^{n+1}|^{\frac2k-1}|\pa_i f^{n+1}|\rt)^k  \mathds{1}_{\{|v|>1\}} + C  \lt((\rho^n)^{1-\frac 1k} |\pa_i h^n| |\nabla_v  f^{n+1}|^{2-\frac2k}\rt)^{\frac{k}{k-1}}  \mathds{1}_{\{|v|>1\}} \cr
&\quad \leq \frac1L  \rho^n |v|^k \lt(\frac{|\pa_i   f^{n+1}|^2}{2/k} + \frac{|\nabla_v   f^{n+1}|^2}{2/(2-k)}  \rt) \mathds{1}_{\{|v|>1\}}   + C(1+M)^{\frac{1}{\gamma-1}\cdot \frac{2k-1}{k-1}}|\nabla_v  f^{n+1}|^2
\end{aligned}
\eq
for any $L > 0$.  we can use
\bq\label{est_v_high22}
2^{\frac{k-2}{k}}|v|^k \le \langle v\rangle^{k-2} |v|^2, \quad |v|\ge 1, \quad k \in (1,2)
\eq
 and choose $L= 2^{\frac{2-k}{k}}\cdot 200$ to obtain
 \begin{align*}
& \lt|\pa_i \rho^n v\cdot  \nabla_v f^{n+1} \pa_i f^{n+1}) \rt|\mathds{1}_{\{|v| > 1\}}\cr
&\quad \leq \frac1{200}  \rho^n \langle v\rangle^{k-2} |v|^2 \lt(\frac{|\pa_i   f^{n+1}|^2}{2/k} + \frac{|\nabla_v   f^{n+1}|^2}{2/(2-k)}  \rt)     + C(1+M)^{\frac2{\gamma-1}\cdot \frac{2k-1}{k-1}}|\nabla_v  f^{n+1}|^2.
 \end{align*}
This yields
\begin{align*}
&\lt| \pa_i \rho^n v \cdot (\nabla_v f^{n+1})( \pa_i f^{n+1})\rt| \cr
&\quad \le  C(1+M)^{\frac{1}{\gamma-1}}\lt| (\nabla_v f^{n+1} )(\pa_i f^{n+1}) \rt|\mathds{1}_{\{|v|\le 1\}} + C(1+M)^{\frac{1}{\gamma-1}\cdot \frac{2k-1}{k-1}}|\nabla_v f^{n+1}|^2\mathds{1}_{\{|v|> 1\}}\\
&\qquad +\frac {1}{200} \rho^n \langle v\rangle^{k-2}|v|^2\lt( \frac k2 |\pa_i f^{n+1}|^2 + \frac{2-k}{2} |\nabla_v f^{n+1}|^2\rt) \mathds{1}_{\{|v|> 1\}}.
\end{align*}
Thus, we obtain
\bq\label{l2_est_x1}
\begin{aligned}
\frac{d}{dt}\|\pa_i f^{n+1}\|_{L_k^2}^2 &\le \frac{2-k}{200}\inttr \rho^n \langle v\rangle^{k-2} |v|^2 e^{2\langle v \rangle^k}|\nabla_v f^{n+1}|^2\,dxdv\\
&\quad + C(1+M)^{\frac{1}{\gamma-1}\cdot\frac{2k-1}{k-1}+13+\alpha}e^{C(1+M)^{\frac{6k}{2-k}}}( \|f^{n+1}\|_{H_k^1}^2 + 1),
\end{aligned}
\eq
where $C>0$ is independent of $n$ and $T$. For $v$-derivatives,
\begin{align*}
&\frac12\frac{d}{dt}\|\pa_{v_i}f^{n+1}\|_{L_k^2}^2 \cr
&\quad = -\inttr \bigg( v\cdot \nabla \pa_{v_i} f^{n+1} +\pa_i f^{n+1}+ \rho^n (u^n-v)\cdot \nabla_v \pa_{v_i} f^{n+1} \\
&\hspace{2.5cm} - 4\rho^n \pa_{v_i}f^{n+1} - \rho_{f^n}^\alpha \pa_{v_i}(\mm(f^n) - f^{n+1})\bigg) e^{2\langle v \rangle^k} \pa_{v_i}f^{n+1}\,dxdv\\
&\quad \le -\frac k2 \inttr \rho^n \langle v\rangle^{k-2}|v|^2 e^{2\langle v\rangle^k} |\pa_{v_i}f^{n+1}|^2\,dxdv + \frac k2 \inttr \rho^n |u^n|^2 \langle v\rangle^{k-2} e^{2\langle v\rangle^k} |\pa_{v_i}f^{n+1}|^2\,dxdv\\
&\qquad +  C(1+M)^{\frac{1}{\gamma-1}}\|f^{n+1}\|_{H_k^1}^2 + C(1+M)^{\alpha} \|\pa_{v_i}\mm(f^n)\|_{L_k^2} \|\pa_{v_i}f^{n+1}\|_{L_k^2}\\
&\quad \le  -\frac k2 \inttr \rho^n \langle v\rangle^{k-2}|v|^2 e^{2\langle v\rangle^k} |\pa_{v_i}f^{n+1}|^2\,dxdv\\
&\qquad +   C(1+M)^{\frac{1}{\gamma-1}+1}\|f^{n+1}\|_{H_k^1}^2 + C(1+M)^{18+\alpha} e^{C(1+M)^{\frac{6k}{2-k}}}\|f^{n+1}\|_{H_k^1},
\end{align*}
where we used Lemma \ref{L2.2}. In other words,
\bq\label{l2_est_v1}
\begin{aligned}
\frac{d}{dt}\|\pa_{v_i}f^{n+1}\|_{L_k^2}^2 &\le   -k \inttr \rho^n \langle v\rangle^{k-2}|v|^2 e^{2\langle v\rangle^k} |\pa_{v_i}f^{n+1}|^2\,dxdv\\
&\quad + C(1+M)^{\frac{1}{\gamma-1}+18+\alpha} e^{C(1+M)^{\frac{6k}{2-k}}}(\|f^{n+1}\|_{H_k^1}^2+1).
\end{aligned}
\eq
So we combine \eqref{l2_est_x1} with \eqref{l2_est_v1} to yield
\bq\label{l2_est_first}
\begin{aligned}
\frac{d}{dt}\|f^{n+1}\|_{\dot{H}_k^1}^2 &\le  -\frac{9k}{10} \inttr \rho^n \langle v\rangle^{k-2}|v|^2 e^{2\langle v\rangle^k} |\nabla_v f^{n+1}|^2\,dxdv\\
&\quad + C(1+M)^{\frac{1}{\gamma-1}\cdot \frac{2k-1}{k-1}+18+\alpha} e^{C(1+M)^{\frac{6k}{2-k}}}(\|f^{n+1}\|_{H_k^1}^2+1).
\end{aligned}
\eq

\noindent $\bullet$ (Step C: Second-order estimates): We first deal with
\begin{align*}
&\frac12\frac{d}{dt}\|\pa_{ij} f^{n+1}\|_{L_k^2}^2\cr
&\quad = -\inttr \bigg( v \cdot \nabla \pa_{ij} f^{n+1} + \rho^n (u^n-v)\cdot \nabla_v \pa_{ij} f^{n+1} + \pa_i \rho^n (u^n-v)\cdot\nabla_v \pa_j f^{n+1}\\
&\hspace{3cm} + \pa_j \rho^n (u^n-v)\cdot\nabla_v \pa_if^{n+1} + \rho^n (\pa_i u^n \cdot\nabla_v \pa_j f^{n+1}+\pa_j u^n \cdot \nabla_v \pa_i f^{n+1})\\
&\hspace{3cm} + \pa_{ij}\rho^n (u^n-v)\cdot\nabla_v f^{n+1} + (\pa_i\rho^n \pa_j u^n + \pa_j \rho^n \pa_i u^n)\cdot\nabla_v f^{n+1}\\
&\hspace{3cm} -3\pa_{ij}\rho^n f^{n+1} -3(\pa_i \rho^n \pa_j f^{n+1} + \pa_j \rho^n \pa_i f^{n+1}) - \alpha\pa_{ij}\rho_{f^n} \rho_{f^n}^{\alpha-1}(\mm(f^n) - f^{n+1})\\
&\hspace{3cm} -\alpha(\alpha-1)\pa_i \rho_{f^n} \pa_j \rho_{f^n} \rho_{f^n}^{\alpha-2}(\mm(f^n)-f^{n+1})\\
&\hspace{3cm} -\alpha \pa_i \rho_{f^n} \rho_{f^n}^{\alpha-1}\pa_j (\mm(f^n)-f^{n+1}) -\alpha \pa_j \rho_{f^n} \rho_{f^n}^{\alpha-1}\pa_i (\mm(f^n)-f^{n+1})\\
&\hspace{3cm} -\rho_{f^n}^\alpha \pa_{ij}(\mm(f^n)-f^{n+1})\bigg) e^{2\langle v \rangle^k} \pa_{ij}f^{n+1}\,dxdv\\
&\quad \le \frac12\inttr |\pa_{ij}f^{n+1}|^2 \nabla_v \cdot (\rho^n(u^n-v)e^{2\langle v\rangle^k})\,dxdv + C(1+M)^{\frac{1}{\gamma-1}+\frac 12} \|f^{n+1}\|_{H_k^2}^2\\
&\qquad + \inttr (\pa_i \rho^n \nabla_v \pa_j f^{n+1} + \pa_j \rho^n \nabla_v \pa_i f^{n+1})\cdot v  \pa_{ij}f^{n+1}e^{2\langle v \rangle^k}\,dxdv\\
&\qquad -\inttr \pa_{ij}\rho^n(u^n-v)\cdot\nabla_v f^{n+1} \pa_{ij}f^{n+1} e^{2\langle v \rangle^k}\,dxdv +\inttr 3\pa_{ij}\rho^n f^{n+1}\pa_{ij}f^{n+1} e^{2\langle v\rangle^k}\,dxdv \\
&\qquad+ \alpha\inttr \pa_{ij}\rho_{f^n}\rho_{f^n}^{\alpha-1} (\mm(f^n)-f^{n+1})\pa_{ij}f^{n+1}e^{2\langle v \rangle^k}\,dxdv\\
&\qquad + C(1+M)^{\alpha+10}e^{C(1+M)^{\frac{6k}{k-2}}}\lt( \|f^{n+1}\|_{H_k^2}^2 + 1\rt) +  C(1+M)^{\alpha+1} (\|\pa_{ij}\mm(f^n)\|_{L_k^2} + \|f^{n+1}\|_{H_k^2})\|f^{n+1}\|_{H_k^2}\\
&\quad\le C(1+M)^{\frac{1}{\gamma-1}+19+\alpha}e^{C(1+M)^{\frac{6k}{k-2}}}\lt( \|f^{n+1}\|_{H_k^2}^2 + 1\rt)\cr
&\qquad  -\frac k2 \inttr \rho^n \langle v\rangle^{k-2}|v|^2 e^{2\langle v\rangle^k} |\pa_{ij}f^{n+1}|^2\,dxdv\\
&\qquad+ \inttr (\pa_i \rho^n \nabla_v \pa_j f^{n+1} + \pa_j \rho^n \nabla_v \pa_i f^{n+1})\cdot  v  \pa_{ij}f^{n+1}e^{2\langle v \rangle^k}\,dxdv\\
&\qquad -\inttr \pa_{ij}\rho^n(u^n-v)\cdot\nabla_v f^{n+1} \pa_{ij}f^{n+1} e^{2\langle v \rangle^k}\,dxdv  +\inttr 3\pa_{ij}\rho^n f^{n+1}\pa_{ij}f^{n+1} e^{2\langle v\rangle^k}\,dxdv \\
&\qquad+ \alpha\inttr (\pa_{ij}\rho_{f^n})\rho_{f^n}^{\alpha-1} (\mm(f^n)-f^{n+1})\pa_{ij}f^{n+1}e^{2\langle v \rangle^k}\,dxdv\\
&\quad =:  C(1+M)^{\frac{1}{\gamma-1}+19+\alpha}e^{C(1+M)^{\frac{6k}{k-2}}}\lt( \|f^{n+1}\|_{H_k^2}^2 + 1\rt)   -\frac k2 \inttr \rho^n \langle v\rangle^{k-2}|v|^2 e^{2\langle v\rangle^k} |\pa_{ij}f^{n+1}|^2\,dxdv\\
&\qquad + \sum_{i=1}^4\sfI_i.
\end{align*}
Here, similarly as before, due to Lemmas \ref{L2.20} and \ref{macro_bd}, we estimated
\begin{align*}
&\inttr |\pa_i \rho_{f^n}| \rho_{f^n}^{\alpha-1}(|\pa_j \mm(f^n)| + |\pa_j f^{n+1}|)  e^{2\langle v \rangle^k} |\pa_{ij}f^{n+1}|\,dxdv \cr
&\quad \leq C(1+M)^{\alpha+1} \|\pa_i \rho_{f^n}\|_{H^1}\|\pa_{ij}f^{n+1}\|_{L^2_k} \cr
&\hspace{3cm} \times \lt(\lt\| \lt( \intr |\pa_j \mm(f^n)|^2 e^{2\langle v \rangle^k}\,dv\rt)^{1/2}\rt\|_{H^1} + \lt\|\lt( \intr |\pa_j f^{n+1}|^2 e^{2\langle v \rangle^k}\,dv\rt)^{1/2}\rt\|_{H^1}   \rt)\cr
&\quad \leq C(1+M)^{\alpha+2}e^{C(1+M)^{\frac{6k}{k-2}}}\|f^{n+1}\|_{H^2_k}\lt( \|\mm (f^n)\|_{H_k^2}+ \|f^{n+1}\|_{H^2_k}\rt)\cr
&\quad \leq C(1+M)^{\alpha+19}e^{C(1+M)^{\frac{6k}{k-2}}}\lt( \|f^{n+1}\|_{H_k^2}^2 + 1\rt).
\end{align*}
We next estimate the terms $\sfI_i,i=1,2,3,4$ as follows.

For $\sfI_1$, we use \eqref{est_v_high12} and \eqref{est_v_high22} to get
\begin{align*}
&\lt| \Big(\pa_i \rho^n \nabla_v \pa_j f^{n+1} + \pa_j \rho^n \nabla_v \pa_i f^{n+1}\Big)\cdot v \pa_{ij}f^{n+1}\rt|\\
&\quad \le C(1+M)^{\frac{1}{\gamma-1}}\lt(|\nabla_v \pa_i f^{n+1} \pa_{ij}f^{n+1}|\mathds{1}_{\{|v|\le 1\}} + |\nabla_v \pa_j f^{n+1} \pa_{ij}f^{n+1}|\mathds{1}_{\{|v|\le 1\}} \rt)\\
&\qquad +C(1+M)^{\frac{1}{\gamma-1} \cdot \frac{2k-1}{k-1}}\lt(|\nabla_v \pa_i f^{n+1} \pa_{ij}f^{n+1}|\mathds{1}_{\{|v|\le 1\}} + |\nabla_v \pa_j f^{n+1} \pa_{ij}f^{n+1}|\mathds{1}_{\{|v|> 1\}} \rt)\\
&\qquad + \frac{1}{200} \rho^n \langle v \rangle^{k-2}|v|^2 \lt( k|\pa_{ij} f^{n+1}|^2 + \frac{2-k}{2}(|\nabla_v \pa_i f^{n+1}|^2 +|\nabla_v \pa_j f^{n+1}|^2)\rt)\mathds{1}_{\{|v|>1\}}.
\end{align*}
This implies
\begin{align*}
\sfI_1&\le C(1+M)^{\frac{1}{\gamma-1} \cdot \frac{2k-1}{k-1}}\|f^{n+1}\|_{H_k^2}^2 + \frac {k}{200} \inttr \rho^n \langle v\rangle^{k-2}|v|^2 e^{2\langle v\rangle^k} |\pa_{ij} f^{n+1}|^2\,dxdv\\
&\quad + \frac{2-k}{400}\inttr \rho^n \langle v \rangle^{k-2}|v|^2 e^{2\langle v \rangle^k} \lt( |\nabla_v \pa_i f^{n+1}|^2 + |\nabla_v \pa_j f^{n+1}|^2\rt)\,dxdv.
\end{align*}
For $\sfI_2$, 
\begin{align*}
\sfI_2&=- \inttr \pa_{ij}\rho^n u^n \cdot \nabla_v f^{n+1} \pa_{ij}f^{n+1} e^{2\langle v\rangle^k}\,dxdv\\
&\quad +  \inttr \pa_{ij}\rho^n v \cdot \nabla_v f^{n+1} \pa_{ij}f^{n+1} e^{2\langle v\rangle^k}\,dxdv\\
&=: \sfI_{21} + \sfI_{22}.
\end{align*}
Here, we can handle $\sfI_{21}$ as
\begin{align*}
\sfI_{21}&\le \|\pa_{ij}\rho^n\|_{L^6}\|u^n\|_{L^\infty} \|(\nabla_v f^{n+1})e^{\langle v\rangle^k}\|_{H^1_x(L^2_v)} \|\pa_{ij}f^{n+1}\|_{L_k^2}\\
&\le C(1+M)^{\frac{1}{\gamma-1}+1} \|f^{n+1}\|_{H_k^2}^2,
\end{align*}
due to Lemma \ref{lem_uf}.
 
For $\sfI_{22}$, we use H\"older inequality, Young's inequality, \eqref{est_v_high22}, and Lemma \ref{lem_uf} to get
\begin{align*}
\sfI_{22}&\le \inttr |\pa_{ij}\rho^n| |v|\mathds{1}_{\{|v|\le 1\}} |\pa_{ij} f^{n+1} \nabla_v f^{n+1}| e^{2\langle v \rangle^k}\,dxdv\\
&\quad + \intt |\pa_{ij}\rho^n| \lt(\intr |v|^k \mathds{1}_{\{|v|>1\}}e^{2\langle v \rangle^k} |\pa_{ij} f^{n+1}|^k |\nabla_v f^{n+1}|^{2-k}\,dv \rt)^{\frac 1k}\cr
&\hspace{7cm} \times \lt(\intr e^{2\langle v \rangle^k} |\nabla_v f^{n+1}|^2\,dv\rt)^{\frac{k-1}{k}}\,dx\\
&\le \|\pa_{ij} \rho^n\|_{L^6}\|\pa_{ij}f^{n+1}\|_{L_k^2} \lt\| \lt( \intr e^{2\langle v \rangle^k} |\nabla_v f^{n+1}|^2\,dv\rt)^{1/2}\rt\|_{L^3} \\
&\quad +C\intt |\pa_{ij}\rho^n| \lt(\intr \langle v\rangle^{k-2}|v|^2 e^{2\langle v\rangle^k} |\pa_{ij} f^{n+1}|^2\,dv\rt)^{\frac 12} \lt(\intr \langle v\rangle^{k-2}|v|^2 e^{2\langle v \rangle^k}|\nabla_v f^{n+1}|^2\,dv \rt)^{\frac{2-k}{2k}}\\
&\hspace{8cm}\times\lt(\intr e^{2\langle v \rangle^k}|\nabla_v f^{n+1}|^2\,dv\rt)^{\frac{k-1}{k}}\,dx\\
&\le C(1+M)^{\frac{1}{\gamma-1}}\|f^{n+1}\|_{H_k^2}^2\\ 
&\quad + C\|\pa_{ij}\rho^n\|_{L^6}\lt\|\lt(\intr \langle v \rangle^{k-2}|v|^2 e^{2\langle v \rangle^k} |\pa_{ij} f^{n+1}|^2\,dv\rt)^{1/2}\rt\|_{L^2}\\
&\hspace{1cm}\times \lt\|\lt( \intr \langle v \rangle ^{k-2} |v|^2e^{2\langle v\rangle^k} |\nabla_v f^{n+1}|^2\,dv\rt)^{\frac{2-k}{2k}} \rt\|_{L^{\frac{3k}{2-k}}} \lt\|\lt(\intr e^{2\langle v\rangle^k}|\nabla_v f^{n+1}|^2\,dv \rt)^{\frac{k-1}{k}} \rt\|_{L^{\frac{3k}{2k-2}}}\\
&\le C(1+M)^{\frac{1}{\gamma-1}}\|f^{n+1}\|_{H_k^2}^2\\
&\quad + C(1+M)^{\frac{1}{\gamma-1}}\lt\|\lt(\intr \langle v \rangle^{k-2}|v|^2 e^{2\langle v \rangle^k} |\pa_{ij} f^{n+1}|^2\,dv\rt)^{1/2}\rt\|_{L^2}\\
&\hspace{1cm}\times \lt\|\lt( \intr  \langle v \rangle ^{k-2} |v|^2e^{2\langle v\rangle^k} |\nabla_v f^{n+1}|^2\,dv\rt)^{\frac 12} \rt\|_{L^3}^{\frac{2-k}{k}}\lt\|\lt(\intr e^{2\langle v\rangle^k}|\nabla_v f^{n+1}|^2\,dv \rt)^{\frac 12} \rt\|_{L^3}^{\frac{2k-2}{k}}\\
&\le C(1+M)^{\frac{1}{\gamma-1}\cdot\frac{2k-1}{k-1}}\|f^{n+1}\|_{H_k^2}^2 \\
&\quad + \frac{k}{200} \inttr \rho^n \langle v \rangle^{k-2} |v|^2 e^{2\langle v \rangle^k} \lt( |\pa_{ij}f^{n+1}|^2 + |\nabla_v f^{n+1}|^2 + |\nabla_x \nabla_v f^{n+1}|^2\rt)\,dxdv.
\end{align*}
Thus, we have
\begin{align*}
&\sfI_2\le C(1+M)^{\frac{1}{\gamma-1}\cdot\frac{2k-1}{k-1}}\|f^{n+1}\|_{H_k^2}^2 \cr
&\quad + \frac{k}{200} \inttr \rho^n \langle v \rangle^{k-2} |v|^2 e^{2\langle v \rangle^k} \lt( |\pa_{ij}f^{n+1}|^2 + |\nabla_v f^{n+1}|^2 + |\nabla_x \nabla_v f^{n+1}|^2\rt)\,dxdv.
\end{align*}
For $\sfI_3$, similarly as before, we obtain
\begin{align*}
\sfI_3&\le 3\|\pa_{ij}\rho^n\|_{L^6}\|\pa_{ij}f^{n+1}\|_{L_k^2} \lt\|\lt( \intr |f^{n+1}|^2 e^{2\langle v\rangle^k}\,dv\rt)^{1/2}\rt\|_{H^1}\\
&\le C(1+M)^{\frac{1}{\gamma-1}}\|\pa_{ij}f^{n+1}\|_{L_k^2}\|f^{n+1}\|_{H_k^1}\cr
& \le C(1+M)^{\frac{1}{\gamma-1}} \|f^{n+1}\|_{H_k^2}^2.
\end{align*}
For $\sfI_4$, we use \eqref{weight_calc0} and Lemma \ref{macro_bd} to yield
\begin{align*}
\sfI_4 &\le C(1+M)^{\alpha+1}e^{C(1+M)^{\frac{6k}{2-k}}}\inttr |\pa_{ij}\rho_{f^n}|  \frac{\rho_{f^n}}{(2\pi T_{f^n})^{3/2}}e^{-\frac{|u_{f^n}-v|^2}{4T_{f^n}}}  |\pa_{ij}f^{n+1}| e^{\langle v \rangle^k}\,dxdv\\
&\quad + C\inttr |\pa_{ij}\rho_{f^n}|  |\pa_{ij}f^{n+1}| f^{n+1} e^{2\la v\ra^k}\,dxdv\\
&\le C(1+M)^{\alpha+1}e^{C(1+M)^{\frac{6k}{2-k}}} \|\pa_{ij}\rho_{f^n}\|_{L^2}\|\rho_{f^n}\|_{L^\infty} \|\pa_{ij}f^{n+1}\|_{L_k^2}\\
&\quad + C\intt |\pa_{ij}\rho_{f^n}| \lt( \intr e^{2\la v\ra^k} |\pa_{ij}f^{n+1}|^2\,dv\rt)^{1/2}\lt( \intr e^{2\la v\ra^k} |f^{n+1}|^2\,dv\rt)^{1/2}\,dx\\
&=:\sfI_4^1  + \sfI_4^2,
\end{align*}
where
\[
\sfI_4^1 \leq C(1+M)^{\alpha+3}e^{C(1+M)^{\frac{6k}{2-k}}} \|\pa_{ij}f^{n+1}\|_{L_k^2}.
\]
For $\sfI_4^2$, analogously as in the proof of Lemma \ref{macro_bd}, we have
\begin{align*}
\sfI_4^2 &\leq C\|\bar g^{n+1}\|_{L^\infty}\|\pa_{ij}\rho_{f^n}\|_{L^2}\|\pa_{ij}f^{n+1}\|_{L^2_k} \cr
&\leq C(1+M)\|\bar g^{n+1}\|_{H^2}\|\pa_{ij}f^{n+1}\|_{L^2_k} \cr
&\leq C(1+M)\lt(\|f^{n+1}\|_{H^2_k} +1\rt) \|f^{n+1}\|_{H^2_k}^2,
\end{align*}
due to Lemma \ref{lem_uf}, where
\[
\bar g^{n+1}(x):= \lt( \intr e^{2\la v\ra^k} |f^{n+1}|^2\,dv\rt)^{1/2}.
\]
This implies
\[
\sfI_4 \leq C(1+M)^{\alpha+3}e^{C(1+M)^{\frac{6k}{2-k}}} \lt( \|f^{n+1}\|_{H_k^2}^3 + \|f^{n+1}\|_{H_k^2}^2 + 1\rt).
\]
Thus, we gather the estimates for $\sfI_i$'s to get
\bq\label{l2_est_x2}
\begin{aligned}
\frac{d}{dt}\|\nabla_x^2 f^{n+1}\|_{L_k^2}^2 &\le C(1+M)^{\frac{1}{\gamma-1}\cdot\frac{2k-1}{k-1}+19+\alpha}e^{C(1+M)^{\frac{2k}{k-2}}}\lt( \|f^{n+1}\|_{H_k^2}^3 + \|f^{n+1}\|_{H_k^2}^2 + 1\rt)\\
&\quad - \frac{4k}{5}\inttr \rho^n \la v\ra^{k-2}|v|^2 e^{2\la v \ra^k}|\nabla_x f^{n+1}|^2\,dxdv\\
&\quad + \frac{k}{5}\inttr \rho^n \la v\ra^{k-2} e^{2\la v \ra^k}\lt( |\nabla_v f^{n+1}|^2 + |\nabla_x \nabla_v f^{n+1}|^2\rt)dxdv.
\end{aligned}
\eq
Next, we consider
\begin{align*}
&\frac12\frac{d}{dt}\|\pa_i\pa_{v_j}f^{n+1}\|_{L_k^2}^2 \cr
&\quad = -\inttr \bigg( v \cdot \nabla \pa_i\pa_{v_j}f^{n+1}  + \pa_{ij}f^{n+1} + \rho^n(u^n-v)\cdot\nabla_v \pa_i \pa_{v_j} f^{n+1}+ \pa_i\rho^n(u^n-v)\nabla_v \pa_{v_j}f^{n+1}\\
&\quad \hspace{2cm} + \rho^n \pa_i u^n \cdot \nabla_v \pa_{v_j}f^{n+1} - 4\pa_i (\rho^n \pa_{v_j}f^{n+1})- \alpha \pa_i\rho_{f^n}\rho_{f^n}^{\alpha-1}\pa_{v_j}(\mm(f^n)-f^{n+1})\\
&\quad\hspace{5cm} -\rho_{f^n}^\alpha\pa_i\pa_{v_j}(\mm(f^n)-f^{n+1})\bigg)e^{2\la v\ra^k}\pa_i \pa_{v_j}f^{n+1}\,dxdv\\
&\quad= \frac12 \inttr |\pa_i \pa_{v_j}f^{n+1}|^2 \nabla_v \cdot \lt(\rho^n(u^n-v)e^{2 \la v \ra^k}\rt)\,dxdv  -\inttr \pa_{ij}f^{n+1} \pa_i \pa_{v_j}f^{n+1} e^{2\langle v \rangle^k}\,dxdv\\
&\qquad- \inttr \pa_i \rho^n (u^n-v)\cdot\nabla_v \pa_{v_j}f^{n+1}\pa_i\pa_{v_j}f^{n+1}e^{2\la v \ra^k}\,dxdv\\
&\qquad +\inttr \bigg(-\rho^n \pa_i u^n \cdot \nabla_v \pa_{v_j}f^{n+1} +4\pa_i (\rho^n \pa_{v_j}f^{n+1})+ \alpha \pa_i\rho_{f^n}\rho_{f^n}^{\alpha-1}\pa_{v_j}(\mm(f^n)-f^{n+1})\\
&\quad\hspace{5cm} +\rho_{f^n}^\alpha\pa_i\pa_{v_j}(\mm(f^n)-f^{n+1})\bigg)e^{2\la v\ra^k}\pa_i \pa_{v_j}f^{n+1}\,dxdv\\
&\quad\le -\frac k2\inttr \rho^n \la v \ra^{k-2}|v|^2e^{2\la v \ra^k}|\pa_i\pa_{v_j}f^{n+1}|^2\,dxdv \\
&\qquad + \inttr \pa_i \rho^n v\mathds{1}_{\{|v|>1\}}\cdot(\nabla_v \pa_{v_j}f^{n+1})(\pa_i\pa_{v_j}f^{n+1})e^{2\la v \ra^k}\,dxdv\\
&\qquad +  C(1+M)^{\frac{1}{\gamma-1}+2}\|f^{n+1}\|_{H_k^2}^2+ C(1+M)^{\alpha+2}e^{C(1+M)^{\frac{6k}{k-2}}}\lt( \|f^{n+1}\|_{H_k^2}^2 + 1\rt)\cr
&\qquad + C(1+M)^{2+\alpha}(\|\pa_i\pa_{v_j}\mm(f^n)\|_{L_k^2} + \|f^{n+1}\|_{H_k^2})\|\pa_i\pa_{v_j}f^{n+1}\|_{L_k^2}\\
&\quad\le -\frac {2k}{5}\inttr \rho^n \la v \ra^{k-2}|v|^2e^{2\la v \ra^k}|\pa_i\pa_{v_j}f^{n+1}|^2\,dxdv \\
&\qquad + \frac{2-k}{200}\inttr \rho^n \la v\ra^{k-2}|v|^2 e^{2\la v \ra^k} |\nabla_v \pa_{v_j}f^{n+1}|^2\,dxdv\\
&\qquad + C(1+M)^{\frac{1}{\gamma-1} \cdot \frac{2k-1}{k-1}+20+\alpha}e^{C(1+M)^{\frac{6k}{2-k}}}\lt( \|f^{n+1}\|_{H_k^2}^2 + 1\rt),
\end{align*}
where we used
 \begin{align*}
&\lt|\inttr \pa_i \rho^n v\mathds{1}_{\{|v|>1\}}\cdot\nabla_v \pa_{v_j}f^{n+1}\pa_i\pa_{v_j}f^{n+1}e^{2\la v \ra^k}\,dxdv\rt|\cr
&\quad \leq  \frac{k}{400}\inttr \rho^n \la v \ra^{k-2}|v|^2e^{2\la v \ra^k}|\pa_i\pa_{v_j}f^{n+1}|^2\,dxdv\cr
&\qquad + \frac{2-k}{200}\inttr \rho^n \la v\ra^{k-2}|v|^2 e^{2\la v \ra^k} |\nabla_v \pa_{v_j}f^{n+1}|^2\,dxdv\\
&\qquad + C(1+M)^{\frac{1}{\gamma-1} \cdot\frac{2k-1}{k-1}}\lt( \|f^{n+1}\|_{H_k^2}^2 + 1\rt)
 \end{align*}
due to \eqref{est_v_high12} and \eqref{est_v_high22} and
\begin{align*}
&\inttr |\pa_i\rho_{f^n}|\rho_{f^n}^{\alpha-1}(|\pa_{v_j}\mm(f^n)| + |\pa_{v_j}f^{n+1}|)e^{2\la v\ra^k}|\pa_i \pa_{v_j}f^{n+1}|\,dxdv\cr
&\quad \leq C(1+M)^{\alpha+1} \|\pa_i \rho_{f^n}\|_{H^1}\|\pa_i \pa_{v_j}f^{n+1}\|_{L^2_k}\lt(\| \pa_{v_j} \mm(f^n)e^{\langle v \rangle^k}\|_{H^1_x(L^2_v)} + \|\pa_{v_j} f^{n+1}e^{\langle v \rangle^k}\|_{H^1_x(L^2_v)}   \rt)\cr
&\quad \leq C(1+M)^{\alpha+2}e^{C(1+M)^{\frac{6k}{k-2}}}\|f^{n+1}\|_{H^2_k}\lt(\| \mm(f^n)\|_{H_k^2} +  \|f^{n+1}\|_{H^2_k}\rt)\cr
&\quad \leq C(1+M)^{\alpha+19}e^{C(1+M)^{\frac{6k}{k-2}}}\lt( \|f^{n+1}\|_{H_k^2}^2 + 1\rt).
\end{align*}
Thus we can get
\bq\label{l2_est_xv}
\begin{aligned}
\frac{d}{dt}\|\nabla_x \nabla_v f^{n+1}\|_{L_k^2}^2 &\le C(1+M)^{\frac{1}{\gamma-1} \cdot \frac{2k-1}{k-1}+20+\alpha} e^{C(1+M)^{\frac{6k}{2-k}}} (\|f^{n+1}\|_{H_k^2}^2 + 1)\\
&\quad-\frac{4k}{5}\inttr \rho^n \la v\ra^{k-2}|v|^2 e^{2\la v \ra^k} |\nabla_x \nabla_v f^{n+1}|^2\,dxdv\\
&\quad + \frac{k}{10}\inttr  \rho^n \la v\ra^{k-2}|v|^2 e^{2\la v \ra^k} | \nabla_v^2 f^{n+1}|^2\,dxdv.
\end{aligned}
\eq
Finally, we estimate
\begin{align*}
&\frac12\frac{d}{dt}\|\pa_{v_iv_j}f^{n+1}\|_{L_k^2}^2 \cr
&\quad = -\inttr \bigg( v\cdot\nabla \pa_{v_iv_j}f^{n+1} + \pa_i\pa_{v_j}f^{n+1} + \pa_j \pa_{v_i}f^{n+1} + \rho^n(u^n-v)\cdot\nabla_v \pa_{v_iv_j}f^{n+1}\\
&\qquad \hspace{2cm} -5\rho^n \pa_{v_i v_j}f^{n+1} - \rho_{f^n}^\alpha\pa_{v_i v_j}\lt( \mm(f^n)-f^{n+1}\rt)\bigg)e^{2\la v \ra^k}\pa_{v_i v_j}f^{n+1}\,dxdv\\
&\quad \le \frac12 \inttr |\pa_{v_i v_j}f^{n+1}|^2 \nabla_v \cdot \lt(\rho^n(u^n-v)e^{2\la v \ra^k}\rt) \,dxdv + C(1+M)^{\frac{1}{\gamma-1}}\|f^{n+1}\|_{H_k^2}^2\\
&\qquad + C(1+M)^{\alpha} \|\pa_{v_i v_j}\mm(f^n)\|_{L_k^2}\|\pa_{v_i v_j}f^{n+1}\|_{L_k^2}\\
&\quad \le -\frac k2 \inttr \rho^n \la v \ra^{k-2}|v|^2 e^{2\la v \ra^k}|\pa_{v_i v_j}f^{n+1}|^2\,dxdv + C(1+M)^{\frac{2}{\gamma-1}+18+\alpha}e^{C(1+M)^{\frac{6k}{2-k}}}(\|f^{n+1}\|_{H_k^2}^2 + 1),
\end{align*}
and this gives
\bq\label{l2_est_v2}
\begin{aligned}
\frac{d}{dt}\|\nabla_v^2 f^{n+1}\|_{L_k^2}^2 &\le -k \inttr \rho^n \la v \ra^{k-2}|v|^2 e^{2\la v \ra^k}|\nabla_v^2 f^{n+1}|^2\,dxdv\\
&\quad +  C(1+M)^{\frac{1}{\gamma-1}+18+\alpha}e^{C(1+M)^{\frac{6k}{2-k}}}(\|f^{n+1}\|_{H_k^2}^2 + 1).
\end{aligned}
\eq
Thus, we gather all the estimates \eqref{l2_est_zero}, \eqref{l2_est_first}, \eqref{l2_est_x2}, \eqref{l2_est_xv}, and \eqref{l2_est_v2} to yield
\[
\frac{d}{dt}\|f^{n+1}\|_{H_k^2}^2 \le C(1+M)^{\frac{1}{\gamma-1}\cdot\frac{2k-1}{k-1}+20+\alpha}e^{C(1+M)^{\frac{6k}{2-k}}}(\|f^{n+1}\|_{H_k^2}^3 + \|f^{n+1}\|_{H_k^2}^2 + 1).
\]
From the above, we can find
\[
\|f^{n+1}(\cdot,\cdot,t)\|_{H_k^2}^2 \leq \frac{\|f_0\|_{H_k^2}^2 + 1}{\lt(1 - \frac12 C(1+M)^{\frac{1}{\gamma-1}\cdot\frac{2k-1}{k-1}+20+\alpha}e^{C(1+M)^{\frac{6k}{2-k}}}t\rt)^2} - 1.
\]
Since the right hand side of the above decays to $\|f_0\|_{H_k^2}^2$ as $t \to 0$ and $\|f_0\|_{H_k^2}^2 < N < M$, we can find a sufficiently small $0<T_4\le T_3$ to obtain the desired result.

\end{proof}

\begin{proof}[Proof of Proposition \ref{prop_seq}] We now choose $T^* := T_4$, then by strong induction, this directly concludes the uniform-in-$n$ estimates of approximations in the desired solution space.
\end{proof}

%
%
%
%
%
%
%
\section{Proof of Theorem \ref{T1.1}}\label{sec:4}

\subsection{Cauchy estimates}
To obtain the unique regular solution to \eqref{main_sys2}, we investigate the following Cauchy estimates.

\begin{lemma}\label{f_cauchy}
Suppose that the initial data $(f_0,\rho_0,u_0)$ satisfy the conditions in Theorem \ref{T1.1}. Then for any $\epsilon > 0$ with $\epsilon \in (0,k)$ we have
\[
\frac{d}{dt}\|f^{n+1}-f^n\|_{L_{k - \epsilon}^2}^2\le C\lt( \|f^{n+1}-f^n\|_{L_{k - \epsilon}^2}^2 + \|h^n -h^{n-1}\|_{H^1}^2 + \|u^n-u^{n-1}\|_{H^1}^2 + \|f^n -f^{n-1}\|_{L_{k - \epsilon}^2}^2\rt)
\]
for $ t \in [0,T^*]$.
\end{lemma}
\begin{remark} Instead of employing $L_{k - \epsilon}^2(\T^3 \times \R^3)$ space, the above estimate can be also obtained in $L^2_{k,q}(\T^3 \times \R^3)$ with $q < 1$ defined by
\[
L^2_{k,q}(\T^3 \times \R^3) := \lt\{ f:\inttr  e^{2q \la v\ra^k} |f|^2\,dxdv < \infty \rt\}.
\]
\end{remark}
\begin{proof}
Direct computation gives
\begin{align*}
&\frac12\frac{d}{dt}\|f^{n+1}-f^n\|_{L_{k - \epsilon}^2}^2 \cr
&\quad = -\inttr \Big( v\cdot \nabla(f^{n+1} - f^n) + \nabla_v\cdot((\rho^n-\rho^{n-1})(u^n-v)f^{n+1}) + \nabla_v \cdot(\rho^{n-1}(u^n-u^{n-1})f^{n+1}) \\
&\hspace{2cm}   + \nabla_v \cdot (\rho^{n-1}(u^{n-1}-v)(f^{n+1}-f^n)) - (\rho_{f^n}^\alpha-\rho_{f^{n-1}}^{\alpha})(\mm(f^n)-f^{n+1}) \\
&\hspace{2cm}  - \rho_{f^{n-1}}^\alpha (\mm(f^n) - \mm(f^{n-1}))  +\rho_{f^n}^\alpha (f^{n+1}-f^n)\Big)e^{2\la v \ra^{k - \epsilon}} (f^{n+1}-f^n)\,dxdv\\
&\quad =: \sum_{i=1}^7 \sfI_i.
\end{align*}
We first easily find that $\sfI_1 = 0$. For $\sfI_2$, we use the fact that
\[
\intr e^{a\la v \ra^{k - \epsilon} - b\la  v \ra^k}\,dv < \infty 
\]
for any $a,b > 0$  to yield
\begin{align*}
\sfI_2&= 3\inttr (\rho^n - \rho^{n-1})f^{n+1}(f^{n+1}-f^n)e^{2\la v \ra^{k - \epsilon}}\,dxdv\\
&\quad - \inttr (\rho^n -\rho^{n-1})(u^n-v)\cdot\nabla_v f^{n+1}(f^{n+1}-f^n)e^{2\la v \ra^{k - \epsilon}}\,dxdv\\
&\le C\lt\|\lt( \intr (f^{n+1})^2 e^{2\la v \ra^{k - \epsilon}}\,dv\rt)^{1/2} \rt\|_{H^1}\|h^n-h^{n-1}\|_{H^1}\|f^{n+1}-f^n\|_{L_{k - \epsilon}^2}\cr
&\quad + C \|u^n\|_{L^\infty} \lt\|\lt( \intr |\nabla_v f^{n+1}|^2 e^{2\la v \ra^{k - \epsilon}}\,dv\rt)^{1/2} \rt\|_{H^1}  \|h^n-h^{n-1}\|_{H^1}\|f^{n+1}-f^n\|_{L_{k - \epsilon}^2}\cr
&\quad + C  \lt\|\lt( \intr \la v \ra^2 |\nabla_v f^{n+1}|^2 e^{2\la v \ra^{k - \epsilon}}\,dv\rt)^{1/2} \rt\|_{H^1}\|h^n-h^{n-1}\|_{H^1}\|f^{n+1}-f^n\|_{L_{k - \epsilon}^2}\cr
&\leq C\|h^n-h^{n-1}\|_{H^1}^2 + \|f^{n+1}-f^n\|_{L_{k - \epsilon}^2}^2,
\end{align*}
where we used
\[
|\rho^n - \rho^{n-1}| = \lt| (1+h^n)^{\frac{2}{\gamma-1}}-(1+h^{n-1})^{\frac{2}{\gamma-1}}\rt| \le C|h^n - h^{n-1}|.
\]
For $\sfI_3$, similarly as before, we obtain
\begin{align*}
\sfI_3 &\le C\|\rho^{n-1}\|_{L^\infty}\lt\|\lt( \intr |\nabla_v f^{n+1}|^2 e^{2\la v \ra^{k - \epsilon}}\,dv\rt)^{1/2} \rt\|_{H^1}\|u^n - u^{n-1}\|_{H^1}\|f^{n+1}-f^n\|_{L_{k - \epsilon}^2}\cr
&\leq C\|u^n - u^{n-1}\|_{H^1}^2 + \|f^{n+1}-f^n\|_{L_{k - \epsilon}^2}^2
\end{align*}
For $\sfI_4$, we use integration by parts and Young's inequality to get
\begin{align*}
\sfI_4&= 3\inttr \rho^{n-1}(f^{n+1}-f^n)^2e^{2\la v \ra^{k - \epsilon}}\,dxdv\\
&\quad +  \frac12 \inttr (f^{n+1}-f^n)^2 \nabla\cdot( \rho^{n-1}(u^n - v) e^{2\la v \ra^{k - \epsilon}})\,dxdv\\
&\le C\|\rho^{n-1}\|_{L^\infty} \|f^{n+1}-f^n\|_{L_{k - \epsilon}^2}^2 -\frac{k}{2}\inttr \rho^{n-1}\la v\ra^{k-2}|v|^2e^{2\la v\ra^{k-\epsilon}}|f^{n+1}-f^n|^2\,dxdv\\
&\quad +\frac{k}{2}\inttr \rho^{n-1}\la v\ra^{k-2}|u^n|^2e^{2\la v\ra^{k-\epsilon}}|f^{n+1}-f^n|^2\,dxdv\\
&\le C\|f^{n+1}-f^n\|_{L_{k - \epsilon}^2}^2.
\end{align*}
For $\sfI_5$, similarly as the estimate of $\sfI_3$, we obtain
\begin{align*}
\sfI_5&\le C\|\rho_{f^n}-\rho_{f^{n-1}}\|_{L^2}\|f^{n+1}-f^n\|_{L_{k - \epsilon}^2}\lt\|\lt( \intr \lt(\mm(f^n)^2 +  (f^{n+1})^2\rt) e^{2\la v \ra^{k - \epsilon}}\,dv\rt)^{1/2} \rt\|_{L^\infty}\cr
&\le C\|f^{n+1}-f^n\|_{L_{k - \epsilon}^2}^2 + C\|f^n-f^{n-1}\|_{L_{k - \epsilon}^2}^2,
\end{align*}
where we used \eqref{weight_calc0} to get
\[
\lt\|\lt(\intr \mm(f^n)^2  e^{2\la v \ra^{k - \epsilon}}\,dv\rt)^{1/2} \rt\|_{L^\infty} \leq C\|\rho_{f^n}\|_{L^\infty} \leq C\|\rho_{f^n}\|_{H^2} \leq C.
\]
For $\sfI_6$, we use Lemma \ref{L2.3} to obtain
\begin{align*}
\sfI_6 &\le C\|\rho_{f^{n-1}}^\alpha\|_{L^\infty}\|\mm(f^n)-\mm(f^{n-1})\|_{L_{k - \epsilon}^2}\|f^{n+1}-f^n\|_{L_{k - \epsilon}^2} \cr
&\le C\|f^{n+1}-f^n\|_{L_{k - \epsilon}^2}^2 + C \|f^n-f^{n-1}\|_{L_{k - \epsilon}^2}^2.
\end{align*}
Clearly, $\sfI_7\le C\|f^{n+1}-f^n\|_{L_{k - \epsilon}^2}^2$. Thus, we gather all the estimates for $\sfI_i$'s and use Young's inequality to have the desired result.
\end{proof}

\begin{lemma}\label{h_cauchy}
Suppose that the initial data $(f_0,\rho_0,u_0)$ satisfy the conditions in Theorem \ref{T1.1}. Then we have
\begin{align*}
\frac{d}{dt}\|h^{n+1}-h^n\|_{H^1}^2 &\le C\lt(\|h^{n+1}-h^n\|_{H^1}^2 + \|h^n-h^{n-1}\|_{H^1}^2  +  \|\nabla (u^{n+1} - u^n)\|_{L^2}^2 + \|u^n -u^{n-1}\|_{H^1}^2\rt)\\
&\quad +\delta \|\nabla^2 (u^{n+1} - u^n)\|_{L^2}^2
\end{align*}
for $ t \in [0,T^*]$, where $\delta > 0$ will be determined later. 
\end{lemma}
\begin{proof}
We estimate
\begin{align}\label{hl2}
\begin{aligned}
\frac12\frac{d}{dt}\|h^{n+1} - h^n\|_{L^2}^2 &=  -\intt (h^{n+1}-h^n)(u^n \cdot \nabla h^{n+1} - u^{n-1}\cdot \nabla h^n)\,dx\\
&\quad - \frac{\gamma-1}{2}\intt (h^{n+1}-h^n)((1+h^n)\nabla \cdot u^{n+1} - (1+h^{n-1}\nabla\cdot u^n)\,dx\\
&= -\intt (h^{n+1}-h^n)(u^n\cdot\nabla (h^{n+1}-h^n) + (u^n-u^{n-1})\cdot\nabla h^n)\,dx\\
&\quad -\frac{\gamma-1}{2} \intt (h^{n+1}-h^n)((h^n-h^{n-1})\nabla\cdot u^n + (1+h^n)\nabla\cdot (u^{n+1}-u^n))\,dx\\
&\le \frac{\|\nabla u^n\|_{L^\infty}}{2}\|h^{n+1}-h^n\|_{L^2}^2 + \|\nabla h^n\|_{L^\infty}\|h^{n+1}-h^n\|_{L^2}\|u^n-u^{n-1}\|_{L^2}\\
&\quad + C \|\nabla u^{n+1}\|_{L^\infty} \|h^{n+1}-h^n\|_{L^2}\|h^n-h^{n-1}\|_{L^2} \\
&\quad +C(\|1+h^n\|_{L^\infty})\|h^{n+1}-h^n\|_{L^2}\|\nabla(u^{n+1} - u^n)\|_{L^2}\cr
&\le C\lt(\|h^{n+1}-h^n\|_{L^2}^2 + \|h^n-h^{n-1}\|_{L^2}^2  + \|u^n -u^{n-1}\|_{L^2}^2 + \|\nabla(u^{n+1} - u^n)\|_{L^2}^2\rt)
\end{aligned}
\end{align}
where we used Young's inequality.

We then estimate that for $i=1,2,3$
\begin{align*}
&\frac12\frac{d}{dt}\|\pa_i (h^{n+1} - h^n)\|_{L^2}^2 \cr
&\quad = -  \intt \pa_i (h^{n+1} - h^n) \lt( \pa_i (u^n - u^{n-1}) \cdot \nabla h^n + (u^n - u^{n-1}) \cdot \nabla \pa_i h^n \rt) dx\cr
&\qquad -  \intt \pa_i (h^{n+1} - h^n) \lt(   \pa_i u^n \cdot \nabla (h^{n+1} - h^n) + u^n \cdot \nabla \pa_i (h^{n+1} - h^n)\rt) dx\cr
&\qquad - \frac{\gamma-1}2 \intt \pa_i (h^{n+1} - h^n) \lt(   \pa_i (h^{n+1} - h^n) \nabla \cdot u^n + (h^n - h^{n-1}) \nabla \cdot \pa_i u^n\rt) dx\cr
&\qquad - \frac{\gamma-1}2 \intt \pa_i (h^{n+1} - h^n) \lt(   \pa_i h^n  \nabla \cdot (u^{n+1} - u^n) + (1 + h^n) \nabla \cdot \pa_i (u^{n+1} - u^n)\rt) dx\cr
&\quad \leq C\lt( \|\nabla h^n\|_{L^\infty} \|\pa_i (u^n - u^{n-1})\|_{L^2} + \|\nabla \pa_i h^n\|_{H^1} \|u^n - u^{n-1}\|_{H^1} \rt) \|\pa_i (h^{n+1} - h^n)\|_{L^2}\cr
&\qquad + C\lt( \|\pa_i u^n\|_{L^\infty} \|\nabla (h^{n+1} - h^n) \|_{L^2} + \|\nabla  u^n\|_{L^\infty} \|\pa_i (h^{n+1} - h^n)\|_{L^2} \rt) \|\pa_i (h^{n+1} - h^n)\|_{L^2}\cr
&\qquad + C\lt( \|\nabla  u^n\|_{L^\infty} \|\pa_i (h^{n+1} - h^n)\|_{L^2}  + \|\nabla \cdot \pa_i u^n\|_{H^1} \|h^n - h^{n-1}\|_{H^1} \rt) \|\pa_i (h^{n+1} - h^n)\|_{L^2}\cr
&\qquad + C\|\pa_i h^n \|_{L^\infty} \|\nabla (u^{n+1} - u^n) \|_{L^2} \|\pa_i (h^{n+1} - h^n)\|_{L^2} \cr
&\qquad + C\|\pa_i (h^{n+1} - h^n)\|_{L^2}^2 + \frac{\delta}{6} \|\nabla^2 (u^{n+1} - u^n)\|_{L^2}^2,
\end{align*}
where $\delta > 0$ will be determined later. 
%
Thus we have
\begin{align*}
\frac{d}{dt}\|\nabla (h^{n+1} - h^n)\|_{L^2}^2 &\leq C \|\nabla (h^{n+1} - h^n)\|_{L^2}^2 + C \|h^n - h^{n-1}\|_{H^1}^2\cr
&\quad  + C \|\nabla (u^{n+1} - u^n)\|_{L^2}^2 + C \|\nabla (u^u - u^{n-1})\|_{L^2}^2\cr
&\qquad + \delta \|\nabla^2 (u^{n+1} - u^n)\|_{L^2}^2.
\end{align*}
We then combine this and \eqref{hl2} to conclude the desired result.
\end{proof}

\begin{lemma}\label{u_cauchy}
Suppose that the initial data $(f_0,\rho_0,u_0)$ satisfy the conditions in Theorem \ref{T1.1}. Then we have
\begin{align*}
&\frac{d}{dt}\| u^{n+1}-u^n\|_{H^1}^2 + \lt(\frac{c_0\mu}{(1+M)^{\frac{2}{\gamma-1}}} - 8\delta\rt)\|\nabla (u^{n+1}-u^n)\|_{H^1}^2\cr
&\quad \leq  C\lt(\|u^n-u^{n-1}\|_{H^1}^2 + \|u^{n+1}-u^n\|_{H^1}^2 + \|h^n-h^{n-1}\|_{H^1}^2 + \|h^{n+1}-h^n\|_{H^1}^2\rt)\cr
&\qquad + C\|f^n-f^{n-1}\|_{L^2_{k - \epsilon}}^2.
\end{align*}
for $t \in [0,T^*]$, where $\delta > 0$, and $c_0$ is a positive constant satisfying
\[
\frac1{(1+h^n)^{\frac{2}{\gamma-1}}} \geq \frac{c_0}{(1 + M)^{\frac{2}{\gamma-1}}}  \ge \frac{16 \delta}{\mu}   \quad \forall \, n \in \N.
\]
\end{lemma}
\begin{proof} For smoothness of reading, we postpone this proof to Appendix \ref{app_u_cauchy}.
\end{proof}

\subsection{Proof of Theorem \ref{T1.1}} Now, we are ready to prove Theorem \ref{T1.1}.\\

\noindent $\bullet$ (Existence): We first gather the estimates in Lemmas \ref{f_cauchy}--\ref{u_cauchy} and choose $\delta > 0$ small enough to yield

\bq\label{cau_est}
\begin{aligned}
\frac{d}{dt}&\lt(\|f^{n+1}-f^n\|_{L_{k - \epsilon}^2}^2 +  \|h^{n+1}-h^n\|_{H^1}^2 + \|u^{n+1}-u^n\|_{H^1}^2 \rt) \\
&\hspace{0.5cm}+ \frac{c_0\mu}{2(1+M)^{\frac{2}{\gamma-1}}}\|\nabla(u^{n+1}-u^n)\|_{H^1}^2\\
&\le C\lt(\|f^{n+1}-f^n\|_{L_{k - \epsilon}^2}^2 + \|h^{n+1}-h^n\|_{H^1}^2 + \|u^{n+1}-u^n\|_{H^1}^2 \rt)\\
&\quad + C\lt(\|f^n-f^{n-1}\|_{L_{k - \epsilon}^2}^2 + \|h^n-h^{n-1}\|_{H^1}^2 + \|u^n-u^{n-1}\|_{H^1}^2 \rt).
\end{aligned}
\eq
Now we set
\begin{align*}
\me^{n+1}(t) &:= \|(f^{n+1}-f^n)(t)\|_{L_{k - \epsilon}^2}^2 +  \|(h^{n+1}-h^n)(t)\|_{H^1}^2 + \|(u^{n+1}-u^n)(t)\|_{H^1}^2,\\
\md^{n+1}(t)&:=  \frac{c_0\mu}{2(1+M)^{\frac{2}{\gamma-1}}}\|\nabla(u^{n+1}-u^n)(t)\|_{H^1}^2.
\end{align*}
Here we note that $\me^{n+1}(0)=0$ for any $n \in \N$. Then we can rewrite  \eqref{cau_est} as
\bq\label{cau_est2}
\frac{d}{dt}\me^{n+1}(t) + \md^{n+1}(t) \le C(\me^{n+1}(t) + \me^n(t)).
\eq
Then, we sum \eqref{cau_est2} over $n$ to get
\begin{align*}
\frac{d}{dt}\lt(\sum_{r=1}^n\me^{r+1}(t)\rt) + \sum_{r=1}^n\md^{r+1}(t) &\le C\lt(\sum_{r=1}^n\me^{r+1}(t) + C\me^1(t)\rt)\cr
&\le C\lt(1+\sum_{r=1}^n\me^{r+1}(t)\rt),
\end{align*}
where we used the uniform-in-$n$ upper bound. Thus, we integrate the above relation with respect to $t$ and use Gr\"onwall's lemma to obtain
\begin{align*}
\sum_{r=1}^n\me^{r+1}(t) + \int_0^t  \sum_{r=1}^n\md^{r+1}(t) \le e^{Ct}-1 
\end{align*}
for $ t \in [0,T^*]$. 
 Since $C>0$ is independent of $n$, the above estimate implies that the sequence $\{(f^n, h^n, u^n)\}_{n\ge 1}$ of triplets forms a Cauchy sequence in 
\[
\mc([0,T^*];L_{k - \epsilon}^2(\T^3\times\R^3)) \times \mc([0,T^*];H^1(\T^3)) \times \lt(\mc([0,T^*];H^1(\T^3))\cap L^2(0,T^*;H^2(\T^3))\rt), 
\]
and hence it converges to
\[
(f,h,u)\in \mc([0,T^*];L_{k - \epsilon}^2(\T^3\times\R^3)) \times \mc([0,T^*];H^1(\T^3)) \times \lt(\mc([0,T^*];H^1(\T^3))\cap L^2(0,T^*;H^2(\T^3))\rt).
\]
Moreover, our uniform-in-$n$ upper bound estimates imply
\begin{align*}
f^n &\stackrel{\ast}{\rightharpoonup}  f  \quad \mbox{in } \ L^\infty(0,T^*; H_k^2(\T^3\times \R^3)),\\
(h^n, u^n) &\stackrel{\ast}{\rightharpoonup}  (h,u)  \quad \mbox{in } \ L^\infty(0,T^*;H^3(\T^3)), \quad \mbox{and}\\
\nabla^4 u^n & \rightharpoonup \nabla^4 u   \quad \mbox{in } \ L^2( \T^3 \times (0,T^*)).
\end{align*}
Thus, we have found a limit $(f,n,u)$ satisfying
\begin{align*}
& f \in L^\infty(0,T^*;  H_k^2(\T^3 \times \R^3)), \cr
&  h \in L^\infty(0,T^*;H^3(\T^3)), \quad \mbox{and} \cr
&u \in  L^\infty(0,T^*;H^3(\T^3))\cap L^2(0,T^*;D^4(\T^3)).
\end{align*}
Necessary arguments for the time continuity can be found in \cite{CJ22, LPZ19}.\\

\noindent $\bullet$ (Uniqueness): Suppose that we have two regular solutions $(f_1, h_1, u_1)$ and $(f_2, h_2, u_2)$ to \eqref{main_sys} corresponding to the same initial data $(f_0,h_0,u_0)$. Then, we can deduce from the arguments in Lemmas \ref{f_cauchy}-\ref{u_cauchy} to get
\begin{align*}
\frac{d}{dt}&\lt(\|f_1-f_2\|_{L_{k - \epsilon}}^2 + \|h_1 - h_2\|_{H^1}^2 + \|u_1 - u_2\|_{H^1}^2 \rt) +\|\nabla(u_1 - u_2)\|_{H^1}^2 \\
&\le C\lt(\|f_1-f_2\|_{L_{k - \epsilon}}^2 + \|h_1 - h_2\|_{H^1}^2 + \|u_1 - u_2\|_{H^1}^2 \rt),
\end{align*}
and Gr\"onwall's lemma gives the desired result.

%
%
%
%
%
%
\section{Proof of Theorem \ref{T2.1}}\label{sec:5}

In this section, we provide the details of proof of Theorem \ref{T2.1}. Since the main idea of proof is similar to that of Theorem \ref{T1.1}, here we only give additional required estimates for $f$. We are now interested in the existence of $W^{1,\infty}_k$-solution $f$, thus we need to estimate $\|\mm(f)\|_{W_k^{1,\infty}}$ parallel to Lemma \ref{L2.2}.

\begin{lemma}\label{L2.22}
Suppose that $\|f\|_{W_k^{1,\infty}}<\infty$ for $k\in(1,2)$ and $\rho$, $u$ and $T$ satisfy
\[
\rho_f + |u_f| + T_f< c_1, \quad \rho_f>c_2, \quad \mbox{and} \quad T_f > c_3^{-1}.
\]
Then we have
\[
\|\mm(f)\|_{W_k^{1,\infty}}\le C(1+c_3)^2e^{Cc_1^{\frac{k}{2-k}}}\|f\|_{W_k^{1,\infty}}(1+\|f\|_{W_k^{1,\infty}}),
\]
where $C$ depends only on  $c_2$ and $k$.
\end{lemma}
\begin{proof} We first use \eqref{weight_calc0} to obtain
\bq\label{weight_calc}
e^{ \la v \ra^k} \mm (f) \le Ce^{C\lt(c_1^k + c_1^{\frac{k}{2-k}}\rt)} \frac{\rho_f}{(2\pi T_f)^{3/2}}e^{-\frac{|v-u_f|^2}{4T_f}} \le Ce^{Cc_1^{\frac{k}{2-k}}}\|f\|_{L_k^\infty},
\eq
where $C>0$ is a constant depending only on $c_2$ and $k$ and we used $k \in (1,2)$. For the derivatives, direct computation gives
\[
\pa_i \mm(f) =\Bigg(\pa_i \rho_f -  \frac32 \frac{\rho_f \pa_i T_f}{ T_f} -\rho_f\bigg(\frac{\pa_i u_f \cdot (u_f - v)}{T_f} - \frac{|v-u_f|^2}{2T_f^2} \pa_i T_f\bigg)\Bigg)\frac{e^{-\frac{|v-u_f|^2}{2T_f}}}{(2\pi T_f)^{3/2}},
\]
and
\begin{align*}
\pa_{v_i}\mm(f) = \frac{\rho_f (u_f-v)_i}{(2\pi)^{3/2} T_f^{5/2}}e^{-\frac{|v-u_f|^2}{2T_f}}
\end{align*}
for $i=1,2,3$.
Here, we note that
\[
|\pa_i \rho_f| \le \intr |\pa_i f|\,dv \le C\|\pa_i f\|_{L_k^\infty},
\]
\begin{align*}
|\pa_i u_f| &\le \frac{|\pa_i \rho_f|}{\rho_f} \frac{1}{\rho_f}\lt|\intr v f\,dv\rt| + \frac1{\rho_f} \intr |v| |\pa_i f|\,dv\\
&\le \frac{C\|\pa_i f\|_{L_k^\infty}|u_f|}{c_2} + \frac{C\|\pa_i f\|_{L_k^\infty}}{c_2} \le C(1+c_1)\|f\|_{W_k^{1,\infty}},
\end{align*}
and
\begin{align*}
|\pa_i T_f| &\le \frac{|\pa_i \rho_f|}{3\rho_f^2}\intr |v-u_f|^2 f\,dv + \frac1{\rho_f} \lt|\intr |u_f-v|^2 \pa_i f\,dv\rt|  + \frac1{\rho_f} \lt|\intr |u_f-v||\pa_i u_f|  f\,dv\rt|\\
&\le \frac{C\|\pa_i f\|_{L_k^\infty} T_f}{c_2}+C(1+|u_f|^2)\|\pa_if\|_{L_k^\infty}  + \sqrt{3 T_f} \|\pa_i u_f\|_{L^\infty}\\
&\le C(1+c_1^2)\|f\|_{W_k^{1,\infty}},
 \end{align*}
 where we used
 \[
  \frac1{\rho_f} \lt|\intr |u_f-v||\pa_i u_f|  f\,dv\rt| \leq \frac1{\rho_f} \lt(\intr |u_f-v|^2 f\,dv \rt)^{1/2}\lt(\intr |\pa_i u_f|^2 f\,dv \rt)^{1/2} \leq \sqrt{3 T_f} \|\pa_i u_f\|_{L^\infty}.
 \]
 Thus, together with \eqref{weight_calc} and the estimate from \cite[Lemma 2.4]{CLY21}, we can get the desired estimate.
\end{proof}

Analogously as in Section \ref{sec:3}, we consider the approximations \eqref{app_seq} and assume 
\[
\max_{1 \leq m \leq n} \overline{\mathfrak{X}}^m_{k}(T) < M \quad \mbox{and}  \quad \min_{1 \leq m \leq n} \inf_{(x,t)\in\T^3\times[0,T]}(1+h^m(x,t))>\frac{\delta}{2},
\]
where
\begin{align*}
\overline{\mathfrak{X}}^n_{k}(T)&:= \max\bigg\{\sup_{0\le t \le T} \|f^n(\cdot,\cdot,t)\|_{W_{k}^{1,\infty}}^2, \sup_{0\le t \le T} \|f^n(\cdot,\cdot,t)\|_{H_k^2}^2, \cr
&\hspace{4cm} \sup_{0\le t \le T}\lt(\frac{4\gamma}{(\gamma-1)^2}\|h^n(\cdot,t)\|_{H^3}^2 + \|u^n(\cdot,t)\|_{H^3}^2 \rt)\bigg\}.
\end{align*}
Then Lemmas \ref{flu_est}--\ref{f_est_l2} clearly hold. Since we used $L^2_{k-\epsilon}$ norm for the Cauchy estimates, for the proof of Theorem \ref{T2.1} it suffices to show the following uniform-in-$n$ estimate of $f^n$ in ${W_k^{1,\infty}}(\T^3 \times \R^3)$.

\begin{lemma}\label{f_est_linf}
Suppose that the initial data $(f_0,\rho_0,u_0)$ satisfy the conditions in Theorem \ref{T2.1}. If \eqref{seq_cond} holds, then we can find $0 < T_5 \le T_4$ depending only on $M$ and $N$ such that
\[
\sup_{0 \le t \le T_4} \|f^{n+1}(t)\|_{W_k^{1,\infty}}^2 <M. 
\]
\end{lemma}
\begin{proof}
For this, we will use the forward characteristics \eqref{for_cha}. We separately estimate the zeroth-order and first-order estimates as follows:\\

\noindent $\bullet$ (Step A: Zeroth-order estimates) First, we have
\begin{align*}
\frac12&\frac{d}{dt}\lt| e^{ \langle V^{n+1}(t)\rangle^k} f^{n+1}(Z^{n+1}(t),t)\rt|^2\\
&=e^{2\langle V^{n+1}\rangle^k}  f^{n+1}(Z^{n+1}(t),t)\bigg( k \langle V^{n+1}\rangle^{k-2} V^{n+1}\cdot \frac{d V^{n+1}}{dt}  f^{n+1}(Z^{n+1}(t),t)+ \frac{d}{dt}f^{n+1}(Z^{n+1}(t),t) \bigg) \\
&= e^{2\langle v\rangle^k}f^{n+1}(z,t) \bigg( k \langle v \rangle^{k-2} \rho^n(x,t) (u^n(x,t) - v) \cdot v f^{n+1}(z,t)\\
&\hspace{4cm} + 3\rho^n(x,t) f^{n+1}(z,t) + \rho_{f^n}^\alpha(x,t)(\mm(f^n)(z,t) -f^{n+1}(z,t))\bigg)\bigg|_{z = Z^{n+1}(t)}\\
&\le -\frac { k}{2}\rho^n \langle V^{n+1} \rangle^{k-2} |V^{n+1}|^2 e^{2\langle V^{n+1}\rangle^k}( f^{n+1}(Z^{n+1}(t),t))^2 \cr
&\quad + \frac{ k}{2} \langle V^{n+1}\rangle^{k-2} \rho^n |u^n|^2 e^{2\langle V^{n+1}\rangle^k} (f^{n+1}(Z^{n+1}(t),t))^2\\
&\quad +C(1+M) e^{2\langle V^{n+1}\rangle^k} (f^{n+1}(Z^{n+1}(t),t))^2 + e^{2\langle V^{n+1}\rangle^k} \lt[\rho_{f^n}^\alpha\mm(f^n) f^{n+1}\rt](Z^{n+1}(t),t)\\
&\le C(1+M)^{\frac{1}{\gamma-1}+1} \|f^{n+1}\|_{L_k^\infty}^2+ Ce^{C(1+M)^{\frac{2k}{2-k}}}(1+M)^{4+\alpha}\|f^{n+1}\|_{L_k^\infty}\\
&\le  Ce^{C(1+M)^{\frac{2k}{2-k}}}(1+M)^{\frac{1}{\gamma-1}+1+4\alpha}(\|f^{n+1}\|_{L_k^\infty}^2+1),
\end{align*}
where we used Young's inequality and Lemmas \ref{L2.2} and \ref{macro_bd}. We take the supremum over all possible characteristics to get
\bq\label{zero_f}
\begin{aligned}
\|f^{n+1}(t)\|_{L_k^\infty}^2 &\le \|f_0\|_{L_k^\infty}^2 + Ce^{C(1+M)^{\frac{2k}{2-k}}}(1+M)^{\frac{1}{\gamma-1}+4+\alpha} \int_0^t \|f^{n+1}(s)\|_{L_k^\infty}^2\,ds \\
&\quad + Ce^{C(1+M)^{\frac{2k}{2-k}}}(1+M)^{\frac{1}{\gamma-1}+4+\alpha} t, \qquad t \le T_4.
\end{aligned}
\eq

\noindent $\bullet$ (Step B: First-order estimates) We first investigate $\pa_i f^{n+1}$. One finds
\begin{align*}
&\pa_t \pa_i f^{n+1} + v \cdot \nabla (\pa_i f^{n+1}) + \rho^n (u^n-v)\cdot \nabla_v \pa_i f^{n+1} \cr
&= -\pa_i \rho^n (u^n-v) \cdot \nabla_v f^{n+1} + \rho^n \pa_i u^n \cdot \nabla_v f^{n+1} + 3\pa_i \rho^n f^{n+1} + 3\rho^n \pa_i f^{n+1} + \pa_i (\rho_{f_n}^\alpha(\mm (f^n) - f^{n+1})).
\end{align*}
Then we use the forward characteristics to get
\begin{align*}
\frac12& \frac{d}{dt}\lt| e^{ \langle V^{n+1}\rangle^k} \pa_i f^{n+1}(Z^{n+1}(t),t)\rt|^2\\
&=  e^{2\langle v\rangle^k} \pa_i f^{n+1} \bigg(k \langle v\rangle^{k-2} \rho^n (u^n - v)\cdot v  \pa_i f^{n+1} -\pa_i \rho^n (u^n - v) \cdot \nabla_v f^{n+1}  \\
&\hspace{3cm}+ \rho^n \pa_i u^n \cdot \nabla_v f^{n+1} + 3\pa_i \rho^n f^{n+1} + 3\rho^n \pa_i f^{n+1}\\
&\hspace{3cm}  + \rho_{f^n}^\alpha(\pa_i \mm(f^n) - \pa_i f^{n+1}) +\alpha \rho_{f^n}^{\alpha-1}\pa_i \rho_{f^n}(\mm(f^n) -  f^{n+1})\bigg)\Bigg|_{z=Z^{n+1}(t)}\\
&\le e^{2\langle V^{n+1}\rangle^k}\lt[\lt(\pa_i f^{n+1}\rt)^2\lt( -\frac { k}{2} \rho^n \langle V^{n+1}\rangle^{k-2}|V^{n+1}|^2 + \frac{k}{2} \langle V^{n+1}\rangle^{k-2} \rho^n |u^n|^2\rt)\rt](Z^{n+1}(t),t)\\
&\quad + C\|\pa_i f^{n+1}\|_{L_k}^\infty\Big((1+M)^{\frac{1}{\gamma-1}+1} \|\nabla_v f^{n+1}\|_{L_k^\infty} +  (1+M)^{\frac{1}{\gamma-1}+1}\|f^{n+1}\|_{L_k^\infty} \\
&\hspace{3.6cm}+ (1+M)^{\frac{1}{\gamma-1}+1}\|\pa_i f^{n+1}\|_{L^\infty} + (1+M)^{2\alpha+1}(\|\mm(f^n)\|_{W_k^{1,\infty}}+ \|f^{n+1}\|_{L_k^\infty})\Big)\\
&\quad + e^{2\langle V^{n+1}\rangle^k} V^{n+1} \cdot \lt(\pa_i\rho^n \nabla_v f^{n+1} \pa_i f^{n+1}\rt)(Z^{n+1}(t),t)\\
&\le -\frac{ k}{2} \rho^n \langle V^{n+1} \rangle^{k-2} |V^{n+1}|^2  e^{2\langle V^{n+1}\rangle^k} (\pa_i f^{n+1})^2(Z^{n+1}(t),t) + C(1+M)^{\frac{1}{\gamma-1}+1}\|f^{n+1}\|_{W_k^{1,\infty}}^2 \\
&\quad+ Ce^{C(1+M)^{\frac{2k}{2-k}}}(1+M)^{4+\alpha}\|\pa_i f^{n+1}\|_{L_k^\infty} + e^{2\langle V^{n+1}\rangle^k} V^{n+1} \cdot \lt(\pa_i\rho^n \nabla_v f^{n+1} \pa_i f^{n+1}\rt)(Z^{n+1}(t),t),
\end{align*}
where we used $|\pa_i \rho_{f^n}| \le C\|\pa_i f^n\|_{L_k^\infty}$, Young's inequality and Lemmas \ref{L2.2} and \ref{macro_bd}. Here, note that the last term on the right hand side of the above can be estimated as
\begin{align*}
\Big|& e^{2\langle V^{n+1}\rangle^k}  V^{n+1} \cdot \lt(\pa_i\rho^n \nabla_v f^{n+1} \pa_i f^{n+1}\rt)(Z^{n+1}(t),t)\Big|\\
&= \Big|e^{2\langle V^{n+1}\rangle^k}   V^{n+1} \cdot \lt(\pa_i\rho^n \nabla_v f^{n+1} \pa_i f^{n+1}\rt)(Z^{n+1}(t),t)\Big|\lt(\mathds{1}_{\{|V^{n+1}|\le 1 \}} +\mathds{1}_{\{|V^{n+1}|> 1 \}} \rt)\\
&\le C(1+M)^{\frac{1}{\gamma-1}}\| f^{n+1}\|_{\dot{W}_k^{1,\infty}}^2 + \Big|e^{2\langle V^{n+1}\rangle^k}   V^{n+1} \cdot \lt(\pa_i\rho^n \nabla_v f^{n+1} \pa_i f^{n+1}\rt)(Z^{n+1}(t),t)\Big|\mathds{1}_{\{|V^{n+1}|> 1 \}}\\
&=  C(1+M)^{\frac{1}{\gamma-1}}\| f^{n+1}\|_{\dot{W}_k^{1,\infty}}^2 + \Big|e^{2\langle V^{n+1}\rangle^k}   V^{n+1} \cdot \lt(\frac{\pa_i\rho^n}{\rho^n} \rho^n \nabla_v f^{n+1} \pa_i f^{n+1}\rt)(Z^{n+1}(t),t)\Big|\mathds{1}_{\{|V^{n+1}|> 1 \}}. 
\end{align*}
Then we use the similar arguments as \eqref{est_v_high12} and \eqref{est_v_high22}
to obtain
\bq\label{est_f_x}
\begin{aligned}
\frac{d}{dt}&\lt| e^{ \langle V^{n+1}\rangle^k} \pa_i f^{n+1}(Z^{n+1}(t),t)\rt|^2\\
&\le -\frac{ k}{2} \rho^n \langle V^{n+1} \rangle^{k-2} |V^{n+1}|^2  e^{2\langle V^{n+1}\rangle^k} (\pa_i f^{n+1})^2(Z^{n+1}(t),t)\\
&\quad + Ce^{C(1+M)^{\frac{2k}{2-k}}}(1+M)^{\frac{1}{\gamma-1}\cdot \frac{2k-1}{k-1}+4+2\alpha}\lt(\|f^{n+1}\|_{W_k^{1,\infty}}^2+1\rt) \\
&\quad+ \frac{(2-k)\rho^n \langle V^{n+1}\rangle^{k-2}|V^{n+1}|^2 e^{2\langle V^{n+1}\rangle^k}}{200}|\nabla_v f^{n+1}|^2(Z^{n+1}(t),t) \mathds{1}_{\{|V^{n+1}|\ge 1\}}.
\end{aligned}
\eq

\noindent Now, we deal with $\pa_{v_j} f^{n+1}$. From direct computation,
\begin{align*}
&\pa_t (\pa_{v_j}f^{n+1}) + v \cdot \nabla (\pa_{v_j}f^{n+1}) + \rho^n (u^n -v)\cdot \nabla_v (\pa_{v_j} f^{n+1}) \cr
&\quad = -\pa_j f^{n+1} + 4 \rho^n f^{n+1} + \rho_{f^n}^\alpha(\pa_{v_j}\mm(f^n) - \pa_{v_j}f^{n+1}).
\end{align*}
Then we get
\begin{align*}
\frac12&\frac{d}{dt}\lt| e^{  \langle V^{n+1}\rangle^k} \pa_{v_j}f^{n+1}(Z^{n+1}(t),t)\rt|^2\\
&= e^{2\langle V^{n+1}\rangle^k}\Bigg[\pa_{v_j}f^{n+1} \lt(k\langle V^{n+1} \rangle^{k-2} V^{n+1}\cdot\frac{dV^{n+1}}{dt} + \frac{d}{dt}\lt( \pa_{v_j}f^{n+1}\rt)\rt)\Bigg](Z^{n+1}(t),t)\\
&= e^{2\langle V^{n+1}\rangle^k}\Bigg[\pa_{v_j}f^{n+1}\bigg( k\langle V^{n+1}\rangle^{k-2}\rho^n (u^n - V^{n+1})\cdot V^{n+1} \pa_{v_j}f^{n+1} \\
&\hspace{3.5cm}- \pa_j f^{n+1} + 4\rho^n\pa_{v_j}f^{n+1}+ \rho_{f^n}^\alpha(\pa_{v_j}\mm (f^n) - \pa_{v_j}f^{n+1} )\bigg)\Bigg](Z^{n+1}(t),t)\\
&\le -\frac{ k}{2} \rho^n \langle V^{n+1} \rangle^{k-2}|V^{n+1}|^2e^{2\langle V^{n+1}\rangle^k} (\pa_{v_j}f^{n+1})^2(Z^{n+1}(t),t) + C(1+M)^{\frac{1}{\gamma-1}}\|f^{n+1}\|_{W_k^{1,\infty}}^2\\
&\quad + Ce^{C(1+M)^{\frac{2k}{2-k}}}(1+M)^{4+\alpha}\|\pa_{v_j}f^{n+1}\|_{L_k^\infty},
\end{align*}
and this gives
\bq\label{est_f_v}
\begin{aligned}
\frac{d}{dt}&\lt| e^{\langle V^{n+1}\rangle^k} \pa_{v_j}f^{n+1}(Z^{n+1}(t),t)\rt|^2\\
&\le -k \rho^n \langle V^{n+1} \rangle^{k-2}|V^{n+1}|^2e^{2\langle V^{n+1}\rangle^k} (\pa_{v_j}f^{n+1})^2(Z^{n+1}(t),t) \\
&\quad + Ce^{C(1+M)^{\frac{2k}{2-k}}}(1+M)^{\frac{1}{\gamma-1}+4+\alpha}(\|f^{n+1}\|_{W_k^{1,\infty}}^2+1).
\end{aligned}
\eq
Thus, we combine \eqref{est_f_x} with \eqref{est_f_v} to yield
\begin{align*}
&\frac{d}{dt}\lt(\lt| e^{ \langle V^{n+1}\rangle^k} \nabla f^{n+1}(Z^{n+1}(t),t)\rt|^2 + \lt| e^{ \langle V^{n+1}\rangle^k} \nabla_v f^{n+1}(Z^{n+1}(t),t)\rt|^2\rt)\cr
&\quad \le Ce^{C(1+M)^{\frac{2k}{2-k}}}(1+M)^{\frac{1}{\gamma-1}\cdot \frac{2k-1}{k-1}+4+\alpha}(\|f^{n+1}\|_{W_k^{1,\infty}}^2+1).
\end{align*}
We integrate the above relation with respect to $t$ and take the supremum over all possible characteristics to obtain
\bq\label{first_f}
\begin{aligned}
\|f^{n+1}(\cdot,\cdot,t)\|_{\dot{W}_k^{1,\infty}}^2 &\le \| f_0\|_{\dot{W}_k^{1,\infty}}^2 + Ce^{C(1+M)^{\frac{2k}{2-k}}}(1+M)^{\frac{1}{\gamma-1}\cdot \frac{2k-1}{k-1}+4+\alpha}\int_0^t \|f^{n+1}(\cdot,\cdot,s)\|_{W_k^{1,\infty}}^2\,ds\\
&\quad + Ce^{C(1+M)^{\frac{2k}{2-k}}}(1+M)^{\frac{1}{\gamma-1}\cdot \frac{2k-1}{k-1}+4+\alpha}t , \qquad t \le T_4. 
\end{aligned}
\eq
Finally, we gather the estimates \eqref{zero_f} and \eqref{first_f} and use Gr\"onwall's lemma to have
\begin{align*}
\|f^{n+1}(t)\|_{W_k^{1,\infty}}^2 &\le \|f_0\|_{W_k^{1,\infty}}^2e^{ Ce^{C(1+M)^{\frac{2k}{2-k}}}(1+M)^{\frac{1}{\gamma-1}\cdot \frac{2k-1}{k-1}+3+2\alpha} t} \cr
&\quad + Ce^{C(1+M)^{\frac{2k}{2-k}}}(1+M)^{\frac{2}{\gamma-1}\cdot \frac{2k-1}{k-1}+2+2\alpha}t e^{ Ce^{C(1+M)^{\frac{2k}{2-k}}}(1+M)^{\frac{1}{\gamma-1}\cdot \frac{2k-1}{k-1}+3+2\alpha} t},
\end{align*}
and hence, we can choose a sufficiently small $0<T_5\le T_4$ which gives the desired estimate.
\end{proof}

%
%
%
%
%
%
\section*{Acknowledgments}
Y.-P. Choi was supported by National Research Foundation of Korea(NRF) grant funded by the Korea government(MSIP) (No. 2022R1A2C1002820). J. Jung was supported by research funds for newly appointed professors of Jeonbuk National University in 2022.

\appendix

%
%
%
%
%
%
\section{Proof of Lemma \ref{L2.20}}\label{app.A0}
In this appendix, we provide the details of the proof of Lemma \ref{L2.20}.

For the zeroth-order estimate, we readily find
\[
\rho_f \leq \intr f\,dv \leq C\lt(\intr  e^{2\langle v\rangle^k} f^2\,dv\rt)^{1/2}, \quad |u_f| \leq \frac{1}{c_2} \intr |v|f\,dv \leq C\lt(\intr  e^{2\langle v\rangle^k} f^2\,dv\rt)^{1/2},
\]
and
\[
T_f \leq \frac{1}{3 c_2} \intr |v|^2f\,dv \leq C\lt(\intr  e^{2\langle v\rangle^k} f^2\,dv\rt)^{1/2}.
\]
Thus, 
\[
\|\rho_f\|_{L^2} + \|u_f\|_{L^2} + \|T_f\|_{L^2} \leq C\|f\|_{L^2_k}
\]
for some $C>0$ depends only on $c_2$ and $k$.

We next estimate that for $i=1,2,3$
\begin{align*}
|\pa_i \rho_f| &\le \intr |\pa_i f|\,dv \leq  C\lt(\intr  e^{2\langle v\rangle^k} |\pa_i f|^2\,dv\rt)^{1/2}, \cr
|\pa_i u_f| &\le \frac{|\pa_i \rho_f|}{\rho_f} \frac{1}{\rho_f}\lt|\intr v f\,dv\rt| + \frac1{\rho_f} \intr |v| |\pa_i f|\,dv \cr
&\leq C|\pa_i \rho_f|\lt(\intr  e^{2\langle v\rangle^k} f^2\,dv\rt)^{1/2} + C\lt(\intr  e^{2\langle v\rangle^k} |\pa_i f|^2\,dv\rt)^{1/2},
\end{align*}
and
 \begin{align*}
|\pa_i T_f| &\le \frac{|\pa_i \rho_f|}{3\rho_f^2}\intr |v-u_f|^2 f\,dv + \frac1{\rho_f} \lt|\intr |u_f-v|^2 \pa_i f\,dv\rt|  + \frac1{\rho_f} \lt|\intr |u_f-v||\pa_i u_f|  f\,dv\rt|\\
&\le C|\pa_i \rho_f| \intr |v|^2 f\,dv + C\|u_f\|_{L^\infty}^2|\pa_i \rho_f| + C\|u_f\|_{L^\infty}^2 \intr |\pa_i f|\,dv + C\intr |v|^2 |\pa_i f|\,dv \cr
&\quad + C\|u_f\|_{L^\infty}|\pa_i u_f| + |\pa_i u_f| \intr |v| f\,dv\cr
&\leq C\lt(|\pa_i \rho_f| +|\pa_i u_f| \rt)\lt(\intr  e^{2\langle v\rangle^k} f^2\,dv\rt)^{1/2} + C\lt(\|u_f\|_{H^2}^2 +1\rt)\lt(\intr  e^{2\langle v\rangle^k} |\pa_i f|^2\,dv\rt)^{1/2} \cr
&\quad + C\|u_f\|_{H^2}^2|\pa_i \rho_f| + C\|u_f\|_{H^2}|\pa_i u_f|.
 \end{align*}
This together with Lemma \ref{lem_uf} yields
\begin{align*}
\|\pa_i \rho_f\|_{L^2} &\leq C\|\pa_i f\|_{L^2_k} \leq C\|f\|_{H^1_k},\cr
\|\pa_i u_f\|_{L^2} &\leq C\|\pa_i \rho_f\|_{H^1} \lt\|\lt(\intr  e^{2\langle v\rangle^k} f^2\,dv\rt)^{1/2} \rt\|_{H^1}+ C\|\pa_i f\|_{L^2_k}
\leq C(\|\pa_i \rho_f\|_{H^1} + 1)\|f\|_{H^1_k},
\end{align*}
and
\begin{align*}
\|\pa_i T_f\|_{L^2} &\leq C\lt(\|\pa_i \rho_f\|_{H^1} + \|\pa_i u_f\|_{H^1}\rt) \lt\|\lt(\intr  e^{2\langle v\rangle^k} f^2\,dv\rt)^{1/2} \rt\|_{H^1} + C\lt(\|u_f\|_{H^2}^2 +1\rt)\|\pa_i f\|_{L^2_k}\cr
&\quad + C\|u_f\|_{H^2}^2\|\pa_i \rho_f\|_{L^2} + C\|u_f\|_{H^2}\|\pa_i u_f\|_{L^2}\cr
&\quad \leq C\lt(\|\pa_i \rho_f\|_{H^1}^2 +  \|u_f\|_{H^2}^2+1\rt)\|f\|_{H^1_k}.
\end{align*}
We finally obtain that for $i,j=1,2,3$
\[
\|\pa_{ij} \rho_f\|_{L^2} \le C\lt\| \lt(\intr e^{2\langle v\rangle^k} |\pa_{ij} f|^2\,dv\rt)^{1/2}\rt\|_{L^2} \le C\|\pa_{ij} f\|_{L_k^2},
\]
\begin{align*}
\lt\|\pa_{ij} u_f \rt\|_{L^2} &\le C \Bigg\| \bigg(\frac{|\pa_{ij} \rho_f|}{\rho_f^2} + \frac{|\pa_i \rho_f| |\pa_j \rho_f|}{\rho_f^3}\bigg)\intr |v| f\,dv + \frac{|\pa_i \rho_f|}{\rho_f^2}\intr |v| |\pa_j f|\,dv\\
&\hspace{3cm} +\frac{|\pa_j \rho_f|}{\rho_f^2}\intr |v| |\pa_i f|\,dv + \frac{1}{\rho_f}\intr |v| |\pa_{ij} f|\,dv \Bigg\|_{L^2}\\
&\leq C\lt(\|\pa_{ij} \rho_f\|_{L^2} + \|\pa_i \rho_f\|_{H^1}\|\pa_j \rho_f\|_{H^1}\rt)\lt\|\lt(\intr  e^{2\langle v\rangle^k} f^2\,dv\rt)^{1/2} \rt\|_{H^2}\cr
&\quad +C\|\pa_i \rho_f\|_{H^1}\lt\|\lt(\intr  e^{2\langle v\rangle^k} |\pa_j f|^2\,dv\rt)^{1/2} \rt\|_{H^1}  +C\|\pa_j \rho_f\|_{H^1}\lt\|\lt(\intr  e^{2\langle v\rangle^k} |\pa_i f|^2\,dv\rt)^{1/2} \rt\|_{H^1}\cr
&\quad + C\|\pa_{ij} f\|_{L^2_k}\cr
&\le C\|f\|_{H^2_k}\lt(1 + \|f\|_{H^2_k}^3\rt),
\end{align*}
where we used
\[
\lt\|\lt(\intr  e^{2\langle v\rangle^k} f^2\,dv\rt)^{1/2} \rt\|_{H^2} \leq C\|f\|_{H^2_k}\lt(1 + \|f\|_{H^2_k}\rt)
\]
due to Lemma \ref{lem_uf}.  We also deduce
\begin{align*}
\|\pa_{ij}T_f\|_{L^2} 
&  \le C \Bigg\| \bigg(\frac{|\pa_{ij}\rho_f|}{\rho_f^2} + \frac{|\pa_i \rho_f| |\pa_j \rho_f|}{\rho_f}\bigg) \intr (|v|^2+|u_f|^2) f\,dv \\
&\quad  + \frac{|\pa_i \rho_f|}{\rho_f}\intr (|\pa_j u_f| (|u_f| + |v|) f + (|u_f|^2+|v|^2) |\pa_j f|)\,dv\\
&\quad  +  \frac{|\pa_j \rho_f|}{\rho_f}\intr (|\pa_i u_f| (|u_f|+|v|) f + (|u_f|^2+|v|^2) |\pa_i f|)\,dv\\
&\quad   + \intr |\pa_{ij} u_f| |u_f -v|f\,dv + \intr |\pa_i u_f| |\pa_j u_f| f\,dv \cr
&\quad   + \intr \Big(|\pa_i u_f| |\pa_j f| + |\pa_j u_f| |\pa_i f|\Big) |u_f -v|\,dv + \intr (|u_f|^2 + |v|^2 )|\pa_{ij} f|\,dv\Bigg\|_{L^2}\\
& \le  C\lt(\|\pa_{ij} \rho_f\|_{L^2} + \|\pa_i \rho_f\|_{H^1}\|\pa_j \rho_f\|_{H^1}\rt)\lt(\|u_f\|_{H^2}^2+1\rt)\lt\|\lt(\intr  e^{2\langle v\rangle^k} f^2\,dv\rt)^{1/2} \rt\|_{H^2} \cr
&\quad + C\lt(\|\pa_i \rho_f\|_{H^1}\|\pa_j u_f\|_{H^1} + \|\pa_j \rho_f\|_{H^1}\|\pa_i u_f\|_{H^1} \rt) (\|u_f\|_{L^\infty} + 1) \lt\|\lt(\intr  e^{2\langle v\rangle^k} f^2\,dv\rt)^{1/2} \rt\|_{H^2}\cr
&\quad + C\lt(\|u_f\|_{L^\infty}^2 + 1\rt)  \|\pa_i \rho_f\|_{H^1} \lt\|\lt(\intr  e^{2\langle v\rangle^k} |\pa_j f|^2\,dv\rt)^{1/2} \rt\|_{H^1}  \cr
&\quad + C\lt(\|u_f\|_{L^\infty}^2 + 1\rt)     \|\pa_j \rho_f\|_{H^1} \lt\|\lt(\intr  e^{2\langle v\rangle^k} |\pa_i f|^2\,dv\rt)^{1/2} \rt\|_{H^1} \cr
&\quad + C\lt(\|\pa_{ij} u_f\|_{L^2}(\|u_f\|_{H^2}+1) + \|\pa_i u_f\|_{H^1}\|\pa_j u_f\|_{H^1}\rt)\lt\|\lt(\intr  e^{2\langle v\rangle^k} f^2\,dv\rt)^{1/2} \rt\|_{H^2} \cr
&\quad + C\lt(\|u_f\|_{H^2}+1\rt) \|\pa_j u_f\|_{H^1} \lt\|\lt(\intr  e^{2\langle v\rangle^k} |\pa_i f|^2\,dv\rt)^{1/2} \rt\|_{H^1}  \cr
&\quad + C\lt(\|u_f\|_{H^2}+1\rt)  \|\pa_i u_f\|_{H^1} \lt\|\lt(\intr  e^{2\langle v\rangle^k} |\pa_j f|^2\,dv\rt)^{1/2} \rt\|_{H^1}  \cr
&\quad + C\lt(\|u_f\|_{H^2}^2+1\rt)\|\pa_{ij}f\|_{L^2_k}\cr
&\leq C\|f\|_{H^2_k}\lt(1 + \|f\|_{H^2_k}^{11}\rt),
\end{align*}
where $C>0$ depends only on $c_2$ and $k$. Combining the above estimates, we conclude the desired results.

%
%
%
%
%
%
\section{Proof of Lemma \ref{flu_est}}\label{app.A}
In this appendix, we present the proof of Lemma \ref{flu_est}. We separately provide the zeroth and higher order estimates as follows:\\

\noindent $\bullet$ (Step A: Zeroth-order estimates) First, we have
\begin{align*}
\frac12\frac{d}{dt}\|h^{n+1}\|_{L^2}^2 &= -\intt (u^n \cdot \nabla h^{n+1}) h^{n+1}\,dx -\frac{\gamma-1}{2} \intt (1+h^n) h^{n+1} (\nabla\cdot u^{n+1})\,dx\\
&= \frac 12 \intt (\nabla\cdot u^n)|h^{n+1}|^2\,dx -\frac{\gamma-1}{2} \intt (1+h^n)h^{n+1} (\nabla\cdot u^{n+1}) \,dx\\
&\le C\|\nabla u^n\|_{L^\infty} \|h^{n+1}\|_{L^2}^2 -\frac{\gamma-1}{2} \intt (1+h^n)h^{n+1}(\nabla\cdot u^{n+1}) \,dx,
\end{align*}
where $C>0$ is a constant independent of $n$ and $T$. Using
\[
\frac1{(1+h^n)^{\frac{2}{\gamma-1}}} \geq \frac{c_0}{(1 + M)^{\frac{2}{\gamma-1}}} \quad \forall \, n \in \N
\]
for some $c_0 > 0$ independent of $n$, we also get
\begin{align*}
\frac12 \frac{d}{dt}\|u^{n+1}\|_{L^2}^2 &= -\intt( u^n \cdot \nabla u^{n+1})\cdot u^{n+1}\,dx - \frac{2\gamma}{\gamma-1} \intt (1+h^n) \nabla h^{n+1}\cdot u^{n+1}\,dx\\
&\quad + \intt \frac{\mu\Delta u^{n+1}}{(1+h^n)^{\frac{2}{\gamma-1}}}\cdot u^{n+1}\,dx -\inttr (u^n-v)\cdot u^{n+1} f^n \,dxdv\\
&= \frac12 \intt (\nabla \cdot u^n) |u^{n+1}|^2\,dx + \frac{2\gamma}{\gamma-1}\intt ( \nabla h^n \cdot u^{n+1}) h^{n+1}\,dx\\
&\quad +\frac{2\gamma}{\gamma-1} \intt (1+h^n) h^{n+1} (\nabla \cdot u^{n+1})\,dx + \frac{2\mu}{\gamma-1}\intt \frac{(\nabla h^n \cdot \nabla u^{n+1})}{(1+h^n)^{\frac{\gamma+1}{\gamma-1}}}\cdot u^{n+1}\,dx\\
&\quad - \mu\intt \frac{|\nabla u^{n+1}|^2}{(1+h^n)^{\frac{2}{\gamma-1}}}\,dx -\inttr (u^n-v)\cdot u^{n+1} f^n \,dxdv\\
&\le C\|\nabla u^n\|_{L^\infty}\|u^{n+1}\|_{L^2}^2 + C\|\nabla h^n\|_{L^\infty}\|u^{n+1}\|_{L^2}\|h^{n+1}\|_{L^2}\\
&\quad + \frac{2\gamma}{\gamma-1} \intt (1+h^n) h^{n+1} (\nabla \cdot u^{n+1})\,dx + C\delta^{-\frac{\gamma+1}{\gamma-1}} \mu \|\nabla h^n\|_{L^\infty}\|\nabla u^{n+1}\|_{L^2}\|u^{n+1}\|_{L^2}\\
&\quad - \frac{c_0\mu}{(1+M)^{\frac{2}{\gamma-1}}}\|\nabla u^{n+1}\|_{L^2}^2 + C(1+\|u^n\|_{L^\infty})\|u^{n+1}\|_{L^2}\|f^n\|_{L_k^2}\\
&\le C(1+M)\lt(\|u^{n+1}\|_{L^2}^2 + \|h^{n+1}\|_{L^2}^2\rt) + C(1+M)^{\frac{1}{\gamma-1}}  \|u^{n+1}\|_{L^2}^2\\
&\quad  + \frac{2\gamma}{\gamma-1} \intt (1+h^n) h^{n+1} (\nabla \cdot u^{n+1})\,dx\\
&\quad - \frac{c_0\mu}{2(1+M)^{\frac{2}{\gamma-1}}}\|\nabla u^{n+1}\|_{L^2}^2 + C(1+M)^2\|u^{n+1}\|_{L^2}\\
&\le C(1+M)^{\frac{2\gamma}{\gamma-1}}\lt(\frac{4\gamma}{(\gamma-1)^2} \|h^{n+1}\|_{L^2}^2 + \|u^{n+1}\|_{L^2}^2\rt) -\frac{c_0\mu}{2(1+M)^{\frac{2}{\gamma-1}}} \|\nabla u^{n+1}\|_{L^2}^2\\
&\quad + \frac{2\gamma}{\gamma-1} \intt (1+h^n) h^{n+1} (\nabla \cdot u^{n+1})\,dx + C(1+M)^2.
\end{align*}
Here $c_0>0$ and $C=C(\delta, \gamma)>0$ are constants independent of $n$ and $T$.

Then we combine the above two estimates to yield
\bq\label{zero_est_flu}
\begin{aligned}
\frac{d}{dt}&\lt( \frac{4\gamma}{(\gamma-1)^2} \|h^{n+1}\|_{L^2}^2 + \|u^{n+1}\|_{L^2}^2\rt) + \frac{c_0\mu}{(1+M)^{\frac{2}{\gamma-1}}}\|\nabla u^{n+1}\|_{L^2}^2\\
&\le C(1+M)^{\frac{2\gamma}{\gamma-1}} \lt(  \frac{4\gamma}{(\gamma-1)^2} \|h^{n+1}\|_{L^2}^2 + \|u^{n+1}\|_{L^2}^2\rt) + C(1+M)^2,\\
\end{aligned}
\eq
where $C=C(\delta,\gamma)>0$ is a constant independent of $n$ and $T$.\\

\noindent $\bullet$ (Step B: Higher-order estimates) For $1 \le \ell \le 3$, we first estimate
\begin{align*}
&\frac 12\frac{d}{dt}\|\pa^\ell h^{n+1}\|_{L^2}^2 \cr
&\quad = -\intt (u^n \cdot \nabla \pa^\ell h^{n+1})\pa^\ell h^{n+1}\,dx -\intt \lt[ \pa^\ell\lt( u^n \cdot \nabla h^{n+1}\rt) - u^n \cdot \nabla \pa^\ell h^{n+1} \rt] \pa^\ell h^{n+1}\,dx\\
&\qquad -\frac{\gamma-1}{2} \intt (1+h^n)(\nabla\cdot \pa^\ell u^{n+1})\pa^\ell h^{n+1}\,dx\\
&\qquad -\frac{\gamma-1}{2} \intt \lt[ \pa^\ell \lt((1+h^n)\nabla\cdot u^{n+1}\rt) - (1+h^n)\nabla\cdot \pa^\ell u^{n+1} \rt]\pa^\ell h^{n+1}\,dx\\
&\quad \le C\|\nabla u^n\|_{L^\infty} \|\pa^\ell h^{n+1}\|_{L^2}^2 + C\|\pa^\ell h^{n+1}\|_{L^2}\lt(\|\nabla u^n\|_{L^\infty}\|\nabla^\ell h^{n+1}\|_{L^2} + \|\nabla h^{n+1}\|_{L^\infty}\|\nabla^\ell u^n\|_{L^2}\rt)\\
&\qquad +C \|\pa^\ell h^{n+1}\|_{L^2}\lt(\|\nabla h^n\|_{L^\infty}\|\nabla^\ell u^{n+1}\|_{L^2} + \|\nabla\cdot u^{n+1}\|_{L^\infty} \|\nabla^\ell h^n\|_{L^2} \rt)\\
&\qquad -\frac{\gamma-1}{2} \intt (1+h^n)(\nabla\cdot \pa^\ell u^{n+1})\pa^\ell h^{n+1}\,dx\\
&\quad \le C(1+M)\lt(\frac{4\gamma}{(\gamma-1)^2} \|\nabla h^{n+1}\|_{H^2}^2 + \|\nabla u^{n+1}\|_{H^2}^2\rt) -\frac{\gamma-1}{2} \intt (1+h^n)(\nabla\cdot \pa^\ell u^{n+1})\pa^\ell h^{n+1}\,dx,
\end{align*}
where we used Lemma \ref{lem_moser} and $C=C(\gamma)>0$ is a constant independent of $n$ and $T$. One also obtains
\begin{align*}
&\frac12\frac{d}{dt} \| \pa^\ell u^{n+1}\|_{L^2}^2 \cr
&\quad = -\intt \lt(u^n \cdot \nabla \pa^\ell u^{n+1}\rt)\cdot \pa^\ell u^{n+1}\,dx - \intt \lt[ \pa^\ell (u^n \cdot \nabla u^{n+1}) - u^n \cdot \nabla \pa^\ell u^{n+1}\rt] \cdot \pa^\ell u^{n+1}\,dx\\
&\qquad-  \frac{2\gamma}{\gamma-1} \intt \pa^\ell \lt( (1+h^n) \nabla h^{n+1}\rt)\cdot \pa^\ell u^{n+1}\,dx +\mu \intt \pa^\ell \lt(\frac{\Delta u^{n+1}}{(1+h^n)^{\frac{2}{\gamma-1}}} \rt)\cdot \pa^\ell u^{n+1}\,dx\\
&\qquad -\inttr \pa^\ell \lt( (u^n-v) f^n\rt)\cdot \pa^\ell u^{n+1}\,dx\\
&\quad = \sum_{i=1}^5 \sfI_i.
\end{align*}

We separately estimate $\sfI_i$'s as follows:\\

\noindent $\diamond$ (Step B-1: Estimates for $\sfI_1$ and $\sfI_2$) We can estimate these two terms as
\begin{align*}
\sfI_1 &\le C\|\nabla u^n \|_{L^\infty}\|\pa^\ell u^{n+1}\|_{L^2}^2 \le C(1+M) \|\nabla u^{n+1}\|_{H^2}^2,\\
\sfI_2 &\le C\|\pa^\ell u^{n+1}\|_{L^2}\lt( \|\nabla u^n \|_{L^\infty} \|\nabla^\ell u^{n+1}\|_{L^2} + \|\nabla u^{n+1}\|_{L^\infty} \|\pa^\ell u^n\|_{L^2}\rt) \le C(1+M)\|\nabla u^{n+1}\|_{H^2}^2,
\end{align*}
where we used Lemma \ref{lem_moser}.\newline

\noindent $\diamond$ (Step B-2: Estimates for $\sfI_3$) One uses integration by parts and Lemma \ref{lem_moser} to obtain
\begin{align*}
\sfI_3 &= -\frac{2\gamma}{\gamma-1} \intt  (1+h^n)\nabla \pa^\ell h^{n+1}\cdot \pa^\ell u^{n+1}\,dx \\
&\quad -\frac{2\gamma}{\gamma-1}\intt \lt[  \pa^\ell\lt((1+h^n)\nabla h^{n+1}\rt) - (1+h^n) \nabla \pa^\ell h^{n+1}\rt] \cdot \pa^\ell u^{n+1}\,dx\\
&= \frac{2\gamma}{\gamma-1}\intt \pa^\ell h^{n+1} \nabla h^n \cdot \pa^\ell u^{n+1}\,dx\\
&\quad + \frac{2\gamma}{\gamma-1}\intt (1+h^n) \pa^\ell h^{n+1} \nabla \cdot u^{n+1}\,dx\\
&\quad -\frac{2\gamma}{\gamma-1}\intt \lt[  \pa^\ell\lt((1+h^n)\nabla h^{n+1}\rt) - (1+h^n) \nabla \pa^\ell h^{n+1}\rt] \cdot \pa^\ell u^{n+1}\,dx\\
&\le C\|\nabla h^n\|_{L^\infty}\|\pa^\ell h^{n+1}\|_{L^2}\|\pa^\ell u^{n+1}\|_{L^2}\\
&\quad + C\|\pa^\ell u^{n+1}\|_{L^2}\lt( \|\nabla h^n\|_{L^\infty}\|\nabla^\ell h^{n+1}\|_{L^2} + \|\pa^\ell h^n\|_{L^2}\|\nabla h^{n+1}\|_{L^2}\rt)\\
&\quad  + \frac{2\gamma}{\gamma-1}\intt (1+h^n) \pa^\ell h^{n+1} \nabla \cdot u^{n+1}\,dx\\
&\le C(1+M)\lt( \|\nabla h^{n+1}\|_{H^2}^2 + \|\nabla u^{n+1}\|_{H^2}^2\rt) + \frac{2\gamma}{\gamma-1}\intt (1+h^n) \pa^\ell h^{n+1} \nabla \cdot u^{n+1}\,dx.
\end{align*}
Here $C=C(\gamma,\ell)>0$ is a constant independent of $n$ and $T$.\\

\noindent $\diamond$ (Step B-3: Estimates for $\sfI_4$) We also estimate $\sfI_4$ term by term as follows:

\begin{align*}
\sfI_4 &= \mu\intt \frac{\pa^\ell \Delta u^{n+1}}{(1+h^n)^{\frac{2}{\gamma-1}}} \cdot \pa^\ell u^{n+1}\,dx + \mu\intt \lt[ \pa^\ell \lt( \frac{\Delta u^{n+1}}{(1+h^n)^{\frac{2}{\gamma-1}}}\rt) - \frac{\pa^\ell \Delta u^{n+1}}{(1+h^n)^{\frac{2}{\gamma-1}}}\rt] \cdot \pa^\ell u^{n+1}\,dx\\
&=: \sfI_{41} + \sfI_{42}.
\end{align*}
Before we estimate $\sfI_{4i}$'s, we first prove the following inequality holds: for any $j =1,2,3$ and $\mu \in \R$,
\bq\label{ineq_press}
\|\pa^j \lt( (1+h^n)^\mu\rt)\|_{L^2} \le C(1+M)^{\max\{j,\mu\}},
\eq
where $C=C(\mu)$ is a constant independent of $n$ and $T$. For $j=1$, we have
\begin{align*}
\|\pa((1+h^n)^\mu)\|_{L^2} &= |\mu| \|(1+h^n)^{\mu-1}\pa h^n\|_{L^2}\\
& \le C\|\pa h^n\|_{L^2} \lt( \delta^{\mu-1} + (1+M)^{\mu-1}\rt)\\
&\le C\lt((1+M) + (1+M)^\mu\rt) \le C(1+M)^{\max\{1,\mu\}}.
\end{align*}
Inductively, if \eqref{ineq_press} holds for $1 \le j <3$, then one uses Lemma \ref{lem_moser} to get
\begin{align*}
\|\pa^{j+1}( (1+h^n)^\mu)\|_{L^2} &= |\mu| \| \pa^j ( (1+h^n)^{\mu-1} \pa h^n)\|_{L^2}\\
&\le C\lt(\|\pa h^n\|_{L^\infty} \|\pa^j ((1+h^n)^{\mu-1})\|_{L^2} + \|(1+h^n)^{\mu-1}\|_{L^\infty} \|\pa^{j+1} h^n\|_{L^2}\rt)\\
&\le CM(1+M)^{\max\{j, \mu-1\}}+ CM\lt(\delta^{\mu-1} + (1+M)^{\mu-1}\rt)\\
&\le C(1+M)^{\max\{j+1, \mu\}},
\end{align*}
and this completes the proof for \eqref{ineq_press}. For $\sfI_{41}$,
\begin{align*}
\sfI_{41}&= \frac{2\mu}{\gamma-1}\intt \frac{(\nabla h^n \cdot \nabla \pa^\ell u^{n+1})}{(1+h^n)^{\frac{\gamma+1}{\gamma-1}}}\cdot \pa^\ell u^{n+1}\,dx -\mu \intt \frac{1}{(1+h^n)^{\frac{2}{\gamma-1}}} |\pa^\ell \nabla u^{n+1}|^2\,dx\\
&\le C\mu\|\nabla h^n\|_{L^\infty}\|\pa^\ell \nabla u^{n+1}\|_{L^2}\|\pa^\ell u^{n+1}\|_{L^2} - \frac{c_0\mu}{(1+M)^{\frac{2}{\gamma-1}}} \|\pa^\ell \nabla u^{n+1}\|_{L^2}^2\\
&\le C\mu (1+M)\|\pa^\ell \nabla u^{n+1}\|_{L^2}\|\pa^\ell u^{n+1}\|_{L^2} - \frac{c_0\mu}{(1+M)^{\frac{2}{\gamma-1}}} \|\pa^\ell \nabla u^{n+1}\|_{L^2}^2,
\end{align*}
where $C>0$ is constant independent of $n$ and $T$. For $\sfI_{42}$,
\begin{align*}
\sfI_{42} &= \mu\sum_{r=1}^\ell \binom{\ell}{r}\intt \pa^r \lt( \frac{1}{(1+h^n)^{\frac{2}{\gamma-1}}}\rt) \pa^{\ell-r}\Delta u^{n+1}\cdot \pa^\ell u^{n+1}\,dx\\
&\le C\mu\|\nabla h^n\|_{L^\infty} \|\pa^{\ell-1}\Delta u^{n+1}\|_{L^2}\|\pa^\ell u^{n+1}\|_{L^2}\\
&\quad + C\mu\lt\|\pa^2 \lt((1+h^n)^{-\frac{2}{\gamma-1}}\rt)\rt\|_{L^6}\|\pa^{\ell-2}\Delta u^{n+1}\|_{L^3}\|\pa^\ell u^{n+1}\|_{L^2}\\
&\quad + C\mu\lt\|\pa^3\lt( (1+h^n)^{-\frac{2}{\gamma-1}}\rt)\rt\|_{L^2}\|\pa^{\ell-3}\Delta u^{n+1}\|_{L^6}\|\pa^\ell u^{n+1}\|_{L^3}\\
&\le C\mu(1+M)^3 (\|u^{n+1}\|_{H^3} + \|\nabla^4 u^{n+1}\|_{L^2})\|u^{n+1}\|_{H^3}.
\end{align*}
Thus, we get
\[
\sfI_4 \le C\mu(1+M)^3 (\|u^{n+1}\|_{H^3} + \|\nabla^4 u^{n+1}\|_{L^2})\|u^{n+1}\|_{H^3} - \frac{c_0\mu}{(1+M)^{\frac{2}{\gamma-1}}} \|\pa^\ell \nabla u^{n+1}\|_{L^2}^2.
\]

\noindent $\diamond$ (Step B-4: Estimates for $\sfI_5$) In this case,
\begin{align*}
\sfI_5&= \inttr (u^n-v)\pa^{\ell-1} f^n \cdot \pa^{\ell+1} u^{n+1}\,dxdv\\
&\quad + \sum_{r=1}^{\ell-1}\inttr \binom{\ell-1}{r}\inttr (\pa^r u^n) \pa^{\ell-1-r} f^n\cdot \pa^{\ell+1} u^{n+1}\,dxdv\mathds{1}_{\{\ell \geq 2\}}\\
&\le C(1+\|u^n\|_{L^\infty}) \|f^n\|_{H_k^2} \|\pa^{\ell+1} u^{n+1}\|_{L^2}\\
&\quad +C\sum_{r=1}^{\ell-1} \|\pa^r u^n\|_{L^4} \|\pa^{\ell-1-r}\rho_{f^n}\|_{L^4}\|\pa^{\ell+1} u^{n+1}\|_{L^2}\mathds{1}_{\{\ell \geq 2\}}\\
&\le C(1+M)^2 \|\pa^{\ell+1} u^{n+1}\|_{L^2}.
\end{align*}
Here we used
\[
\sum_{r=1}^{\ell-1} \|\pa^r u^n\|_{L^4} \|\pa^{\ell-1-r}\rho_{f^n}\|_{L^4} \leq C\sum_{r=1}^{\ell-1}\|\pa^r u^n\|_{H^1} \|\pa^{\ell-1-r}\rho_{f^n}\|_{H^1} \leq C(1+M).
\]
Hence, we gather all the estimates for $\sfI_i$'s to yield
\begin{align*}
\frac12\frac{d}{dt}\|\pa^\ell u^{n+1} \|_{L^2}^2 &\le C(1+M)\lt(\|u^{n+1}\|_{H^3}^2 + \|h^{n+1}\|_{H^3}^2\rt)+ C\mu(1+M)^3 (\|u^{n+1}\|_{H^3} + \|\nabla^4 u^{n+1}\|_{L^2}) \|u^{n+1}\|_{H^3} \\
&\quad - \frac{c_0\mu}{(1+M)^{\frac{2}{\gamma-1}}}\|\pa^\ell \nabla u^{n+1}\|_{L^2}^2+ \frac{2\gamma}{\gamma-1}\intt (1+h^n) \pa^\ell h^{n+1} \nabla \cdot u^{n+1}\,dx \\
&\quad  + C(1+M)^2 \|\pa^{\ell+1}u^{n+1}\|_{L^2},
\end{align*}
where $C=C(\gamma,\ell)>0$ is a constant independent of $n$ and $T$. We combine this with the estimate for $\pa^\ell h^{n+1}$ to obtain
\begin{align*}
\frac{d}{dt}&\lt( \frac{4\gamma}{(\gamma-1)^2} \|\pa^\ell h^{n+1}\|_{L^2}^2 + \|\pa^\ell u^{n+1}\|_{L^2}^2\rt)\\
&\le C(1+M)^3\lt( \frac{4\gamma}{(\gamma-1)^2} \|h^{n+1}\|_{H^3}^2 + \|u^{n+1}\|_{H^3}^2\rt) +C\mu(1+M)^3 \|\nabla^4 u^{n+1}\|_{L^2}( \|u^{n+1}\|_{H^3}+1)\\
&\quad - \frac{2c_0\mu}{(1+M)^{\frac{2}{\gamma-1}}}\|\pa^\ell \nabla u^{n+1}\|_{L^2}^2  + C(1+M)^2 \|\pa^{\ell+1}u^{n+1}\|_{L^2}.
\end{align*}
Thus, we combine the above estimates for every $\ell=1,2,3$ and use Young's inequality to get
\bq\label{high_est_flu}
\begin{aligned}
\frac{d}{dt}&\lt(\frac{4\gamma}{(\gamma-1)^2}\|\nabla h^{n+1}\|_{H^2}^2 + \|\nabla u^{n+1}\|_{H^2}^2\rt) + \frac{c_0\mu}{(1+M)^{\frac{2}{\gamma-1}}}\|\nabla^2 u^{n+1}\|_{H^2}^2\\
&\le C(1+M)^{\frac{6\gamma-4}{\gamma-1}}\lt( \frac{4\gamma}{(\gamma-1)^2} \|h^{n+1}\|_{H^3}^2 + \|u^{n+1}\|_{H^3}^2\rt) + C(1+M)^{\frac{6\gamma-4}{\gamma-1}},
\end{aligned}
\eq
where we used Young's inequality and $C=C(\gamma)>0$ is a constant independent of $n$ and $T$. Finally, we combine \eqref{zero_est_flu} with \eqref{high_est_flu} to yield
\begin{align*}
\frac{d}{dt}&\lt( \frac{4\gamma}{(\gamma-1)^2} \|h^{n+1}\|_{H^3}^2 + \|u^{n+1}\|_{H^3}^2 \rt) + \frac{c_0\mu}{(1+M)^{\frac{2}{\gamma-1}}}\|\nabla u^{n+1}\|_{H^3}^2 \\
&\le C(1+M)^{\frac{6\gamma-4}{\gamma-1}}\lt( \frac{4\gamma}{(\gamma-1)^2} \|h^{n+1}\|_{H^3}^2 + \|u^{n+1}\|_{H^3}^2 \rt) + C(1+M)^{\frac{6\gamma-4}{\gamma-1}},
\end{align*}
from which, we use Gr\"onwall's lemma to find, for $0\le t \le T$,
\begin{align*}
&\lt( \frac{4\gamma}{(\gamma-1)^2} \|h^{n+1}\|_{H^3}^2 + \|u^{n+1}\|_{H^3}^2 \rt) + \frac{c_0\mu}{(1+M)^{\frac{2}{\gamma-1}}}\int_0^t e^{C(1+M)^{\frac{6\gamma-4}{\gamma-1}}(t-s)} \|\nabla u^{n+1}(s)\|_{H^3}^2\,ds\\
&\quad \le \lt( \frac{4\gamma}{(\gamma-1)^2} \|h_0\|_{H^3}^2 + \|u_0\|_{H^3}^2 \rt)e^{C(1+M)^{\frac{6\gamma-4}{\gamma-1}}t} + C(1+M)^{\frac{6\gamma-4}{\gamma-1}} te^{C(1+M)^{\frac{6\gamma-4}{\gamma-1}}t}.
\end{align*}
Thus, we can choose a sufficiently small $0<T_1 \le T$ such that
\[
\sup_{0\le t \le T_1} \lt( \frac{4\gamma}{(\gamma-1)^2} \|h^{n+1}(t)\|_{H^3}^2 + \|u^{n+1}(t)\|_{H^3}^2 \rt)  +  \frac{c_0\mu}{(1+M)^{\frac{2}{\gamma-1}}}\int_0^{T_1} \|\nabla u^{n+1}(s)\|_{H^3}^2\,ds < M,
\] 
and this completes the proof

%
%
%
%
%
%
\section{Proof of Lemma \ref{u_cauchy}}\label{app_u_cauchy}

In a straightforward manner, we get
\begin{align*}
\frac12\frac{d}{dt}\|u^{n+1}-u^n\|_{L^2}^2&= -\intt (u^{n+1}-u^n)\cdot\lt(u^n\cdot\nabla u^{n+1} - u^{n-1}\cdot\nabla u^n \rt)\,dx\\
&\quad -\frac{2\gamma}{\gamma-1} \intt  (u^{n+1}-u^n)\cdot((1+h^n)\nabla h^{n+1} - (1+h^{n-1})\nabla h^n))\,dx\\
&\quad +\mu \intt(u^{n+1}-u^n)\cdot\lt(\frac{\Delta u^{n+1}}{(1+h^n)^{\frac{2}{\gamma-1}}} - \frac{\Delta u^n}{(1+h^{n-1})^{\frac{2}{\gamma-1}}} \rt)\,dx\\
&\quad - \inttr (u^{n+1}-u^n)\lt( (u^n-v)f^n - (u^{n-1}-v)f^{n-1}\rt)\,dxdv\\
&=: \sum_{i=1}^4 \sfJ_i.
\end{align*}
For $\sfJ_1$, we use Young's inequality to have
\begin{align*}
\sfJ_1&= -\intt (u^{n+1}-u^n)\cdot\lt(u^n \cdot \nabla(u^{n+1}-u^n) + (u^n-u^{n-1})\cdot\nabla u^n\rt)\,dx\\
&\le C\|\nabla u^n\|_{L^\infty}\|u^{n+1}-u^n\|_{L^2}^2 + C\|\nabla u^n\|_{L^\infty}\|u^{n+1}-u^n\|_{L^2}\|u^n -u^{n-1}\|_{L^2}\\
&\le C(\|u^{n+1}-u^n\|_{L^2}^2 + \|u^n-u^{n-1}\|_{L^2}^2).
\end{align*}
For $\sfJ_2$, we use the uniform-in-$n$ bound for $h^n$ and Young's inequality to get
\begin{align*}
\sfJ_2&= -\frac{2\gamma}{\gamma-1} \intt (u^{n+1}-u^n) \cdot \lt((1+h^n)\nabla(h^{n+1}-h^n) + ((h^n-h^{n-1})\nabla h^n \rt)\,dx\\
&= \frac{2\gamma}{\gamma-1} \intt \lt(\nabla\cdot(u^{n+1}-u^n) \rt) (1+h^n) (h^{n+1}-h^n) \,dx \\
&\quad + \frac{2\gamma}{\gamma-1}\intt (u^{n+1}-u^n)\cdot( (h^{n+1}-h^n)\nabla h^n)\,dx\\
&\quad -\frac{2\gamma}{\gamma-1} \intt (u^{n+1}-u^n)\cdot((h^n-h^{n-1})\nabla h^n)\,dx\\
&\le C \|u^{n+1}-u^n\|_{L^2}\|h^{n+1}-h^n\|_{L^2}\|\nabla h^n\|_{L^\infty}\\
&\quad + C \|u^{n+1}-u^n\|_{L^2}\|h^n-h^{n-1}\|_{L^2}\|\nabla h^n\|_{L^\infty}\\
&\quad +C(\|1+h^n\|_{L^\infty})\|h^{n+1}-h^n\|_{L^2}\|\nabla(u^{n+1} - u^n)\|_{L^2} \\
&\le C\lt(\|h^{n+1}-h^n\|_{L^2}^2 + \|u^{n+1}-u^n\|_{H^1}^2 + \|h^n-h^{n-1}\|_{L^2}^2\rt).
\end{align*}
For $\sfJ_3$, we use the upper bound of $h^n$, Sobolev embedding and Young's inequality to get
\begin{align*}
\sfJ_3&= \mu\intt (u^{n+1}-u^n)\cdot\lt( \frac{\Delta (u^{n+1}-u^n)}{(1+h^n)^{\frac{2}{\gamma-1}}} - \Delta u^n \lt(\frac{1}{(1+h^n)^{\frac{2}{\gamma-1}}} - \frac{1}{(1+h^{n-1})^{\frac{2}{\gamma-1}}} \rt)\rt)\,dx\\
&= \frac{2\mu}{\gamma-1}\intt (u^{n+1}-u^n)\frac{\nabla h^n \cdot \nabla(u^{n+1}-u^n)}{(1+h^n)^{\frac{2}{\gamma-1} +1}}\,dx - \mu\intt \frac{|\nabla (u^{n+1}-u^n)|^2}{(1+h^n)^{\frac{2}{\gamma-1}}}\,dx\\
&\quad - \mu \intt (u^{n+1}-u^n) \cdot \Delta u^n \lt(\frac{1}{(1+h^n)^{\frac{2}{\gamma-1}}} - \frac{1}{(1+h^{n-1})^{\frac{2}{\gamma-1}}} \rt) \,dx\\
&\le C\mu \|\nabla h^n\|_{L^\infty}\|u^{n+1}-u^n\|_{L^2}\|\nabla(u^{n+1}-u^n)\|_{L^2} - \frac{c_0\mu}{(1+M)^{\frac{2}{\gamma-1} }}\|\nabla(u^{n+1}-u^n)\|_{L^2}^2\\
&\quad + C\mu\|u^{n+1}-u^n\|_{L^6}\|h^n-h^{n-1}\|_{L^2}\|\Delta u^n\|_{L^3}\\
&\le C\mu\|u^{n+1}-u^n\|_{L^2}\|\nabla(u^{n+1}-u^n)\|_{L^2}- \frac{c_0\mu}{(1+M)^{\frac{2}{\gamma-1}}}\|\nabla(u^{n+1}-u^n)\|_{L^2}^2\\
&\quad + C\mu\|(u^{n+1}-u^n)\|_{H^1}\|h^n-h^{n-1}\|_{L^2}\\
&\le C\mu\lt(\|u^{n+1}-u^n\|_{L^2}^2 + \|h^n-h^{n-1}\|_{L^2}^2\rt) - \frac{c_0\mu}{2(1+M)^{\frac{2}{\gamma-1}}}\|\nabla(u^{n+1}-u^n)\|_{L^2}^2.
\end{align*}
For $\sfJ_4$, we use Young's inequality to obtain
\begin{align*}
\sfJ_4&= -\inttr (u^{n+1}-u^n)\cdot( (u^n-u^{n-1})f^n + (u^{n-1}-v)(f^n-f^{n-1}))\,dxdv\\
&\le C\|\rho_{f^n}\|_{L^\infty}\|u^{n+1}-u^n\|_{L^2}\|u^n-u^{n-1}\|_{L^2} + C(1+\|u^{n-1}\|_{L^\infty})\|u^{n+1}-u^n\|_{L^2}\|f^n-f^{n-1}\|_{L_{k - \epsilon}^2}\\
&\le C\lt(\|u^{n+1}-u^n\|_{L^2}^2 + \|u^n -u^{n-1}\|_{L^2}^2 + \|f^n -f^{n-1}\|_{L_{k - \epsilon}^2}^2\rt).
\end{align*}
Thus, we gather all the estimates for $\sfJ_i$'s to deduce
\begin{align}\label{ul2}
\begin{aligned}
\frac{d}{dt}&\|u^{n+1}-u^n\|_{L^2}^2 + \frac{c_0\mu}{(1+M)^{\frac{2}{\gamma-1}}}\|\nabla (u^{n+1}-u^n)\|_{L^2}^2\\
& \le C\lt( \|u^{n+1}-u^n\|_{H^1}^2 + \|h^{n+1}-h^n\|_{L^2}^2 + \|u^n-u^{n-1}\|_{L^2}^2 + \|h^n -h^{n-1}\|_{L^2}^2 + \|f^n -f^{n-1}\|_{L_{k - \epsilon}^2}^2\rt).
\end{aligned}
\end{align}

We next obtain  for $i=1,2,3$
\begin{align*}
&\frac12\frac{d}{dt}\|\pa_i (u^{n+1}-u^n)\|_{L^2}^2\cr
&\quad = -\intt \pa_i (u^{n+1}-u^n)\cdot\lt( \pa_i u^n\cdot\nabla (u^{n+1}-u^n) + u^n\cdot\nabla \pa_i (u^{n+1}-u^n) \rt)\,dx\\
&\qquad  -\intt \pa_i (u^{n+1}-u^n)\cdot\lt( \pa_i (u^n - u^{n-1}) \cdot\nabla u^n + (u^n - u^{n-1})\cdot\nabla \pa_i u^n \rt)\,dx\\
&\qquad  - \frac{2\gamma}{\gamma-1}\intt \pa_i (u^{n+1}-u^n)\cdot\lt( \pa_i h^n \nabla (h^{n+1} - h^n) + (1+h^n)\nabla \pa_i (h^{n+1} - h^n)  \rt)\,dx\\
&\qquad  - \frac{2\gamma}{\gamma-1}\intt \pa_i (u^{n+1}-u^n)\cdot\lt( \pa_i (h^n - h^{n-1}) \nabla h^n + (h^n - h^{n-1}) \nabla \pa_i h^n  \rt)\,dx\\
&\qquad  + \mu \intt \pa_i (u^{n+1}-u^n)\cdot\lt( \frac{\Delta \pa_i(u^{n+1}-u^n) }{(1+h^n)^{\frac{2}{\gamma-1}}} + \Delta(u^{n+1}-u^n) \pa_i \lt(\frac1{(1+h^n)^{\frac{2}{\gamma-1}}} \rt) \rt)\,dx\\
&\qquad  + \mu \intt \pa_i (u^{n+1}-u^n)\cdot\lt( \Delta \pa_i u^n \lt(\frac1{(1+h^n)^{\frac{2}{\gamma-1}}}  - \frac1{(1+h^{n-1})^{\frac{2}{\gamma-1}}}  \rt) \rt)\,dx\\
&\qquad  + \mu \intt \pa_i (u^{n+1}-u^n)\cdot\lt( \Delta  u^n \pa_i \lt(\frac1{(1+h^n)^{\frac{2}{\gamma-1}}}  - \frac1{(1+h^{n-1})^{\frac{2}{\gamma-1}}}  \rt) \rt)\,dx\\
&\qquad  -   \intt \pa_i (u^{n+1}-u^n)\cdot\lt( \pa_i (u^n - u^{n-1}) \rho_{f^n} - (u^n - u^{n-1}) \pa_i \rho_{f^n} \rt)\,dx\\
&\qquad  -   \intt \pa_i (u^{n+1}-u^n)\cdot\lt( \pa_i u^{n-1} (\rho_{f^n} - \rho_{f^{n-1}}) - \intr (u^{n-1} -v) \pa_i (f^n - f^{n-1})\,dv \rt)\,dx\\
&=: \sum_{i=1}^{16} \sfK_i.
\end{align*}
We then estimate $\sfK_i$ as follows.
\begin{align*}
\sfK_1 &\leq \|\pa_i u^n\|_{L^\infty} \|\nabla (u^{n+1}-u^n)\|_{L^2}\|\pa_i (u^{n+1}-u^n)\|_{L^2}\cr
& \leq C\|u^{n+1}-u^n\|_{H^1}^2,\cr
\sfK_2 &\leq \|\nabla u^n\|_{L^\infty}  \|\pa_i (u^{n+1}-u^n)\|_{L^2}^2\cr
&\leq C\|u^{n+1}-u^n\|_{H^1}^2,\cr
\sfK_3 &\leq \|\nabla u^n\|_{L^\infty} \|\pa_i (u^n-u^{n-1})\|_{L^2}\|\pa_i (u^{n+1}-u^n)\|_{L^2} \cr
&\leq C\|u^n-u^{n-1}\|_{H^1}^2 + C\|u^{n+1}-u^n\|_{H^1}^2,\cr
\sfK_4 &\ls \|\nabla \pa_i u^n\|_{H^1} \|u^n-u^{n-1}\|_{H^1}\|\pa_i (u^{n+1}-u^n)\|_{L^2}\cr
& \leq C\|u^n-u^{n-1}\|_{H^1}^2 + C\|u^{n+1}-u^n\|_{H^1}^2,\cr
\sfK_5 &\ls \|\pa_i h^n\|_{L^\infty} \| \nabla (h^{n+1}-h^n)\|_{L^2}\|\pa_i (u^{n+1}-u^n)\|_{L^2} \cr
&\leq C\|h^{n+1}-h^n\|_{H^1}^2 + C\|u^{n+1}-u^n\|_{H^1}^2,\cr
\sfK_7 &\ls \|\nabla h^n\|_{L^\infty} \| \pa_i (h^n-h^{n-1})\|_{L^2}\|\pa_i (u^{n+1}-u^n)\|_{L^2} \cr
&\leq C\|h^n-h^{n-1}\|_{H^1}^2 + C\|u^{n+1}-u^n\|_{H^1}^2,\cr
\sfK_8 &\ls \|\nabla \pa_i h^n\|_{H^1} \|h^n-h^{n-1}\|_{H^1}\|\pa_i (u^{n+1}-u^n)\|_{L^2} \cr
&\leq C\|h^n-h^{n-1}\|_{H^1}^2 + C\|u^{n+1}-u^n\|_{H^1}^2,\cr
\sfK_{10} &\ls \|\nabla h^n\|_{L^\infty} \|\nabla^2 (u^{n+1}-u^n)\|_{L^2}\|\pa_i (u^{n+1}-u^n)\|_{L^2}\cr
& \leq \frac{\delta}{3}\|\nabla^2 (u^{n+1}-u^n)\|_{L^2}^2 + C\|u^{n+1}-u^n\|_{H^1}^2,\cr
\sfK_{11} &\ls \|\Delta \pa_i u^n\|_{L^2} \|h^n-h^{n-1}\|_{H^1}\|\pa_i (u^{n+1}-u^n)\|_{H^1} \cr
&\leq C\lt(\|h^n-h^{n-1}\|_{H^1}^2 + \|u^{n+1}-u^n\|_{H^1}^2\rt) + \frac{\delta}{3} \|\nabla^2 (u^{n+1}-u^n)\|_{L^2}^2,\cr
\sfK_{12} &\ls \|\Delta u^n\|_{H^1} \|h^n-h^{n-1}\|_{H^1}\|\pa_i (u^{n+1}-u^n)\|_{H^1} \cr
&\leq  C\lt(\|h^n-h^{n-1}\|_{H^1}^2 + \|u^{n+1}-u^n\|_{H^1}^2\rt) + \frac{\delta}{3}\|\nabla^2(u^{n+1}-u^n)\|_{L^2}^2,\cr
\sfK_{13} &\ls \|\rho_{f^n}\|_{L^\infty} \|\pa_i (u^n-u^{n-1})\|_{L^2}\|\pa_i (u^{n+1}-u^n)\|_{L^2} \cr
&\leq C\|u^n-u^{n-1}\|_{H^1}^2 + C\|u^{n+1}-u^n\|_{H^1}^2,\cr
\sfK_{14} &\ls \|\pa_i \rho_{f^n}\|_{H^1} \|u^n-u^{n-1}\|_{H^1}\|\pa_i (u^{n+1}-u^n)\|_{L^2} \cr
&\leq C\|u^n-u^{n-1}\|_{H^1}^2 + C\|u^{n+1}-u^n\|_{H^1}^2,\cr
\sfK_{15} &\ls \|\pa_i u^{n-1} \|_{L^\infty} \|\rho_{f^n} - \rho_{f^{n-1}}\|_{L^2}\|\pa_i (u^{n+1}-u^n)\|_{L^2} \cr
&\leq C\|f^n-f^{n-1}\|_{L^2_{k - \epsilon}}^2 + C\|u^{n+1}-u^n\|_{H^1}^2.
\end{align*}
For the other terms, we obtain
\begin{align*}
\sfK_6 &= \frac{2\gamma}{\gamma-1}\intt \lt( \nabla h^n \pa_i(u^{n+1}-u^n) + (1 + h^n) \nabla \cdot \pa_i (u^{n+1}-u^n) \rt) \pa_i (h^{n+1}-h^n)\,dx\cr
&\leq C\lt(\|\nabla h^n\|_{L^\infty}\|\pa_i(u^{n+1}-u^n) \|_{L^2} + \|1+h^n\|_{L^\infty}\|\nabla \pa_i(u^{n+1}-u^n) \|_{L^2} \rt)\|\pa_i (h^{n+1}-h^n)\|_{L^2}\cr
&\leq C\|u^{n+1}-u^n\|_{H^1}^2 + C\|h^{n+1}-h^n\|_{H^1}^2 + \frac{\delta}{3} \|\nabla \pa_i(u^{n+1}-u^n) \|_{L^2}^2,
\end{align*}
\begin{align*}
\sfK_9 &= -\mu \intt \frac{|\nabla \pa_i (u^{n+1}-u^n)|^2}{(1 + h^n)^\frac{2}{\gamma-1}}\,dx + \frac{2\mu}{\gamma-1}\intt \nabla \pa_i (u^{n+1}-u^n) \pa_i(u^{n+1}-u^n) \cdot \frac{\nabla h^n}{(1 + h^n)^{\frac{2}{\gamma-1}+1}}\,dx\cr
&\leq - \frac{c_0\mu}{(1+M)^{\frac{2}{\gamma-1}}}\|\nabla \pa_i (u^{n+1}-u^n)\|_{L^2}^2 + C\|\nabla h^n\|_{L^\infty}\|\nabla \pa_i (u^{n+1}-u^n)\|_{L^2}\| \pa_i (u^{n+1}-u^n)\|_{L^2}\cr
&\leq  - \frac{c_0\mu}{2(1+M)^{\frac{2}{\gamma-1}}}\|\nabla \pa_i (u^{n+1}-u^n)\|_{L^2}^2 + C\| \pa_i (u^{n+1}-u^n)\|_{L^2}^2,
\end{align*}
and
\begin{align*}
\sfK_{16} &= \inttr \lt( \pa_i^2 (u^{n+1}-u^n) (u^{n-1} - v) + \pa_i (u^{n+1}-u^n) \pa_i u^{n-1} \rt) (f^{n} - f^{n-1})\,dxdv\cr
&\leq C\lt(\|u^{n-1}\|_{L^\infty}\|\pa_i^2 (u^{n+1}-u^n)\|_{L^2} + \|\pa_i u^{n-1}\|_{L^\infty}\|\pa_i (u^{n+1}-u^n)\|_{L^2}\rt)\|f^{n} - f^{n-1}\|_{L^2_{k - \epsilon}} \cr
&\leq \frac{\delta}{3} \|\pa_i^2 (u^{n+1}-u^n)\|_{L^2}^2 + C\|f^n-f^{n-1}\|_{L^2_{k - \epsilon}}^2 + C\|u^{n+1}-u^n\|_{H^1}^2.
\end{align*}
Combining all of the above estimates yields
\begin{align*}
&\frac{d}{dt}\|\nabla (u^{n+1}-u^n)\|_{L^2}^2 + \lt(\frac{c_0\mu}{(1+M)^{\frac{2}{\gamma-1}}} - 8\delta\rt)\|\nabla^2 (u^{n+1}-u^n)\|_{L^2}^2\cr
&\quad \leq C\lt(\|u^n-u^{n-1}\|_{H^1}^2 + \|u^{n+1}-u^n\|_{H^1}^2 + \|h^n-h^{n-1}\|_{H^1}^2 + \|h^{n+1}-h^n\|_{H^1}^2\rt)\cr
&\qquad + C\|f^n-f^{n-1}\|_{L^2_{k - \epsilon}}^2.
\end{align*}
This together with \eqref{ul2} asserts the desired result.

%
%
%
%
%
%

\end{document}